\documentclass[12pt,reqno]{article}
\usepackage[letterpaper,margin=1.6cm]{geometry}
\usepackage{latexsym, amsfonts, amsthm,amssymb,amscd,amsmath,makeidx}
\usepackage{enumerate}
\usepackage[normalem]{ulem}
\usepackage{exscale}
\usepackage{overpic} 
\usepackage{color} 
\usepackage{graphicx}
\usepackage{overpic}  
\usepackage{bm}
\usepackage{comment}
\usepackage{caption}
\usepackage{float}
\usepackage{array} 
\usepackage{booktabs} 
\usepackage{tabularx}
\usepackage{mathrsfs}
\usepackage{amssymb}
\usepackage{mathtools}

\usepackage{tikz}
\usepackage{blindtext}
\usepackage{bbm}
\usepackage{fancyhdr} 
\usepackage{mathrsfs}
\usepackage{wrapfig}
\usepackage{bookmark}
\usepackage{cleveref}
\usepackage{cases}
\usepackage{multirow}
\usepackage[title]{appendix}
\usepackage{hyperref}

\hypersetup{
	colorlinks,
	citecolor=red,
	filecolor=black,
	linkcolor=blue,
	urlcolor=black
}

\definecolor{DarkGreen}{rgb}{0.2,0.6,0.2}

\def\eps{\varepsilon}

\def\Ind#1{{\mathbbmss 1}_{_{\scriptstyle #1}}}

\newcommand{\norm}[1]{\lVert {#1}\rVert}
\newcommand{\bnorm}[1]{\big\lVert {#1}\big\rVert}
\newcommand{\Bnorm}[1]{\Big\lVert {#1}\Big\rVert}
\newcommand{\bbnorm}[1]{\bigg\lVert {#1}\bigg\rVert}

\newcolumntype{Y}{>{\centering\arraybackslash}X}

\def\ua{\uparrow}
\def\da{\downarrow}

\def\wh{\widehat}
\def\wt{\widetilde}

\newcommand\scalemath[2]{\scalebox{#1}{\mbox{\ensuremath{\displaystyle #2}}}}

\newcolumntype{C}{>{\centering\arraybackslash}X}

\def\ignore#1{}

\numberwithin{equation}{section}

\usetikzlibrary{positioning,calc,cd}

\allowdisplaybreaks[4]

\def\cA{{\mathscr A}}

\def\cC{{\mathscr C}}

\def\cL{{\mathscr L}}

\def\cR{{\mathscr R}}

\def\cX{{\mathscr X}}
\def\cY{{\mathscr Y}}

\newtheorem{theorem}{Theorem}[section]
\newtheorem{proposition}[theorem]{Proposition}
\newtheorem{lemma}[theorem]{Lemma}
\newtheorem{corollary}[theorem]{Corollary}

\theoremstyle{definition}

\newtheorem{example}[theorem]{Example}

\newtheorem{remark}[theorem]{Remark}

\def\Ind#1{{\mathbbmss 1}_{_{\scriptstyle #1}}}

\def\eps{\varepsilon}

\def\<{\langle}
\def\>{\rangle}

\def\ua{\uparrow}
\def\da{\downarrow}

\def\wh{\widehat}
\def\wt{\widetilde}

\def\argmin{\mathop{\hbox{\rm arg\,min}}}

\newcommand{\bR}{\mathbb{R}}

\newcommand{\bN}{\mathbb{N}}

\newcommand{\bT}{\mathbb{T}}
\newcommand{\bZ}{\mathbb{Z}}

\usetikzlibrary{arrows.meta}

\tikzset{
	myarrow/.style={-{Triangle[length=3mm,width=1mm]}}
}
\begin{document}
	
	\title{Robust Faber--Schauder approximation\\ based on discrete observations of an antiderivative}
	\author{Xiyue Han and Alexander Schied\thanks{
			Department of Statistics and Actuarial Science, University of Waterloo, 200 University Ave W, Waterloo, Ontario, N2L 3G1, Canada. E-Mails: {\tt xiyue.han@uwaterloo.ca, aschied@uwaterloo.ca}.
			}}
	\date{\normalsize First version: November 20, 2022\\
		\normalsize  This version: October 10, 2024}
	\maketitle

	%%%%%%%%%%%%%%%%%%%%%%%%%%%%%%%%%%%%%%%%%%%%%%%%%%%%%%%%%%%%%%%%%%%%%
	\begin{abstract}
		We study the problem of reconstructing the Faber--Schauder coefficients of a continuous function $f$ from discrete observations of its antiderivative $F$. For instance, this question arises in financial mathematics when estimating the roughness of volatility from the integrated volatility of an asset price trajectory. Our approach starts with mathematically formulating the reconstruction problem through piecewise quadratic spline interpolation. We then provide a closed-form solution and an in-depth error analysis. These results lead to some surprising observations, which also throw new light on the classical topic of quadratic spline interpolation itself: They show that the well-known instabilities of this method can be located exclusively within the final generation of estimated Faber--Schauder coefficients, which suffer from non-locality and strong dependence on the initial value. By contrast, all other Faber--Schauder coefficients depend only locally on the data, are independent of the initial value, and admit uniform error bounds.  We thus conclude that a robust and well-behaved estimator for our problem can be obtained by simply dropping the final-generation coefficients from the estimated Faber--Schauder coefficients.
	\end{abstract}

\vspace{0.2cm}

\section{Introduction}	
Suppose we have discrete observations $\{F(t):t\in\bT_{n+1}\}$ of a function $F \in C^1[0,1]$ for $\bT_{n+1}:=\{k2^{-n-1}:k=0,\dots,2^{n+1}\}$ and 
some $n \in \bN_0$. We are interested in a robust reconstruction of the derivative $f = F'$ from these observations. 
This question arises whenever an effect is observed at certain time points, and we are interested in its rate of change. Problems of this type appear in a vast number of applications in science, engineering, statistics, economics, and finance. Concrete examples include the estimation of the density $f$ of a random variable from its empirical cumulative distribution function \cite{Devroye1987DensityEstimation, Izenman1991DensityEstimation} or peak detection in signal processing \cite{palshikar2009simple}. Here, we are mainly interested in reconstructing the Faber--Schauder coefficients of the derivative $f$. This objective is mainly motivated by the following two reasons.

First, reconstructing the Faber--Schauder coefficients naturally arises in the task of estimating the \lq roughness' of the spot volatility of a financial asset. Since the seminal work by Gatheral et al.~\cite{GatheralRosenbaum}, this problem has recently received substantial interest in the literature, and many estimation schemes were subsequently proposed; see, e.g., \cite{ContDas22, FukasawaTabatabake}. For instance, Han and Schied \cite{HanSchiedHurst} investigated this problem in a strictly pathwise manner and proposed to measure the degree of roughness of a financial time series by the so-called roughness exponent. Moreover, it was shown therein that a model-free estimator for the roughness exponent of a continuous function $f$ can be obtained through its Faber--Schauder coefficients $\theta_{n,k}$ as follows,
\begin{equation}\label{Hurst estimator eq}
	\wh R^\ast_n=1-\frac1n\log_2\sqrt{\sum_{k=0}^{2^n-1}\theta_{n,k}^2}.
\end{equation}
Thus, applying the estimator $\wh R^\ast_n$ requires direct observations of $f$ over the dyadic partition $\bT_{n+1}$. However, in volatility estimation, $f$ would correspond to the square of the spot volatility, but only an antiderivative of $f$ is observable in the form of the quadratic variation of the asset over discrete observation times. Therefore, it is necessary to construct a precise approximation of the $\ell_2$-norm as in \eqref{Hurst estimator eq} so as to estimate the roughness exponent of the volatility process. For the more refined estimators discussed in \cite{HanSchiedHurst}, one needs knowledge of the full vector $(\theta_{0,0},\theta_{1,0},\dots,\theta_{n,2^n-1})$ and not just of its $\ell_2$-norm of the $n^{\rm th}$ generation, $(\theta_{n,0},\theta_{n,1}, \theta_{n,2^n-1})$, as in \eqref{Hurst estimator eq}. For this reason, we focus here on a robust and efficient reconstruction of the Faber--Schauder coefficients from the observations $\{F(t): t\in \bT_{n+1}\}$. 

Second, the localized nature of the Faber-Schauder functions makes them effective for characterizing the smoothness of functions. The connection between the Faber--Schauder development and various function spaces has been well established in the literature. For instance, H\"older spaces can be characterized by the Faber--Schauder coefficients through Ciesielski's isomorphism \cite{Ciesielski1960}. Furthermore, Ciesielski et al.~\cite{Ciesielski1993Gaussian} provided a norm expressed in terms of Faber--Schauder coefficients, which is equivalent to the traditional norm equipped with Besov spaces; see also Section 6 in \cite{HanSchiedHurst}. Thus, to determine the function spaces associated with $f$, but based on its antiderivative $F$, a precise approximation of the Faber-Schauder coefficients is needed.

Now, let us discuss how we can estimate the Faber--Schauder coefficients of a function $f$ based on the observations $\{F(t): t \in \bT_{n+1}\}$. In \Cref{formula section}, we will show that this task is actually equivalent to determining the quadratic spline interpolation $\hat F_n$ over $\bT_{n+1}$ and then computing corresponding Faber--Schauder coefficients $\vartheta^{(n)}_{m,k}$ for $m \le n$ from the derivative $\hat f_n = \hat F'_n$. However, it is well-known that the interpolating function $\hat F_n$ cannot be uniquely determined by the given observations $\{F(t): t \in \bT_{n+1}\}$, depends non-locally on the given data, and is extremely sensitive to small changes to the estimate $\hat f_0$ of the unknown initial value $f(0)$. Thus, as illustrated in \Cref{Example Takagi} and the subsequent error analysis in \Cref{thm Faber errors}, Faber--Schauder coefficients estimated in the above manner are very imprecise and fail to accomplish any of the above-mentioned tasks. For this reason, we are here interested in an accurate reconstruction of the Faber--Schauder coefficients of $f$ that is independent of the initial estimate $\hat f_0$.

Our analysis starts by reformulating the quadratic spline interpolation via wavelet expansion. To this end, we construct a Schauder basis of $C^1[0,1]$ by integrating Faber--Schauder functions, which essentially transfers the problem of estimating the coefficients $\vartheta_{m,k}^{(n)}$ into a matrix-vector equation. This allows us to obtain explicit formulas for all $\vartheta^{(n)}_{m,k}$ in terms of the observed values of $F$; see \Cref{thm represent}. Furthermore, sharp error bounds between the approximated and actual Faber--Schauder coefficients can be derived from computing various matrix norms. From these explicit formulas and our error analysis, we can make three surprising observations, which locate the instability of the quadratic spline interpolation and the source of the inaccuracy of the estimated coefficients.

First, only the coefficients in the final $n^{\rm th}$ generation,  i.e., the numbers 
$\vartheta^{(n)}_{n,0},\dots, \vartheta^{(n)}_{n,2^n-1}$, depend on the initial value $\hat f_0$. Moreover, their sensitivity with respect to $\hat f_0$ grows by the factor $2^{n/2+2}$ as $n$ gets larger.  All other Faber--Schauder coefficients are 
independent of $\hat f_0$. Thus, the notorious instability of the quadratic spline interpolation with respect to the initial value $\hat f_0$  arises exclusively from the coefficients in the final generation.

Second, each Faber--Schauder coefficient $\vartheta^{(n)}_{n,k}$ in the final generation depends in a highly non-local manner on all observations in the interval $[0,(k+1)2^{-n}]$, which fails to retain the localized nature of the Faber--Schauder basis. By contrast, all other coefficients depend only on observations in a local neighbourhood as in the case of computing the Faber--Schauder coefficients directly from $f$.  

The third observation is a consequence of the error analysis between the true and estimated Faber--Schauder coefficients at different generations in \Cref{Faber--Schauder section}. These results imply in particular that $F$ can always be chosen in such a way that the error in generation $n$ exceeds the combined error of all generations up to $n-1$. More precisely, the error is larger by a factor of order $\mathcal{O}(2^n)$ for the $\ell_2$- and $\ell_1$-norms and by a factor of order $\mathcal{O}(2^{n/2})$ for the $\ell_\infty$-norm. 

These observations shed new light on the classical topic of quadratic spline interpolation as they locate the cause for the well-known instability of this method exclusively within the final generation of the Faber--Schauder coefficients. They also allow us to construct a robust approximation for the Faber--Schauder coefficients of $f$ by simply dropping the final generation of coefficients. This truncated approximation is independent of the initial value $\hat f_0$ and avoids the counterintuitive non-local dependence explained above. Moreover, the sharp bounds on the approximation error further support the fact that our approach provides a robust and accurate approximation. 

This paper is organized as follows. In \Cref{formula section}, we introduce our problem and state an explicit solution. In \Cref{Faber--Schauder section}, we state our bounds on the approximation error between the estimated and the true Faber--Schauder coefficients. These bounds can be translated into error estimates between the estimated functions $\hat f_n$ and $\hat F_n$ and the true functions $f$ and $F$. Corresponding results are stated in \Cref{section main functional error}. The proofs of our main results rely on an analysis of the approximation problem for $F$ in terms of the basis formed by the integrated Faber--Schauder functions. The corresponding analysis is presented in \Cref{Sec Matrix}. Most proofs for the results stated in \Cref{Section Main} are deferred to \Cref{Section 1 proofs section}.

\section{Statement of main results}\label{Section Main}

\subsection{Problem formulation and an explicit formula}\label{formula section}
Recall that the \emph{Faber--Schauder functions} are defined as 
\begin{equation*}
	e_{-1,0}(t) = t, \quad e_{0,0}(t) = \left(\min\{t,1-t\}\right)^+ ,\quad \text{and} \quad e_{m,k}(t):= 2^{-m/2}e_{0,0}(2^mt -k),
\end{equation*}
for $t \in \bR$, $m \in \bN_0$ and $k \in \bZ$. It is well known that the restriction of the Faber-Schauder functions to [0,1] forms a Schauder basis of $C[0,1]$. Indeed, for a given function $f\in C[0,1]$, the function
\begin{equation}\label{eq Faber approx}
	f_{n}:= f(0)+\theta_{-1,0}e_{-1,0} + \sum_{m = 0}^{n}\sum_{k = 0}^{2^m-1}\theta_{m,k}e_{m,k},
\end{equation} 
with coefficients $\theta_{-1,0}=f(1) - f(0)$ and 
\begin{equation}\label{eq interpolation}
	\theta_{m,k} = 2^{m/2}\left(2f\Big(\frac{2k+1}{2^{m+1}}\Big)-f\Big(\frac{k}{2^m}\Big)-f\Big(\frac{k+1}{2^m}\Big)\right) \quad \text{for} \quad 0 \le m \le n \quad \text{and} \quad 0 \le k \le 2^{m}-1,
\end{equation}
is just the piecewise linear interpolation of $f$ with supporting grid $\bT_{n+1}:=\{k2^{-(n+1)}:k=0,\dots , 2^{n+1}\}$, and so $\lim_n f_n = f$ uniformly. In this paper, we study the question of finding a robust approximation to the Faber--Schauder coefficients of $f$, if only the values $\{F(t):t\in\bT_{n+1}\}$ of an antiderivative, $F$, of $f$ are observed. 

Our starting point to this question is the following interpolation problem.  Based on the given data $\{F(t):t\in\bT_{n+1}\}$, estimate  Faber--Schauder coefficients $\vartheta^{(n)}_{m,k}$ and an initial value $\hat f_0$ such that
\begin{equation}\label{eq inverse Faber}
	\begin{split}
		\hat{F}_{n}(t) &= F(t) \quad \text{for all $t\in\bT_{n+1}$, where $\displaystyle \hat{F}_{n}(t) = F(0)+\int_{0}^{t}\hat{f}_{n}(s)\,ds$ for}\\
		\hat{f}_{n} &= \hat f_0+\vartheta^{(n)}_{-1,0}e_{-1,0} + \sum_{m = 0}^{n}\sum_{k = 0}^{2^m-1}\vartheta^{(n)}_{m,k}e_{m,k} .	\end{split}
\end{equation}

Since the function $\hat{f}_n$ is continuous and linear on each interval $[k2^{-n-1}, (k+1)2^{-n-1}]$, its antiderivative, $\hat{F}_n$, must be a piecewise quadratic $C^1$ function that coincides with $F$ on $\bT_{n+1}$. In other words, $\hat F_n$ is a quadratic spline interpolation of $F$ with supporting grid $\bT_{n+1}$. The properties of such a quadratic spline interpolation of $F$  have been studied before; see, e.g., Ingtem \cite[Theorem 1]{IngtemSpline} and Mettke et al. \cite[Theorem 3.2]{MettkeQuadraticSpline}. But it is well known and shown in the above literature, that the quadratic spline interpolation suffers from the following three issues,
\begin{enumerate}
	\item The number of unknowns is greater than the number of equations in \eqref{eq inverse Faber}. Thus, the coefficients $\vartheta^{(n)}_{m,k}$ and interpolating functions $\hat F_n$ and $\hat f_n$  are not uniquely determined from the data $\{F(t): t \in \bT_{n+1}\}$. To determine estimated coefficients $\vartheta^{(n)}_{m,k}$ and interpolating functions, the initial value $\hat f_0$ needs to be given. 
	\item  However, the interpolating functions $\hat F_n$ and $\hat f_n$ are highly dependent to the initial value $\hat f_0$, hence, so are estimated Faber--Schauder coefficients. As a result, this rules out the feasibility of estimating the initial value $\hat f_0$ first and then interpolating the data $\{F(t): t \in \bT_{n+1}\}$ with a quadratic spline $\hat F_n$; see Figure \ref{instability figure} for an illustration.
	\item The values $\hat F_n(s)$ depend in a nonlocal manner on the data $\{F(t):t \in \bT_{n+1}\}$, i.e, changing one data point $F(t)$ may affect the value $\hat F_n(s)$ even if $s$ is located far away from $t$.
	\end{enumerate}
	In this paper, we will revisit the classical problem \eqref{eq inverse Faber} within the context of the Faber--Schauder coefficients of $\hat f_n$, and we will show that, surprisingly, the source for all the issues (a), (b) and (c), can be identified. It is located in the final generation of the Faber--Schauder coefficients of $\hat f_n$, i.e., in the numbers $\vartheta^{(n)}_{n,0},\dots, \vartheta^{(n)}_{n,2^n-1}$. On the other hand, the coefficients $\vartheta^{(n)}_{m,k}$ for $m \le n-1$ do not possess the above undesirable characteristics and provide robust and accurate estimates for corresponding Faber--Schauder coefficients $\theta_{m,k}$. Some first clues to this observation can be seen from the following theorem, which provides a closed-form solution to the problem \eqref{eq inverse Faber}.
\begin{figure}
	\centering
	\includegraphics[width=7cm]{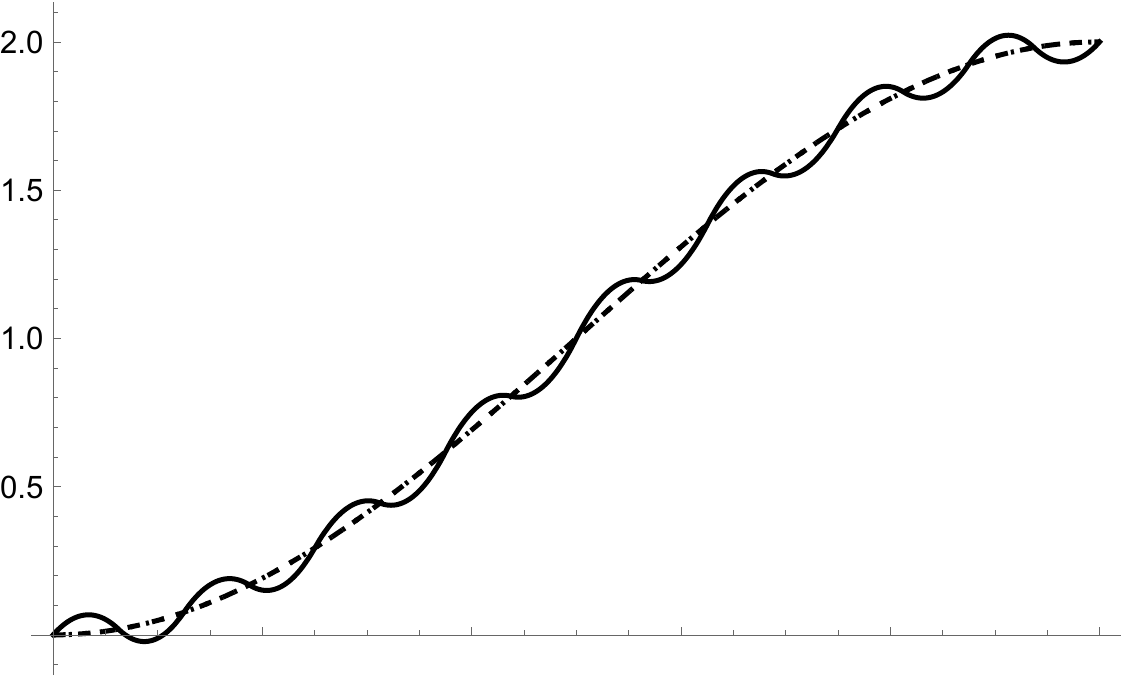}\qquad\qquad
	\includegraphics[width=7cm]{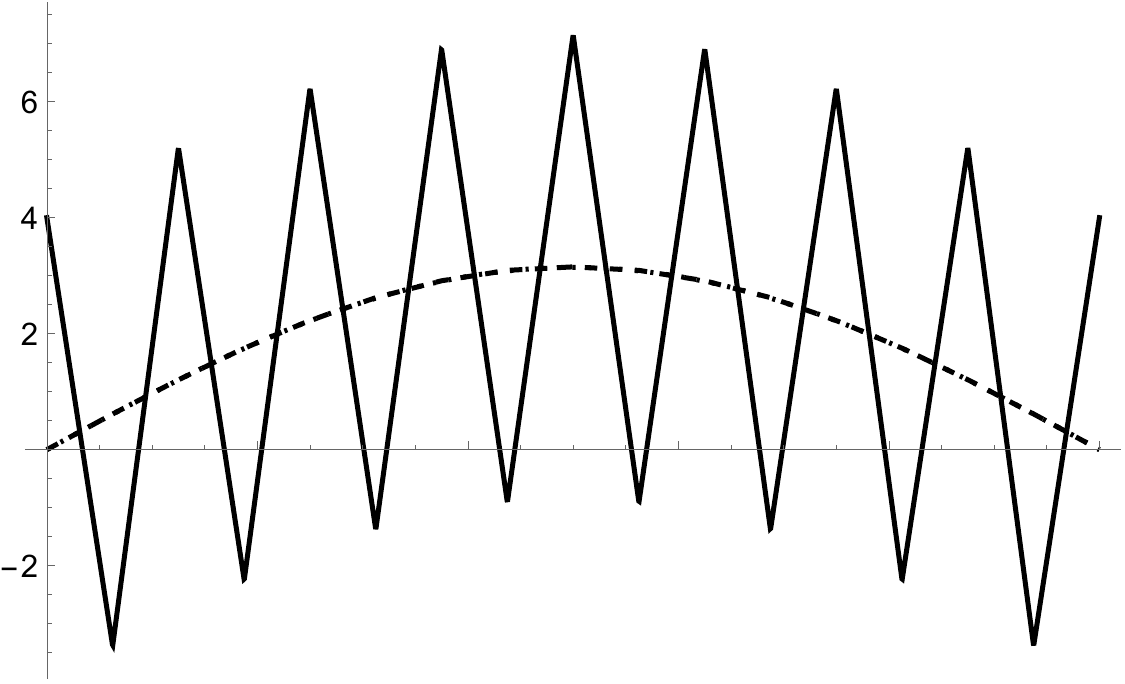}
	\caption{Graphical illustration of the instability of the solution to problem \eqref{eq inverse Faber} with $F(t)=1-\cos\pi t$ and $n=3$. While $F$ and $\hat F_n$ (left, dotted and dashed respectively) and $f$ and $\hat f_n$ (right, dotted and dashed respectively) are basically indistinguishable when taking $\hat f_0=0$, the value $\hat f_0=4$ yields wildly oscillating functions $\hat F_n$ and $\hat f_n$ (solid lines).}
	\label{instability figure}
\end{figure}

	\begin{theorem}\label{thm represent}
	For given $\hat f_0$, the coefficients $\{\vartheta^{(n)}_{m,k}\}$ solving problem \eqref{eq inverse Faber} are given as follows. For $m = -1$, we have
	\begin{equation}\label{eq zeta minus}
		\vartheta^{(n)}_{-1,0} = 2^{n+2}\sum_{j = 1}^{2^{n+1}}(-1)^j\left(F\Big(\frac{j}{2^{n+1}}\Big) - F\Big(\frac{j-1}{2^{n+1}}\Big)\right).
	\end{equation}
	For $0 \le m \le n-1$ and $0 \le k \le 2^m-1$, we have
	\begin{equation}\label{eq zeta}
		\scalemath{0.85}{\vartheta^{(n)}_{m,k} = 2^{n+m/2+2}\sum_{j = 1}^{2^{n-m}}(-1)^j\left(F\Big(\frac{k}{2^m}+\frac{j}{2^{n+1}}\Big)-F\Big(\frac{k}{2^m}+\frac{j-1}{2^{n+1}}\Big)+F\Big(\frac{k+1}{2^m}-\frac{j-1}{2^{n+1}}\Big)-F\Big(\frac{k+1}{2^m}-\frac{j}{2^{n+1}}\Big)\right),}
	\end{equation}
	Finally, for $m=n$ and $0 \le k \le 2^n-1$, we have
	\begin{equation}\label{eq vartheta final}
		\begin{split}
			\vartheta^{(n)}_{n,k} &=-2^{n/2+2}\hat f_0 -2^{3n/2+4}\sum_{j = 1}^{2k}(-1)^{j}\left(F\Big(\frac{j}{2^{n+1}}\Big)-F\Big(\frac{j-1}{2^{n+1}}\Big)\right)\\&+3 \cdot 2^{3n/2+2} \left(F\Big(\frac{2k+1}{2^{n+1}}\Big)-F\Big(\frac{2k}{2^{n+1}}\Big)\right)-2^{3n/2+2} \left(F\Big(\frac{2k+2}{2^{n+1}}\Big)-F\Big(\frac{2k+1}{2^{n+1}}\Big)\right).
		\end{split}
	\end{equation}
\end{theorem}

The first observation we can make from \Cref{thm represent} is that only the coefficients of the final generation, i.e., the numbers 
$\vartheta^{(n)}_{n,0},\dots, \vartheta^{(n)}_{n,2^n-1}$, depend on the initial value $\hat f_0$. As a result, the derivative $\hat f_n$ and, thus, the quadratic interpolation $\hat F_n$ are not uniquely determined by the given data $\{F(t): t \in \bT_{n+1}\}$. Moreover, the formula for the final-generation coefficients $\vartheta^{(n)}_{n,0},\dots, \vartheta^{(n)}_{n,2^n-1}$ contains the additive term $-2^{n/2+2}\hat f_0$. Consequently, the interpolating functions $\hat F_n$ and $\hat f_n$ are highly sensitive to this initial value $\hat f_0$. Furthermore, this additive term will translate
any error made in estimating $\hat f_0$ into an additional $2^{n/2+2}$-fold error for each final-generation coefficient. In particular, 
the approximation errors $|\vartheta^{(n)}_{n,k}-\theta_{n,k}|$ can be made arbitrarily large by varying the initial condition $\hat f_0$. By contrast, all other Faber--Schauder coefficients, i.e.,  $\vartheta^{(n)}_{m,k}$ with $m\le n-1$, are 
independent of $\hat f_0$. Thus, the instability  with respect to $\hat f_0$ observed in \Cref{instability figure} arises exclusively from the coefficients $\vartheta^{(n)}_{n,0},\dots, \vartheta^{(n)}_{n,2^n-1}$ of the final generation.

Our second observation is the following counterintuitive phenomenon, which arises independently of the initial value $\hat f_0$. On the one hand,  each estimated coefficient $\vartheta^{(n)}_{m,k}$ with  $m \le n-1$ is a linear combination of the values of $F$ over the grid $\bT_{n+1}$ within the interval $[2^{-m\vee0} k,2^{-m\vee0}(k+1)]$, where $a\vee b$ is shorthand for $\max\{a,b\}$. This interval coincides with the support of the corresponding Faber--Schauder function $e_{m,k}$. By contrast, for $m = n$, the estimated coefficient $\vartheta^{(n)}_{n,k}$ involves  all observations of $F$ within the interval $[0,2^{-n}(k+1)]$. Thus, altering any observed values of $F$ within the above interval results in a change in $\vartheta^{(n)}_{n,k}$, and thus, changes in $\hat f_n$ and $\hat F_n$. A graphical illustration of this phenomenon is presented in  \Cref{fig illustration}, and a concrete computation will be given in \Cref{Example Takagi}.  This nonlocal dependence of the coefficients $\vartheta^{(n)}_{n,k}$ on the observed data is counterintuitive and another cause for the instability of the quadratic spline interpolation.

\begin{figure}[h]
	\centering
	\includegraphics[width=14cm]{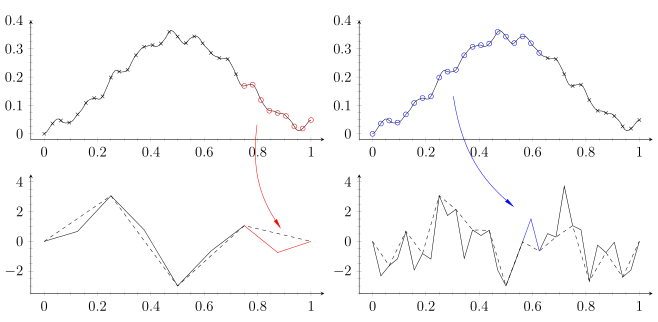}
		\caption{Graphical illustrations of the data dependence of the  coefficients $\vartheta^{(n)}_{m,k}$. In the bottom-left graph, we plot the approximating functions $\hat{f}_{4,1}$ and $\hat{f}_{4,2}$ in dotted and solid lines respectively, and the contribution $\vartheta^{(4)}_{2,3}e_{2,3}$ is highlighted in red. In the bottom-right graph, we plot the approximating functions $\hat{f}_{4,3}$ and $\hat{f}_{4,4}$ in dotted and solid lines respectively, and the contribution $\vartheta^{(4)}_{4,9}e_{4,9}$ is highlighted in blue. Respective data points of the function $F$ that contribute to the estimation of $\vartheta^{(4)}_{2,3}$ and $\vartheta^{(4)}_{4,9}$ are highlighted by circles  in the corresponding upper graphs. All data points that do not contribute to the computation of the coefficients in question are represented by cross marks.}\label{fig illustration}
\end{figure}

In the following section, we will discover yet another issue with the final-generation Faber--Schauder coefficients of $\hat f_n$, which arises even if the true initial value $f(0)=F'(0)$ is known, and $\hat f_0$ is set to that value. In this context, we will derive sharp upper and lower bounds on the approximation errors for the estimated Faber--Schauder coefficients up to generation $n-1$ and at generation $n$. That is, we will be looking at the Euclidean $\ell_p$-distances between the Faber--Schauder coefficients of $\hat f_n$ and the true Faber--Schauder coefficients of $f$ for $p\in\{1,2,\infty\}$ and under the assumption that $f(0)=F'(0)$ is known. On the one hand, we will discover robust upper error bounds for the estimated Faber--Schauder coefficients up to generation $n-1$, which hold universally for each function $F\in C^1[0,1]$. On the other hand, we will see that $F\in C^1[0,1]$ can always be chosen in such a way that the error in generation $n$ exceeds the combined errors of all generations up to and including generation $n-1$ by a factor of size $\mathcal{O}(2^n)$ for the $\ell_2$- and $\ell_1$-norms and by a factor of size  $\mathcal{O}(2^{n/2})$ for the $\ell_\infty$-norm.

In particular, the error analysis concerning the $\ell_2$-norm demonstrates that the estimated coefficients at the final generation, i.e., $\vartheta^{(n)}_{n,0},\dots, \vartheta^{(n)}_{n,2^n-1}$, lack the precision and robustness required for accurately estimating the roughness exponent of the volatility process; see \Cref{Remark roughness exponent} for further explanation. An extreme case will also be presented in \Cref{Example Takagi}, in which the approximation errors up to and including generation $n-1$ are all zero, while the error in generation $n$ can go to infinity even if the correct initial value $f(0)$ is known.

\medskip

\noindent{\bf Conclusion.} Our findings can informally be summarized as follows.
\begin{enumerate}
	\item The estimated coefficients $\vartheta^{(n)}_{m,k}$ for $m\le n-1$ are independent of the estimated initial value $\hat f_0$, and they admit robust error bounds for their convergence to the corresponding true Faber--Schauder coefficients of $f=F'$.
	\item By contrast, the final-generation coefficients $\vartheta^{(n)}_{n,0},\dots, \vartheta^{(n)}_{n,2^n-1}$ are strongly dependent on the estimated initial value $\hat f_0$. Even if the true initial value is known, they depend on $F$ in a nonlocal manner.  Furthermore,  the errors for their convergence to the corresponding true Faber--Schauder coefficients of $f=F'$ are highly sensitive with respect to the function $F$. 
\end{enumerate}

\begin{remark}\label{Remark roughness exponent}As mentioned in the introduction, the original motivation for this research stems from the problem of estimating the \lq\lq roughness" of the function $f$. It was shown in \cite[Proposition 4.1]{HanSchiedHurst} that under mild assumptions on $f$, the functional
	$$ \widehat R^\ast_n:=1-\frac1n\log_2 \sqrt{\sum_{k = 0}^{2^n-1}\theta^2_{n,k}} 
	$$
	is a strongly consistent estimator for the roughness exponent of $f$, provided that the true Faber--Schauder coefficients $\theta_{m,k}$ of $f$ are known. If, as in the case of volatility estimation (Gatheral et al. \cite{GatheralRosenbaum}), the function $f$ itself is unknown, and only the values of its antiderivative $F$ are observed on the discrete grid $\bT_{n+2}$ in the form of the integrated volatility, then our findings suggest to use the estimator
	\begin{equation}\label{eq cR}
		\wh \cR_n:=1-\frac1n\log_2 \sqrt{\sum_{k = 0}^{2^n-1}(\vartheta^{(n+1)}_{n,k}})^2,
	\end{equation}
	which is obtained by simply replacing the unknown Faber--Schauder coefficients $\theta_{n,k}$ with $\vartheta^{(n+1)}_{n,k}$. Based on observations $\{F(t): t \in \bT_{n+2}\}$, the final-generation coefficients $\vartheta^{(n+1)}_{n+1,0},\dots, \vartheta^{(n+1)}_{n+1,2^{n+1}-1}$ are obtainable but explicitly excluded from the construction. This approach is motivated and justified by the result in \Cref{cor eps}, which suggests that the approximation error in the generation $n+1$ will exceed that in generation $n$ by a factor of size $\mathcal{O}(2^n)$ in the $\ell_2$-norm. As a result, the estimator constructed with final-generation coefficients $\vartheta^{(n+1)}_{n+1,0},\dots, \vartheta^{(n+1)}_{n+1,2^{n+1}-1}$ are highly unbiased; see also \cite[Example 3.1]{HanSchiedHurstDerivative}. In contrast, by integrating the error analysis in \Cref{Faber--Schauder section} and the consistency criteria for $\wh R^\ast_n$ in \cite[Proposition 4.1]{HanSchiedHurst}, one can formulate conditions that guarantee the consistency of the estimator $\wh \cR_n$. Those conditions are very mild and explicitly stated in \cite[Equation (3.13)]{HanSchiedHurstDerivative}. Utilizing the error bounds in \Cref{thm Faber errors}, a formal proof for the consistency of $\wh \cR_n$ is provided in \cite[Proposition 3.3]{HanSchiedHurst}. However, by plugging the explicit solutions of $\vartheta^{(n+1)}_{n,k}$ into \eqref{eq cR}, one can verify that the estimator $\wh \cR_n$ is not scale-invariant. Nevertheless, scale-invariant modifications of $\wh \cR_n$ can be constructed following the approach in \cite[Section 2.2]{HanSchiedHurstDerivative}.
\end{remark}

We conclude this section with the following computationally explicit example, in which the formulas from \Cref{thm represent}
recover the exact Faber--Schauder coefficients of $f$  up to and including generation $n-1$. In generation $n$, however, the estimated Faber--Schauder coefficients will differ from the true ones to the extent that their approximation errors may go to infinity even if the correct initial value $f(0)$ is known.

\begin{example}\label{Example Takagi}
	Consider the Takagi class, which was introduced by Hata and Yamaguti \cite{hata1984Takagi} and motivated by the celebrated, continuous but nowhere differentiable Takagi function. It consists of all functions of the following form:
	\begin{equation}\label{eq def Takagi}
		f(t):= \sum_{m = 0}^\infty c_m \phi(2^m t) = \sum_{m = 0}^\infty2^{m/2}c_m\sum_{k = 0}^{2^m-1}e_{m,k}(t), \quad \text{for} \quad t \in [0,1],
	\end{equation}
	where $\{c_m\}_{m \in \bN_0}$ is an absolutely summable sequence and $\phi(t): = \min_{z \in \bZ}|t - z|$ denotes the tent map. The Takagi class has been studied intensively over the past decades; see, e.g., the surveys, \cite{AllaartKawamura} and \cite{Lagarias}. Letting $F(t):= \int_{0}^{t}f(s)\, ds$ and $\hat f_0:=f(0)=0$, the  coefficients solving \eqref{eq inverse Faber}
	are given as follows,
	\begin{equation*}
		\vartheta^{(n)}_{m,k} = 
		\begin{cases}
			2^{m/2}c_m \quad &\text{for} \quad 0 \le m \le n-1 \quad \text{and} \quad 0 \le k \le 2^m-1,\\
			2^{n/2}\sum_{m = n}^{\infty} c_m \quad &\text{for} \quad m = n \quad \text{and} \quad 0 \le k \le 2^n-1. 
		\end{cases}
	\end{equation*}
	The proof for these formulas is deferred to \Cref{Section 1 proofs section}. It follows that up to and including generation $n-1$, the approximation error between all true and estimated Faber--Schauder coefficients is zero. For generation $n$, however, 
	we have $|\vartheta^{(n)}_{n,k} -\theta_{n,k}|=2^{n/2}|\sum_{m=n+1}^\infty c_m|$. In particular, if the decay of $|\sum_{m=n+1}^\infty c_m|$ is slower than $\mathcal{O}(2^{-n/2})$, then the approximation errors $|\vartheta^{(n)}_{n,k} -\theta_{n,k}|$ go to infinity as $n\ua\infty$.
\end{example}

\subsection{Error analysis for the Faber--Schauder coefficients}
\label{Faber--Schauder section}

In this section, we consider again functions $F\in C^1[0,1]$ with derivative $F'=f$, and we suppose that $f(0)$ is known, so that we can set $\hat f_0:=f(0)$ in our problem \eqref{eq inverse Faber}. As a matter of fact, we may assume without loss of generality that $f(0)=\hat f_0=0$ and $F(0)=0$; otherwise, we consider the function $ F(t)-F(0)-f(0)t$, whose derivative has the same Faber--Schauder coefficients as $f$. We refer to \cite{IngtemSpline} for 
a method for estimating the unknown value $f(0)$ from the data $\{F(t):t\in\bT_{n+1}\}$, where the estimate of $f(0)$ is set to be the minimizer $\argmin \norm{\hat f_n}_{L^2[0,1]}$, and $\hat f_n$ is as in \eqref{eq inverse Faber}.

For studying the Euclidean $\ell_p$-norms of the approximation errors, it will be convenient to denote 
the Faber--Schauder coefficients of $f$ at or up to generation $m\ge-1$  by the respective column vectors $\bar{\bm \theta}_m$ and $\bm \theta_m$, i.e., 
\begin{equation}\label{eq vector}
	\bar{\bm \theta}_m = \big(\theta_{m,0}, \theta_{m,1} \cdots, \theta_{m,2^{m \vee 0}-1}\big)^\top \in \bR^{2^{m \vee 0}}\quad \text{and} \quad \bm \theta_m = (\bar{\bm \theta}^\top_{-1}, \bar{\bm \theta}^\top_0,\cdots, \bar{\bm\theta}^\top_{m})^\top \in \bR^{2^{m+1}}.
\end{equation}
An analogous notation will be used for the coefficients in \Cref{thm represent}  solving \eqref{eq inverse Faber}. That is, for $-1 \le m \le n$, we set column vectors	
\begin{equation*}
	\bar{\bm \vartheta}^{(n)}_m = \big(\vartheta^{(n)}_{m,0}, \vartheta^{(n)}_{m,1} \cdots, \vartheta^{(n)}_{m,2^{m \vee 0}-1}\big)^\top \in \bR^{2^{m \vee 0}}\quad \text{and} \quad \bm \vartheta^{(n)}_m = ((\bar{\bm \vartheta}^{(n)}_{-1})^\top, (\bar{\bm \vartheta}^{(n)}_{0})^\top,\cdots, (\bar{\bm \vartheta}^{(n)}_{m})^\top)^\top \in \bR^{2^{m+1}}.
\end{equation*}
Note that in particular $\bm \vartheta^{(n)}_n = \bm \vartheta^{(n)}$, and both notions will be used interchangeably in this paper. In this section, we study the approximation errors $\bm \theta_m - \bm \vartheta^{(n)}_m$  and $\bar{\bm \theta}_m - \bar{\bm \vartheta}^{(n)}_m$  by giving upper and lower bounds on their Euclidean $\ell_p$-norms for $p\in\{1,2,\infty\}$.  These bounds will be given in terms of the $\ell_p$-norms of the  $2^{n+1}$-dimensional column vector $\bm z_{n+1} := (z^{(n+1)}_1,\dots, z^{(n+1)}_{2^{n+1}})^\top$, where
\begin{equation}\label{eq def zn}
	z^{(n+1)}_i = 2^{3(n+1)/2}\sum_{m = n+1}^{\infty}2^{-3m/2}\sum_{k = 2^{m-n-1}(i-1)}^{2^{m-n-1}i-1}\theta_{m,k} \quad \text{for} \quad 1 \le i \le 2^{n+1}.
\end{equation}
The following proposition describes the asymptotic behaviour of the $\ell_p$-norms of $\bm z_{n+1}$ as $n \ua \infty$. In particular, some of the following criteria guarantee that the $\ell_p$-norm of $\bm z_{n+1}$ converges to zero as $n\ua\infty$. Combined with our subsequent error bounds, this proposition will yield a sufficient criterion for the convergence of the estimated Faber--Schauder coefficients to their true values. Finally, it is worthwhile to point out that the following convergence criterion is not necessary. For instance, \Cref{lemma zn} yields that for the class of Takagi functions in \Cref{Example Takagi}, we clearly have $\bm z_{n+1} = 0$.

We will write $f \in C^{k,\alpha}[0,1]$ if the $k^{\rm th}$ order derivative $f^{(k)}$ exists and is H\"{o}lder continuous with exponent $\alpha \in (0,1]$ for $k \ge 1$. If $k = 0$, we will write $f \in C^{0,\alpha}[0,1]$ if $f$ itself is $\alpha$-H\"{o}lder continuous. 

\begin{proposition}\label{zn+1 proposition} Suppose that $f \in C^{k,\alpha}[0,1]$. For $k \in \{0,1\}$, we have 
	\begin{equation}\label{eq norm z holder}
		\norm{\bm z_{n+1}}_{\ell_p} = \begin{cases}
			\mathcal{O}(2^{-(n+1)(\alpha-k-\frac{3}{2})}) &\quad \text{for} \quad p = 1,\\
			\mathcal{O}(2^{-(n+1)(\alpha-k)}) &\quad \text{for} \quad p = 2,\\
			\mathcal{O}(2^{-(n+1)(\alpha-k-\frac{1}{2})}) &\quad \text{for} \quad p = \infty.\\
		\end{cases}
	\end{equation}

\end{proposition}

The following theorem gives sharp upper and lower bounds for the $\ell_2$-norm of the approximation error $\bm \vartheta^{(n)}_{n-1} - \bm \theta_{n-1}$ of the estimated Faber--Schauder coefficients up to and including generation $n-1$.

\begin{theorem}\label{thm Faber errors}
	The following assertions hold.
	\begin{enumerate}
		\item \label{thm Faber errors item a} For arbitrary choice of $F \in C^1[0,1]$, we have 
		\begin{equation}\label{eq error 1}
			\bnorm{\bm \vartheta^{(n)}_{n-1} - \bm \theta_{n-1}}_{\ell_2} \le \frac{1}{2}\norm{\bm z_{n+1}}_{\ell_2}.
		\end{equation}
		\item \label{thm Faber errors item b} For each $\varepsilon > 0$ and $-1 \le m \le n-1$, the function $F \in C^1[0,1]$ can be chosen in such a way that
		\begin{equation*}
			\bnorm{\bm \vartheta^{(n)}_{n-1} - \bm \theta_{n-1}}_{\ell_2}\ge \bnorm{\bar{\bm \vartheta}^{(n)}_{m} - \bar{\bm \theta}_{m}}_{\ell_2} \ge  (1- \varepsilon)\frac{1}{2}\norm{\bm z_{n+1}}_{\ell_2}.
		\end{equation*}
	\end{enumerate}
\end{theorem}

Our next results looks specifically at the error term $\|\bar{\bm \vartheta}^{(n)}_{n} - \bar{\bm \theta}_{n}\|_{\ell_2}$ for the Faber--Schauder coefficients of the $n^\text{th}$ generation.

\begin{theorem}\label{thm Faber errors last}
	For $n \ge 2$, the following assertions hold:
	\begin{enumerate}
		\item \label{thm Faber errors last item a} For arbitrary choice of $F \in C^1[0,1]$, we have 
		\begin{equation}\label{eq error 2}
			\bnorm{\bar{\bm \vartheta}^{(n)}_{n} - \bar{\bm \theta}_{n}}_{\ell_2} \le \sqrt{2\left(1 - \cos\frac{\pi}{2^{n+1}}\right)^{-1} + \frac{3}{4}}\cdot\norm{\bm z_{n+1}}_{\ell_2}.
		\end{equation}
		
		\item \label{thm Faber errors last item b} For each $\varepsilon > 0$, the function $F \in C^1[0,1]$ can be chosen in such a way that 
		\begin{equation}\label{eq error 3}
			\bnorm{\bar{\bm \vartheta}^{(n)}_{n} - \bar{\bm \theta}_{n}}_{\ell_2} \ge  (1- \varepsilon)\sqrt{2\left(1 - \cos\frac{\pi}{2^{n+1}}\right)^{-1} - \frac{3}{4}}\cdot\norm{\bm z_{n+1}}_{\ell_2} .
		\end{equation}
	\end{enumerate}
\end{theorem}

For functions $g(x)$ and $h(x)$, we say $g \sim h$ as $x \to x_0$  if $g(x)/h(x) \to 1$ for $x\to x_0$. Since $1-\cos x\sim x^2/2$ as $x\da0$, the coefficients in \eqref{eq error 2} and \eqref{eq error 3} behave asymptotically as
\begin{equation}\label{eq error asymptotic}
	\sqrt{2\left(1 - \cos\frac{\pi}{2^{n+1}}\right)^{-1} \pm \frac{3}{4}}  \sim \frac{2^{n+1}}{\pi} \quad \text{for} \quad n \ua \infty.
\end{equation}
As a matter of fact, by combining the elementary inequality $1-\cos x\le x^2/2$ with \Cref{thm Faber errors} (a) and \Cref{thm Faber errors last}
(b), we obtain the following corollary. As announced above, it states that the $\ell_2$-norm of the error in the final generation of the Faber--Schauder coefficients can be substantially larger than a factor of size $\mathcal{O}(2^n)$ times the $\ell_2$-norm of the error of all previous generations combined. Moreover, it is worthwhile to point out that both vectors $\bar{\bm \vartheta}^{(n)}_{n} - \bar{\bm \theta}_{n}$ and $\bm \vartheta^{(n)}_{n-1} - \bm \theta_{n-1}$ are of length $2^n$.

\begin{corollary}\label{cor eps}For all $n\ge2$ and $\eps>0$, the function $F \in C^1[0,1]$ can be chosen in such a way that 
	$$\bnorm{\bar{\bm \vartheta}^{(n)}_{n} - \bar{\bm \theta}_{n}}_{\ell_2}\ge 2(1-\eps) \sqrt{\frac{2^{2n+2}}{\pi^2}-\frac34}\bnorm{\bm \vartheta^{(n)}_{n-1} - \bm \theta_{n-1}}_{\ell_2}.
	$$
\end{corollary}

Our bounds for the $\ell_1$- and $\ell_\infty$-norms, as stated in the following theorem,  are even stronger than \eqref{eq error 2} and \eqref{eq error 3}, because the constants for the upper and lower bounds are the same.

\begin{theorem}\label{thm Faber errors ellinf ellp} We fix  $n\ge2$ and let for $m=-1,0,\dots, n-1$ and $p\in\{1,\infty\}$, 
	\begin{equation*}
		\bar c_{p,m}:=\begin{cases}2^{\frac{m \vee 0-n-3}2}&\text{if $p=1$},\\
			2^{\frac{n-m \vee 0-1}2}&\text{if $p=\infty$},
		\end{cases}\qquad 
		c_{p,m}:=\begin{cases}2^{-\frac{n+3}2}\big(\frac{2^{\frac{m+1}2}-1}{2^{1/2}-1}\Ind{\{m\ge0\}}+1\big)&\text{if $p=1$},\\
			2^{\frac{n-1}2}&\text{if $p=\infty$},
		\end{cases} 
	\end{equation*}
	$$\bar c_{p,n}:=\begin{cases}2^{n+\frac12}-2^{-3/2}&\text{if $p=1$},\\
		2^{n+\frac32}-\sqrt2&\text{if $p=\infty$}.
	\end{cases}
	$$
	Then the following assertions hold for $m=-1,0,\dots, n-1$ and $p\in\{1,\infty\}$.
	\begin{enumerate}
		\item \label{thm Faber errors ellinf item 1}For arbitrary choice of $F \in C^1[0,1]$, we have $\|\bar {\bm \vartheta}^{(n)}_m - \bar{\bm \theta}_m\|_{\ell_p} \le \bar c_{p,m}\norm{\bm z_{n+1}}_{\ell_p}$, and for any given $\varepsilon > 0$, the function $F$ can be chosen in such a way that $\|\bar {\bm \vartheta}^{(n)}_m - \bar{\bm \theta}_m\|_{\ell_p} >  (1-\eps)\bar c_{p,m}\norm{\bm z_{n+1}}_{\ell_p} $.
		\item \label{thm Faber errors ellinf item 2}For arbitrary choice of $F \in C^1[0,1]$, we have $\| {\bm \vartheta}^{(n)}_m - {\bm \theta}_m\|_{\ell_p} \le c_{p,m}\norm{\bm z_{n+1}}_{\ell_p}$, and for any given $\varepsilon > 0$, the function $F$ can be chosen in such a way that $\|{\bm \vartheta}^{(n)}_m - {\bm \theta}_m\|_{\ell_p} >  (1-\eps)c_{p,m}\norm{\bm z_{n+1}}_{\ell_p} $.
		\item \label{thm Faber errors ellinf item 3}For arbitrary choice of $F \in C^1[0,1]$, we have $\| \bar{\bm \vartheta}^{(n)}_n - \bar{\bm \theta}_n\|_{\ell_p} \le \bar c_{p,n}\cdot\norm{\bm z_{n+1}}_{\ell_p}$, and for any given $\varepsilon > 0$, the function $F$ can be chosen in such a way that $\|\bar{\bm \vartheta}^{(n)}_n - \bar{\bm \theta}_n\|_{\ell_p} >  (1-\eps)\bar c_{p,n}\norm{\bm z_{n+1}}_{\ell_p} $.
		\item  \label{thm Faber errors ellinf item 4} For all $n\ge2$ and $\eps>0$, the function $F \in C^1[0,1]$ can be chosen in such a way that 
		$\|\bar{\bm \vartheta}^{(n)}_n - \bar{\bm \theta}_n\|_{\ell_p}\ge(1-\eps)\bar c_{p,n}c_{p,n-1}^{-1} \| {\bm \vartheta}^{(n)}_{n-1} - {\bm \theta}_{n-1}\|_{\ell_p}.
		$
	\end{enumerate}
\end{theorem}

Note that part \ref{thm Faber errors ellinf item 4}  in the preceding theorem plays the role of \Cref{cor eps} for the $\ell_1$- and $\ell_\infty$-norms and that the coefficients in that part satisfy 
\begin{equation*}
	\frac{\bar c_{1,n}}{c_{1,n-1}}=\frac{\big(2^{n+\frac{1}{2}} - 2^{-3/2}\big)}{2^{-\frac{n+3}{2}}\left(\frac{2^{n/2}-1}{2^{1/2}-1} + 1\right)} \sim \big(\sqrt{2}-1\big)2^{n+2} \qquad\text{and}\qquad 
	\frac{\bar c_{\infty,n}}{c_{\infty,n-1}}=\frac{2^{n + \frac{3}{2}} - \sqrt{2}}{2^{\frac{n-1}{2}}} \sim 2^{\frac{n}{2}+2}\qquad\text{as} \quad n \ua \infty.
\end{equation*}
Thus, for all $p\in\{1,2,\infty\}$, the ratio between $\|\bar{\bm \vartheta}^{(n)}_n - \bar{\bm \theta}_n\|_{\ell_p}$ and $ \| {\bm \vartheta}^{(n)}_{n-1} - {\bm \theta}_{n-1}\|_{\ell_p}$ can become arbitrarily large as $n$ goes to infinity. This underscores our claim that $ {\bm \vartheta}^{(n)}_{n-1} $ is a more robust estimate for the true Faber--Schauder coefficients of $f$ than $ {\bm \vartheta}^{(n)}_{n} $.

Theorems \ref{thm Faber errors}, \ref{thm Faber errors last}, and \ref{thm Faber errors ellinf ellp} are actually corollaries to our results stated in \Cref{Sec Matrix}. More specifically, they follow immediately from the equations \eqref{operator norm upper bound} and \eqref{operator norm lower bound} and \Cref{thm V norm}.

Finally, we will conclude this section by analyzing the difference $\vartheta^{(n+1)}_{m,k} - \vartheta^{(n)}_{m,k}$ in an element-wise manner for $-1 \le m \le n-1$ and $0 \le k \le 2^m-1$. Our analysis starts with constructing alternative representations for the estimated coefficients $\vartheta^{(n)}_{m,k}$ for $-1 \le m \le n-1$. To this end, let us denote the Faber--Schauder coefficients of $F$ by $\rho_{m,k}$. That is, $\rho_{-1,0} = F(1) - F(0)$ and 
\begin{equation*}
	\rho_{m,k} := 2^{m/2}\left(2F\Big(\frac{2k+1}{2^{m+1}}\Big)-F\Big(\frac{k}{2^m}\Big)-F\Big(\frac{k+1}{2^m}\Big)\right) \quad \text{for} \quad 0 \le m \le n \quad \text{and} \quad 0 \le k \le 2^{m}-1,
\end{equation*}
In addition, for given $s \ge 0$, let us introduce the notation 
\begin{equation*}
	\scalemath{0.9}{\theta_{m,k}(s):= 2^{m/2}\left(2f\Big(\frac{2k+1}{2^{m+1}}+s\Big)-f\Big(\frac{k}{2^m} +s\Big)-f\Big(\frac{k+1}{2^m} +s\Big)\right) \quad \text{for} \quad 0 \le m \le n \quad \text{and} \quad 0 \le k \le 2^m-1.}
\end{equation*}
That is, $\theta_{m,k}(s)$ are the Faber--Schauder coefficients of the function $t \mapsto f(s+t)$ for given $s \ge 0$. One can avoid undefined arguments of functions in case $s + t > 1$ by assuming without loss of generality that all occurring functions on $[0,1]$ are, in fact, defined on all of $[0,\infty)$. With these notations, we are able to obtain the following lemma, which provides alternative representations for coefficients $\vartheta^{(n)}_{m,k}$ for $-1 \le m \le n-1$. This lemma also points out that any coefficients $\vartheta^{(n)}_{m,k}$ for $-1 \le m \le n-1$ can be exclusively represented by the Faber--Schauder coefficients of $F$ at generation $n$. 

\begin{lemma}\label{lemma Faber antiderivative}
	For any given $\hat f_0$, the coefficients $\{\vartheta^{(n)}_{m,k}\}$ solving problem \eqref{eq inverse Faber} are given as follows. For $m = -1$, we have
	\begin{equation}\label{eq vartheta Faber minus}
		\vartheta^{(n)}_{-1,0} = -2^{n/2+2}\sum_{j = 0}^{2^n-1}\rho_{n,j}.
	\end{equation}
	For $0 \le m \le n-1$ and $0 \le k \le 2^m-1$, we have
	\begin{align}
		\vartheta^{(n)}_{m,k} &= 2^{(n+m)/2 + 2}\sum_{j = 0}^{2^{n-m-1}-1} \left(\rho_{n,2^{n-m}k + 2^{n-m-1} + j} - \rho_{n,2^{n-m}k + j}\right)\label{eq vartheta Faber}\\&= 2^{(n-m)/2+1}\sum_{j = 0}^{2^{n-m-1}-1}\int_0^{1} \theta_{n,2^{n-m}k + j}\left(\frac{s}{2^{m+1}}\right)\,ds.\label{eq integral representation}
	\end{align}
\end{lemma}

The above representations allow us to obtain upper bounds for the difference $\vartheta^{(n+1)}_{m,k} - \vartheta^{(n)}_{m,k}$ for $-1 \le m \le n-1$. To this end, recall that we write $f \in C^{k,\alpha}[0,1]$ if the $k^{\rm th}$ order derivative $f^{(k)}$ exists and is H\"{o}lder continuous with exponent $\alpha \in (0,1]$ for $k \ge 1$. If $k = 0$, we will write $f \in C^{0,\alpha}[0,1]$ if $f$ itself is $\alpha$-H\"{o}lder continuous. The following theorem establishes upper bounds for the difference $\vartheta^{(n+1)}_{m,k} - \vartheta^{(n)}_{m,k}$ based on the smoothness of $f$. 

\begin{theorem}\label{thm element minus}
	Suppose that $f \in C^{i,\alpha}[0,1]$ for $i \in \bN_0$ and $\alpha \in (0,1]$, then the following assertions hold: 
	\begin{enumerate}
		\item For $m = -1$ and $ i  \in\{ 0,1,2\}$ , there exists $\kappa > 0$ such that 
		\begin{equation}\label{eq element minus bound holder}
			\big|\vartheta^{(n+1)}_{-1,0} - \vartheta^{(n)}_{-1,0}\big| \le \kappa \cdot 2^{(1-i-\alpha)n} .\end{equation}
		Furthermore, if $f \in C^3[0,1]$, then 
		\begin{equation}\label{eq element minus bound holder 2}
			\big|\vartheta^{(n+1)}_{-1,0} - \vartheta^{(n)}_{-1,0}\big| \le 2^{-2n-6}\sup\limits_{t \in [0,1]}|f'''(t)|.
		\end{equation}
		\item For given $0 \le m \le n-1$, $0 \le k \le 2^m-1$, and $ i \in\{ 0,1,2,3\}$, there exists $\kappa > 0$ such that 
		\begin{equation}\label{eq element bound holder}
			\big|\vartheta^{(n+1)}_{m,k} - \vartheta^{(n)}_{m,k}\big| \le \kappa \cdot 2^{(2-i-\alpha)n-3m/2} .
		\end{equation}
		Furthermore, if $f \in C^4[0,1]$, then 
		\begin{equation}\label{eq element bound holder 2}
			\big|\vartheta^{(n+1)}_{m,k} - \vartheta^{(n)}_{m,k}\big| \le 2^{-2n-3m/2-8}\sup_{t \in [0,1]} |f^{(4)}(t)|.
		\end{equation}
	\end{enumerate}
\end{theorem}

\subsection{Error analysis for the approximating functions}\label{section main functional error}

As discussed above, our analysis of the estimated Faber--Schauder coefficients 
suggests that dropping the final generation $\bar{\bm \vartheta}^{(n)}_{n}=( \vartheta^{(n)}_{n,0},\dots,  \vartheta^{(n)}_{n,2^n-1})^\top$ of coefficients  leads to a more robust estimate $\bm \vartheta^{(n)}_{n-1}$ of the true Faber--Schauder coefficients of $f$. A Faber--Schauder expansion based on these remaining coefficients should in turn lead to more robust approximations for the unknown functions $f$ and $F$. There is, however, a price to be paid for this robustness: the new approximating functions will no longer be interpolations of $f$ and $F$. That is, they may no longer coincide with $f$ and $F$ on the observation grid $\bT_{n+1}$.

To analyze these induced approximating functions, we let again $F\in C^1[0,1]$ with $f=F'$ and assume as in \Cref{Faber--Schauder section} that $F(0)=f(0)=0$. For $m=-1,0,\dots, n$, we then define 
\begin{equation}\label{eq linear interpolation}
	\hat{f}_{n,m} = \sum_{i = -1}^{m}\sum_{k = 0}^{2^{i \vee 0}-1} \vartheta^{(n)}_{i,k}e_{i,k} \quad \text{and} \quad \hat{F}_{n,m} = \sum_{i = -1}^{m}\sum_{k = 0}^{2^{i \vee 0}-1} \vartheta^{(n)}_{i,k}\psi_{i,k} ,\end{equation}
where the coefficients $ \vartheta^{(n)}_{i,k}$ are as above.
Then both $\hat{f}_{n,m}$ and the function $f_m$ defined in \eqref{eq Faber approx} will  be piecewise linear with  supporting grid $\bT_{m+1}$. 

Error analysis for quadratic spline interpolation has been well studied in the existing literature. Typically, results such as those in \cite{MettkeQuadraticSpline} are based on the  mean-value theorem and provide error bounds for the $L_\infty[0,1]$-norm in terms of the first-order modulus of continuity of $f$. By contrast, our results presented here are based on  the theorems in  \Cref{Faber--Schauder section}. They  provide error bounds  with respect to the $L_p[0,1]$-norms   for $p\in\{1,2,\infty\}$ and in terms of the second-order modulus of continuity of $f$,
\begin{equation*}
	\omega_2(f,t):= \sup_{\substack{\tau_1,\tau_2 \in [0,1]\\|\tau_1 - \tau_2| \le t}}\left\{\left|f(\tau_1)+f(\tau_2)-2f\Big(\frac{\tau_1 + \tau_2}{2}\Big)\right|\right\}.
\end{equation*}

Our results will involve the following constants, which we define for fixed $n\ge2$,
\begin{equation}
	\begin{split}
		A_p&:=\begin{cases}
			2^{-(n+3)/2}&\text{for $p=1$,}\\
			1/\sqrt3&\text{for $p=2$,}\\
			2^{(n+1)/2}&\text{for $p=\infty$,}\\
		\end{cases}\qquad
		B_p:=\begin{cases}
			2^n&\text{for $p=1$,}\\
			2^{n+1}/\sqrt3&\text{for $p=2$,}\\
			2^{n+1}&\text{for $p=\infty$,}\\
		\end{cases}\\
		C_p&:=\begin{cases}
			2^{-(n+1)/2}&\text{for $p=1$,}\\
			\gamma_2&\text{for $p=2$,}\\
			2^{(n+3)/2}&\text{for $p=\infty$,}\\
		\end{cases}\qquad
		D_p:=\begin{cases}
			2^{n+1}&\text{for $p=1$,}\\
			2^{n+1}\gamma_2&\text{for $p=2$,}\\
			2^{n+2}&\text{for $p=\infty$,}\\
		\end{cases}
	\end{split}
\end{equation}
where
$$\gamma_2:=\frac1{\sqrt3}\bigg(1+2^{-n-1}\sqrt{2\Big(1-\cos\frac{\pi}{2^{n+1}}\Big)^{-1}+\frac34}\bigg).
$$

\begin{theorem}\label{thm functional error all ps} The following inequalities hold for fixed $n\ge2$, $0 \le m \le n-1$, and $p\in\{1,2,\infty\}$, 
	\begin{align*}
		\bnorm{f_m - \hat{f}_{n,m}}_{L_p[0,1]} &\le A_p\norm{\bm z_{n+1}}_{\ell_p} \le B_p\sum_{k = n+1}^{\infty} \omega_2(f,2^{-k}),\\
		\bnorm{f_n - \hat{f}_{n,n}}_{L_p[0,1]} &\le C_p\norm{\bm z_{n+1}}_{\ell_p} \le D_p\sum_{k = n+1}^{\infty} \omega_2(f,2^{-k}).
	\end{align*}
	
\end{theorem}

%Comparing the above two inequalities, there is no strong evidence that the approximation $\hat{f}_{n,n}$ outperforms the approximation $\hat{f}_{n,n-1}$. On the contrary, note that \Cref{thm Faber errors ell1} implies that for any $\varepsilon > 0$, the function $F$ can be chosen so that 
%\begin{equation*}
%	\bbnorm{\sum_{k = 0}^{2^n-1}\big(\vartheta^{(n)}_{n,k} - \theta_{n,k}\big)e_{n,k}}_{L_1[0,1]} = 2^{-3n-2}\bnorm{\bar{\bm \vartheta}^{(n)}_n - \bar{\bm \theta}_n}_{\ell_1} \ge 2^{-(n+3)/2}\norm{\bm z_{n+1}}_{\ell_1} - \varepsilon.
%\end{equation*}
%This means whatever $f$ is, some poorly estimated $n^\text{th}$-order Faber--Schauder coefficients $\bm \vartheta^{(n)}_n$ shall gives an approximation error that is as large as $\hat{f}_{n,n-1} - f_{n-1}$. 

\begin{remark}Recall from \eqref{zn+1 proposition} that $\norm{\bm z_{n+1}}_{\ell_p}\to0$ as $n\ua\infty$ if $f$ is continuously differentiable with H\"older continuous derivative. Thus, in this case, \Cref{thm functional error all ps}  yields that $\bnorm{f_{m_n} - \hat{f}_{n,m_n}}_{L_p[0,1]} \to0$ if $(m_n)$ is any sequence with $m_n\le n$.
\end{remark}

\begin{remark} \Cref{thm functional error all ps} gives error bounds between the truncated function $f_m$ and the approximation $\hat f_{n,m}$.  Error bounds for   $\norm{f-\hat f_{n,m}}_{L_p[0,1]}$ can be obtained as follows. For instance, for $p=1$, we observe that $
	\norm{e_{m,k}}_{L_1[0,1]}= 2^{-3m/2-2}
	$ and $|\theta_{m,k}|\le 2^{m/2}\omega_2(f,2^{-m})$. From here, it is easy to show that
	$$\bnorm{f - f_m}_{L_1[0,1]}\le \frac{1}{4}\sum_{k = m+1}^\infty \omega_{2}(f,2^{-k}).
	$$
	The desired error bound for $\norm{f - \hat f_{n,m}}_{L_1[0,1]}$ can now be obtained in a straightforward manner from the triangle inequality and  \Cref{thm functional error all ps}. Similar arguments also work in the cases of the $L_2$- and $L_\infty$-norms, where we can use that $	\norm{e_{m,k}}_{L_2[0,1]}  = 2^{-m-1}/{\sqrt{3}}$ and $\norm{e_{m,k}}_{L_\infty[0,1]} = 2^{-m/2-1}$. In the $L^\infty$-case, one can alternatively use the fact that  $\norm{f - f_m}_{L_\infty[0,1]} \le 4\omega_{2}(f, 2^{-m-1})$, which was established in \cite[Theorem 1]{MatveevSchauder}. \end{remark}

By combining the bounds  explained in the preceding remark with the next theorem, one obtains bounds for the respective error $|F-\hat F_{n,n}|$ in various $L_p$-norms.

\begin{theorem}\label{thm functional errors} For $n\ge2$, and $p\in\{1,2,\infty\}$, we have
	$$\bnorm{F - \hat{F}_{n,n}}_{L_p[0,1]} \le a_p\bnorm{f - \hat{f}_{n,n}}_{L_p[0,1]} ,
	$$
	where $a_1=2^{-n-1}$, $a_2=2^{-n}/\pi$, and $a_\infty=2^{-n-2}$.
\end{theorem}

\section{Wavelet and matrix analysis of problem \eqref{eq inverse Faber}.}\label{Sec Matrix}

\subsection{A wavelet system based on Faber--Schauder basis.}\label{Sec basis}
In this section, we present our approach to proving the main results in this paper.  It consists of a formulation of the problem \eqref{eq inverse Faber} by means of linear algebra. 
To this end, consider the following dyadic basis. We define $\psi_{-1,0}(t) = t^2/2$,
\begin{equation}\label{psi def eq}
	\psi_{0,0}(t) = 
	\begin{cases}
		0 &\quad t \in (-\infty,0),\\
		\frac{1}{2}t^2  &\quad      t \in [0,\frac{1}{2}),\\
		\frac{1}{4}- \frac{1}{2}(1-t)^2  &\quad      t \in [\frac{1}{2},1),\\
		\frac{1}{4} &\quad t \in [1,\infty),\\
	\end{cases} \quad \text{and} \quad \psi_{m,k}(t) = 2^{-3m/2}\psi_{0,0}(2^mt - k)
\end{equation}
for $m \in \bN_0$ and $0 \le k \le 2^m-1$. 

\begin{remark}\label{rem FFS basis} One easily verifies  that $\int_{0}^{t} e_{m,k}(s)\, ds=\psi_{m,k}(t)$ for $m\ge-1$, $0\le k\le (2^m-1)\vee0$, and $t\in\bR$. Therefore, the fact that the Faber--Schauder functions are a Schauder basis of $C[0,1]$ implies that the collection of all functions $\psi_{m,k}$ augmented with $\{1,t\}$ forms a Schauder basis for the Banach space $C^1[0,1]$. We will refer to this basis as the \textit{integrated Faber--Schauder} basis.
\end{remark}

\begin{remark}\label{rem nondyadic FFS}
	It is also worthwhile to point out that our approach to constructing a Schauder basis for $C^1[0,1]$ can be naturally extended to the case of a non-dyadic grid. In this case, the Schauder basis can be obtained by integrating the generalized Schauder basis functions in \cite{DasErhan,ContDas21}. This particular basis will transfer the quadratic spline interpolation problem over a non-dyadic grid to a linear equation analogously. Error analysis can be conducted analogously but will be much more involved.
\end{remark}

The integrated Faber--Schauder basis is the cornerstone of the error analysis for the estimated Faber--Schauder coefficients and the approximating functions. To wit, representing the quadratic spline interpolation \eqref{eq inverse Faber} in terms of the integrated Faber--Schauder basis implies that the estimated Faber--Schauder coefficients $\vartheta^{(n)}_{m,k}$ and the initial value $\hat{f}_0$ are the solutions to 

\begin{equation}\label{eq psi problem}
	\hat F_n(t) = F(0) + \hat{f}_0 \cdot t + \vartheta^{(n)}_{-1,0} \psi_{-1,0}(t) + \sum_{m = 0}^{n}\sum_{k = 0}^{2^m-1}\vartheta^{(n)}_{m,k} \psi_{m,k}(t), \qquad \text{for} \quad t \in \bT_{n+1}. 
\end{equation}

As shown in \Cref{thm represent}, the initial value $\hat f_0$ and estimated coefficients $\vartheta^{(n)}_{m,k}$ are not uniquely determined by \eqref{eq psi problem}, which stems from the following fact. It follows from \eqref{psi def eq} that for $n \in \bN_0$, 

\begin{equation}\label{eq linear wavelet}
	t = 2^{n/2 + 2} \sum_{k = 0}^{2^n-1} \psi_{n,k}(t) \qquad \text{for} \quad t \in \bT_{n+1}.
\end{equation} 
The right-hand side of \eqref{eq linear wavelet} consists of only the integrated Faber--Schauder functions at generation $n$. Thus, a change in the value of $\hat f_0$ to $\hat g_0$ can be completely absorbed by translating only the coefficients in generation $n$ by $2^{n/2+2}(\hat g_0 - \hat f_0)$. In particular, this change will not affect any of the values of coefficients $\vartheta^{(n)}_{m,k}$ for $m \le n-1$, while introducing a deviation of order $2^{n/2+2}|\hat g_0 - \hat f_0|$ into the final generation.  

In practice, a function can only be observed up to a certain finest scale, and attempting to approximate finer structures beyond this level may result in highly imprecise outcomes. This limitation is often known as \lq viscous cutoff' or \lq ultraviolet cutoff'. For instance, when approximating a function in terms of the Fourier basis, the number of basis functions has to be carefully chosen to achieve great precision. To wit, for a given threshold $n_0 \in \bN_0$, the trigonometric series 
\begin{equation}\label{eq DTFT}
	\wt F_{n,n_0}(t):= \sum_{m = -n_0}^{n_0} \hat{c}_m e^{2\pi i m t} \qquad \text{where} \quad \hat{c}_m := 2^{-n}\sum_{k = 0}^{2^{n}-1}F\left(\frac{k}{2^{n}}\right)e^{-\frac{2 \pi i m k}{2^{n}}},
\end{equation}
approximates the function $F$ based on discrete observations on $\bT_{n}$ and is referred to as the discrete Fourier series. However, the discrete Fourier coefficients $\wh c_m$ admit a periodic behaviour. Note that for any $j \in \bZ$, we have \begin{equation*}
	\wh c_{m + j2^n} = 2^{-n}\sum_{k = 0}^{2^{n}-1}F\left(\frac{k}{2^{n}}\right)e^{-\frac{2 \pi i (m + j2^n)k}{2^{n}}} = 2^{-n}\sum_{k = 0}^{2^{n}-1}F\left(\frac{k}{2^{n}}\right)e^{-\frac{2 \pi i m k}{2^{n}}}\cdot e^{-2\pi i j k} = \wh c_{m},
\end{equation*} 
where the last equality holds as $e^{-2\pi i j k} = 1$ for any $j,k \in \bZ$. As a result, on the one hand, if the threshold $n_0$ is too small, then the Fourier series $\wt F_{n,n_0}$ fails to reconstruct the very fine structure of function $F$. On the other hand,
if the threshold $n_0$ is too large, the Fourier series $\wt F_{n,n_0}$ yields a wildly oscillating approximation, even for a simple linear function $F(t) = t$ as shown in Figure \ref{fig FaberCutoff}. 

In practice, it is difficult to choose the number of bases properly. However, this is not the case for the integrated Faber--Schauder basis. First, the integrated Faber--Schauder basis arises from the task of quadratic spline interpolation over a dyadic mesh, and the number of basis functions naturally matches the number of observed data points as in our formulation \eqref{eq psi problem}. Second, the localized nature of the integrated Faber--Schauder basis makes them resistant to the \lq ultraviolet cutoff' behaviour, even if the number of basis functions mismatches the number of observations. To wit, let us consider an extended version of \eqref{eq linear wavelet}, where the number of basis functions is \lq overfitted'. For $n_0 \ge n$, suppose that there are coefficients $\wh \theta_{m,k}$ for $ n \le m \le n_0$ such that 
\begin{equation}\label{eq higher linear condition}
	2^{-n/2-2}\wh\theta_{n,\lfloor \frac{k}{2} \rfloor} + \sum_{m = n+1}^{n_0}\sum_{j = 2^{m-n-1}k}^{2^{m-n-1}(k+1)-1}2^{n-3m/2-2}\wh\theta_{m,k} = 1 \qquad \text{for} \quad 0 \le k \le 2^{n+1}-1,
\end{equation} then we still have 
\begin{equation}\label{eq higher linear}
	t = \sum_{m = n}^{n_0}\sum_{k = 0}^{2^m-1} \wh\theta_{m,k}\psi_{m,k}(t) \qquad \text{for} \quad t \in \bT_{n+1}.
\end{equation}
Clearly, the above equation extends \eqref{eq linear wavelet} to the case where higher-generation wavelet functions are involved. Moreover, it is worthwhile to point out that we can always select those coefficients $\wh{\theta}_{m,k}$ for $n \le m \le n_0$ that satisfy \eqref{eq higher linear condition}. This is because when viewing \eqref{eq higher linear condition} as a linear system, the number of unknowns is larger or equal to the number of equations. Note that condition \eqref{eq higher linear condition} only involves $\wh \theta_{m,k}$ for $n \le m \le n_0$. As a result, a change in the value of the initial estimate $\hat f_0$ can be completely absorbed by varying coefficients from generation $n$ to $n_0$. In other words, estimated coefficients prior to generation $n-1$ remain unaffected by this change, as in the case where the number of basis functions matches the number of observations. Moreover, the following graphic illustration shows that including the higher-generation integrated Faber--Schauder basis still provides precise quadratic spline interpolations. Hence, we conclude that the cutoff phenomenon presented in the estimated Faber--Schauder coefficients is essentially distinct from the \lq ultraviolet cutoff' phenomenon.

\begin{figure}[H]

	\centering
	\includegraphics[width=7cm]{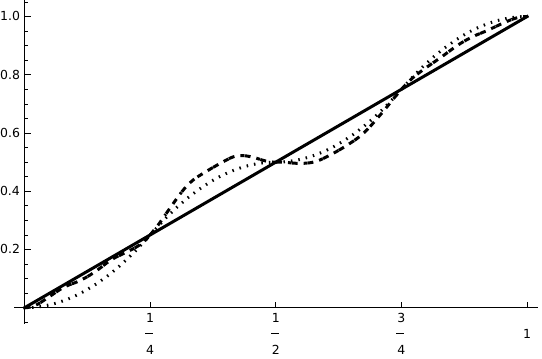}\qquad\qquad
	\includegraphics[width=7cm]{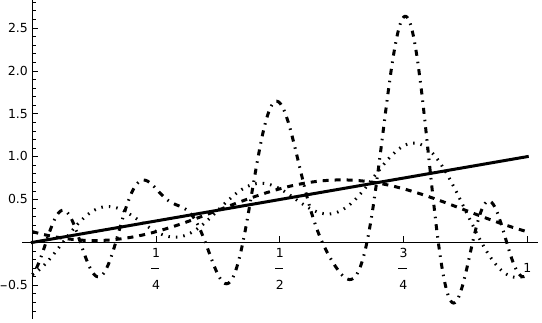}
		\caption{Graphical illustration of the persistence of integrated Faber--Schauder basis and instability of the Fourier basis; The left graph plots $F(t) = t$ (solid), $F_1(t) = 2^{n/2+2}\psi_{n,k}(t)$ (dotted) and $F_2(t) = 2^{n/2+2}\psi_{n,k}(t) + 4\sum_{m = n+2}^{n_0}\sum_{k = 0}^{2^{m-1}-1}(-1)^k\psi_{m,k}(t) + 2\sum_{m = n+2}^{n_0}\sum_{k = 2^{m-1}}^{2^{m}-1}(-1)^k\psi_{m,k}(t)$ (dashed) with $n = 1$ and $n_0 = 4$. Clearly, values of $F_1$ and $F_2$ coincide with those of $F$ over dyadic partition $\bT_{n+1}$, and both $F_1$ and $F_2$ yield well-behaved approximations; the right graph plots $F(t) = t$ (solid), $\wt F_{n,1}(t)$ (dashed), $\wt F_{n,3}(t)$ (dotted) and $\wt F_{n,7}(t)$ (dotdashed) with $n = 2$. Here, $\wt F_{n,7}$ yields a wildly oscillating approximation of $F$.}
		\label{fig FaberCutoff}
\end{figure}

\subsection{Matrix analysis of problem \eqref{eq inverse Faber}.}
Using the functions $\{\psi_{m,k}\}$, we can transform the problem \eqref{eq inverse Faber} into a standard linear equation. 
To this end, we assume again that $f(0)=0$ is known, that $\hat f_0=0$, and that $F(0)=0$; see the beginning of \Cref{Faber--Schauder section}.   We denote for $m \ge -1$ and ${k} \in \bN_0$,
\begin{equation}\label{eq Qm}
	Q_{(m,k)}:= \left[ 
	\begin{array}{c c c c} 
		\psi_{m,0}\Big(\frac{1}{2^{k}}\Big) & \psi_{m,1}\Big(\frac{1}{2^{k}}\Big) & \cdots & \psi_{m,2^{m \vee 0}-1}\Big(\frac{1}{2^{k}}\Big)\\ 
		\psi_{m,0}\Big(\frac{2}{2^{k}}\Big) & \psi_{m,1}\Big(\frac{2}{2^{k}}\Big) & \cdots & \psi_{m,2^{m \vee 0}-1}\Big(\frac{2}{2^{k}}\Big)\\ 
		\vdots & \vdots & \vdots & \vdots\\
		\psi_{m,0}(1) & \psi_{m,1}(1) & \cdots & \psi_{m,2^{m \vee 0}-1}(1)
	\end{array}  \\ 
	\right] \in \bR^{2^{k} \times 2^{m \vee 0}}.
\end{equation}
Next, for $m, {k} \in \bN_0$, we write
\begin{equation}\label{eq_Sigma_def}
	\Psi_{(m,{k})}	= \Big[Q_{(-1,{k})}, \cdots, Q_{(m,{k})}
	\Big] \in  \bR^{2^{{k}} \times 2^{m+1}} .
\end{equation}
\Cref{rem FFS basis} implies that $\Psi_{(m,{k})}$ has full rank. In particular, $\Psi_{n+1}:= \Psi_{(n,n+1)}\in  \bR^{2^{n+1} \times 2^{n+1}}$ is an invertible square matrix. We also use the shorthand notation 
\begin{equation*}
	\bm y_{n+1} := \left(F\Big(\frac{1}{2^{n+1}}\Big), F\Big(\frac{2}{2^{n+1}}\Big),\cdots, F(1)\right)^\top \in \bR^{2^{n+1}}
\end{equation*}
for the values of $F$ over $\bT_{n+1}$, our given data points.
Using our assumptions $F(0)=0$ and $\hat f_0=0$, 
the problem \eqref{eq inverse Faber}, or equivalently, the problem \eqref{eq psi problem} then becomes finding the coefficients $\vartheta^{(n)}_{m,k}$ for $-1 \le m \le n$ and $0 \le k \le 2^{m \vee 0}-1$ that solves
\begin{equation*}
	\hat F_n(t) =  \vartheta^{(n)}_{-1,0} \psi_{-1,0}(t) + \sum_{m = 0}^{n}\sum_{k = 0}^{2^m-1}\vartheta^{(n)}_{m,k} \psi_{m,k}(t), \qquad \text{for} \quad t \in \bT_{n+1}. 
\end{equation*}

The above problem is clearly equivalent to the $2^{n+1}$-dimensional  linear equation
\begin{equation}\label{eq inverse Faber lin eq}
	\Psi_{n+1}\bm \vartheta^{(n)}= \bm y_{n+1},
\end{equation}
in which the solution $ \bm \vartheta^{(n)}$ describes the Faber--Schauder coefficients of $\hat f_n$ or, equivalently, the coefficients in the expansion of $\hat F_n$ in terms of the  basis functions $\psi_{m,k}$.

It is a natural guess that the linear system \eqref{eq inverse Faber lin eq}
can be solved by a simple numerical inversion of the corresponding matrix $\Psi_{n+1}$. In practice, however, the numerical inversion of the matrix $\Psi_n$ becomes extremely unstable as the number of observation points increases; see \Cref{det table}. The main reason for this instability is the fact that matrix columns corresponding to neighbouring wavelet functions become asymptotically co-linear. Fortunately, it is possible to algebraically invert the matrix $\Psi_n$. This inversion formula gives rise to an explicit solution of  \eqref{eq inverse Faber lin eq}, which was stated in \Cref{thm represent}. The insight gained into the structure of $\Psi_n$ by deriving our inversion formula also forms the basis for our analysis of the approximation error. Our first lemma relates the two vectors $\bm \vartheta^{(n)}-\bm\theta_n$ and  $\bm z_{n+1}$, where $\bm z_{n+1}$ was defined in \eqref{eq def zn}, and   $\bm \theta_n$ denotes as before the true Faber--Schauder coefficients of $f$ up to and including generation $n$.

\begin{table}[H]
	\begin{center}
		\begin{tabular}{|c|c|c|c|c|c|}
			\hline
			\multicolumn{6}{|c|}{Numerical values for  $\log_{10}\det(\Psi_{n})$ for several values of $n$}\\
			\hline
			$n$ &\hspace*{0.3cm}2\hspace*{0.3cm}  &\hspace*{0.3cm}3\hspace*{0.3cm}  &\hspace*{0.3cm}4\hspace*{0.3cm}  &\hspace*{0.3cm}5\hspace*{0.3cm}  &\hspace*{0.3cm}6\hspace*{0.3cm}  \\
			\hline $\log_{10}\det(\Psi_{n})$&\hspace*{0.3cm}-4.97\hspace*{0.3cm}  &\hspace*{0.3cm}-13.4\hspace*{0.3cm}  &\hspace*{0.3cm}-33.9\hspace*{0.3cm}  &\hspace*{0.3cm}-82.03\hspace*{0.3cm}  &\hspace*{0.3cm}-192.81\hspace*{0.3cm} \\ \hline
		\end{tabular}
	\end{center}
		\caption{Numerical values for  $\log_{10}\det(\Psi_{n})$ for several values of $n$. Already for small values of $n$, the determinant of $\Psi_{n}$ is extremely small, which makes the problem \eqref{eq inverse Faber} ill-posed from a numerical point of view.}\label{det table}
\end{table}

\begin{lemma}\label{lemma zn}
	For $n \in \bN_0$, we set $P_n:= \Psi_{n+1}^{-1}Q_{(n+1,n+1)}$. Then $\bm \vartheta^{(n)} - \bm\theta_n = P_n \bm z_{n+1}$.
\end{lemma}

\begin{proof}
\Cref{rem FFS basis} implies that $\bm y_{n+1} = \lim_{m \ua \infty}\Psi_{(m,n+1)} \bm \theta_{m}$. Thus, we have 
\begin{equation}\label{eq mu decomposition}
	\begin{split}
		\bm \vartheta^{(n)} &= \Psi^{-1}_{(n,n+1)} \bm y_{n+1} = \Psi^{-1}_{(n,n+1)} \lim_{m \ua \infty}\Psi_{(m,n+1)} \bm \theta_m 
		\\&= \lim_{m \ua \infty} \Psi_{(n,n+1)}^{-1} \Big[\Psi_{(n,n+1)}, \,Q_{(n+1,n+1)},\,\cdots,\,Q_{(m,n+1)}\Big]\cdot \Big[\bm\theta^\top_{n},\bar{\bm\theta}^\top_{n+1},\cdots,\bar{\bm\theta}^\top_{m}\Big]^\top\\ &= \bm \theta_{n} + \Psi_{(n,n+1)}^{-1}\sum_{m = n+1}^{\infty}Q_{(m,n+1)}\bar{\bm \theta}_m,
	\end{split}
\end{equation}
where the third equality follows from \eqref{eq_Sigma_def}. For $n \in \bN_0$ and $m \ge n+1$, we define the $2^{n+1} \times 2^m$ dimensional matrix
\begin{equation*}
	U_{(m,n+1)} := 2^{(3n+3-3m)/2}\cdot\left[
	\begin{array}{c c c c  }
		1\,1\,\cdots\,1 & 0\,0\,\cdots\,0 & \cdots\cdots & 0\,0\,\cdots\,0 \\
		0\,0\,\cdots\,0 & 1\,1\,\cdots\,1 & \cdots\cdots & 0\,0\,\cdots\,0 \\
		\vdots & \vdots & \vdots &\vdots \\
		\underbrace{0\,0\,\cdots\,0}_{2^{m-n-1} \text{ times}} & \underbrace{0\,0\,\cdots\,0}_{2^{m-n-1} \text{ times}} & \cdots\cdots & \underbrace{1\,1\,\cdots\,1}_{2^{m-n-1} \text{ times}} \\
	\end{array} 
	\right],
\end{equation*}
and by a simple matrix calculation, one can easily verify that
\begin{equation}\label{eq z U}
	\bm z_{n+1} = \sum_{m = n+1}^{\infty} U_{(m,n+1)}\bar{\bm \theta}_m.
\end{equation}
Moreover, for $m \ge n+1$, we have $\psi_{m,k}(2^{-(n+1)}j) = 2^{-3m/2}\psi_{0,0}(2^{m-n-1}j - k) = 0$ for $k \le 2^{m-n-1}j$ and otherwise $2^{-3m/2-2}$. Thus, for $m \ge n+1$, the matrix $Q_{(m,n+1)}$ is of the following form: 
\begin{equation}\label{eq Q m n}
	Q_{(m,n+1)} = 2^{-3m/2-2}\cdot\left[
	\begin{array}{c c c c  }
		1\,1\,\cdots\,1 & 0\,0\,\cdots\,0 & \cdots\cdots & 0\,0\,\cdots\,0 \\
		1\,1\,\cdots\,1 & 1\,1\,\cdots\,1 & \cdots\cdots & 0\,0\,\cdots\,0 \\
		\vdots & \vdots & \vdots &\vdots \\
		\underbrace{1\,1\,\cdots\,1}_{2^{m-n-1} \text{ times}} & \underbrace{1\,1\,\cdots\,1}_{2^{m-n-1} \text{ times}} & \cdots\cdots & \underbrace{1\,1\,\cdots\,1}_{2^{m-n-1} \text{ times}} \\
	\end{array} 
	\right].
\end{equation}
In particular, the matrix $Q_{(n+1,n+1)}$ is a lower triangular matrix that every non-zero entries are equal to one after proper scaling. Therefore, we get $Q_{(m,n+1)} = Q_{(n+1,n+1)} \cdot U_{(m,n+1)}$, this is because $Q_{(n+1,n+1)}$ can be regarded as a row operation acting on the matrix $U_{(m,n+1)}$ as a partial sum of rows. Applying this identity and \eqref{eq z U} to \eqref{eq mu decomposition} gives 
$$\bm \vartheta^{(n)} - \bm \theta_n =  \Psi_{(n,n+1)}^{-1}\sum_{m = n+1}^{\infty}Q_{(m,n+1)}\bar{\bm \theta}_m = P_n\sum_{m = n+1}^{\infty}U_{(m,n+1)}\bar{\bm \theta}_m = P_n \bm z_{n+1}. $$
This completes the proof.
\end{proof}

Next, for any $-1 \le m \le n$, the vectors $\bm \vartheta^{(n)}_{m}$ and $\bar{\bm \vartheta}^{(n)}_m$ are subvectors of $\bm \vartheta^{(n)}$, and we can represent them via simple linear transformations of $\bm \vartheta^{(n)}$. To this end, we denote for $-1 \le m \le n$, 
\begin{equation*}
	R^{(n)}_m:= \left[\bm I_{2^{m+1}\times2^{m+1}},\bm 0_{(2^{n+1} - 2^{m+1})\times2^{m+1}}\right] \quad \text{and} \quad \bar{R}^{(n)}_m:= \left[\bm 0_{(2^{n+1} - 2^{m \vee 0})\times2^{m \vee 0}}, \bm I_{2^{m\vee0}\times2^{m\vee0}},\bm 0_{(2^{n+1} - 2^{m+1})\times2^{m\vee 0}}\right],
\end{equation*}
where $\bm I_{m \times m}$ denotes the $m \times m$ dimensional identity matrix, and $\bm 0_{m \times k}$ denotes the $m \times k$ dimensional zero matrix. Thus, $R^{(n)}_m$ is a $2^{m+1} \times 2^{n+1}$ dimensional matrix, and $R^{(n)}_m$ is a $2^{m\vee0} \times 2^{n+1}$ dimensional matrix. Furthermore,  one sees by means of a simple matrix manipulation  that 
\begin{equation*}
	\bm \vartheta^{(n)}_m - \bm \theta_m= R^{(n)}_mP_n \bm z_{n+1} \quad \text{and} \quad \bar{\bm \vartheta}^{(n)}_m - \bar{\bm \theta}_m = \bar{R}^{(n)}_mP_n \bm z_{n+1}.
\end{equation*}
It follows that for $p\in[1,\infty]$,
\begin{equation}\label{operator norm upper bound}
	\bnorm{\bm \vartheta^{(n)}_m - \bm \theta_m}_{\ell_p} \le \norm{R^{(n)}_mP_n}_p \norm{\bm z_{n+1}}_{\ell_p},
\end{equation}
where $\norm{\cdot}_p$ denotes the $\ell_p$-induced operator norm for a matrix. Moreover, 
and for any $\varepsilon > 0$, we can choose a vector {$\bm z^\ast_{n+1} = \left(z^{(n+1,\ast)}_{1}, \cdots, z^{(n+1,\ast)}_{2^n} \right)^{\top} \in \bR^{2^{n+1}}$} such that 
\begin{equation}\label{operator norm lower bound}
	\bnorm{\bm \vartheta^{(n)}_m - \bm \theta_m}_{\ell_p} \ge (1-\eps)\norm{R^{(n)}_mP_n}_p \norm{{\bm z^\ast_{n+1}}}_{\ell_p}.
\end{equation}
Indeed, for any given vector ${\bm z^\ast_{n+1}}$, one can find Faber--Schauder coefficients $\theta_{m,k}$ that yield back $\bm z_{n+1}$ via \eqref{eq def zn}; for instance we can take $\theta_{m,k}$ arbitrary for $m\le n$,  $\theta_{m,k}=0$ for $m\ge n+2$, and {$\theta_{{n+1},k}=z^{(n+1,\ast)}_{k+1}$} for $k=0,\dots, 2^{n+1}-1$. {In particular, a function $f$ admitting the above Faber--Schauder coefficients is piecewise continuous over the supporting grid $\bT_{n+2}$, and its antiderivative $F$ is well-defined.} Hence, 
Thus, the Theorems \ref{thm Faber errors} through \ref{thm Faber errors ellinf ellp} will follow immediately from corresponding bounds on the matrix norm $\norm{R^{(n)}_mP_n}_p$.

For the cases $p=1$ and $p=\infty$, the following result states exact expressions for this matrix norm. For $p=2$, we have in part upper and lower bounds that are slightly different from each other.

\begin{theorem}\label{thm V norm} 	For $n \in \bN_0$, $-1 \le m \le n$, and $p\in\{1,\infty\}$,
	$\norm{\bar{R}^{(n)}_mP_n}_{p} = \bar c_{p,m} $ and $ \norm{R^{(n)}_mP_n}_{p} = c_{p,m}$, 	where the constants $\bar c_{p,m}$ and $c_{p,m}$ are as in \Cref{thm Faber errors ellinf ellp}. For $p=2$, we have  
	$\norm{\bar{R}^{(n)}_mP_n}_{2} =  \norm{R^{(n)}_mP_n}_{1} = \frac{1}{2}$
	if $-1 \le m \le n-1$, and for  $m = n$ we have 
	\begin{equation*}
		\sqrt{2\left(1 - \cos\frac{\pi}{2^{n+1}}\right)^{-1} - \frac{3}{4}} \le \bnorm{\bar{R}^{(n)}_nP_n}_{2} = {\bnorm{{{R}^{(n)}_nP_n}}_{2}} \le \sqrt{2\left(1 - \cos\frac{\pi}{2^{n+1}}\right)^{-1} + \frac{3}{4}}.
	\end{equation*}
	
\end{theorem}

To prove results in \Cref{thm V norm}, we first need an explicit representation of the matrix $P_n$. This is the content of the next lemma, whose proof will in turn require another lemma. Before stating it, though, we need to introduce some notation. For $m \ge 0$, we denote the column Rademacher vectors $\bm r^+_m, \bm r^-_m \in \bR^{2^m}$ by
\begin{equation}\label{eq Rademacher}
	\bm r^{+}_{m} = (\underbrace{-1,+1,-1,+1, \cdots, -1,+1}_{2^m \,\, \text{times}})^\top \quad \text{and} \quad \bm r^{-}_{m} = (\underbrace{+1,-1,+1,-1, \cdots, +1,-1}_{2^m \,\, \text{times}})^\top.
\end{equation}
For given $n\in\bN_0$ and $0\le m\le n-1$, we define the row vectors 
\begin{equation*}
	{\bm\eta}_{i,j}^{(m)} = \begin{cases}
		2^{\frac{m-n-3}{2}}\left((\bm r^+_{n-m})^\top,\, (\bm r^-_{n-m})^\top\right) &\quad \text{for} \quad 1 \le i =j \le 2^m,\\
		\bm 0_{1 \times 2^{n+1-m}} &\quad \text{for} \quad 1 \le i \neq j \le 2^m.
	\end{cases}
\end{equation*}
We also define the {row} vectors $\bm v = (\frac{3}{2\sqrt{2}},\frac{-1}{2\sqrt{2}})$ and $\bm u = (\sqrt{2},-\sqrt{2})$. Then we define $\cC_{-1},\cC_0,\dots,\cC_n$ as follows. For $0 \le m \le n-1$, we let
\begin{equation}\label{eq Cm}
	\cC_{m}:= \left[
	\begin{array}{c c c c c}
		{\bm \eta}^{(m)}_{1,1} & {\bm \eta}^{(m)}_{1,2} &\cdots  & {\bm \eta}^{(m)}_{1,2^m-1}  & {\bm \eta}^{(m)}_{1,2^m}\\
		{\bm \eta}^{(m)}_{2,1} & {\bm \eta}^{(m)}_{2,2} &\cdots  & {\bm \eta}^{(m)}_{2,2^m-1}  & {\bm \eta}^{(m)}_{2,2^m}\\
		\vdots & \vdots & \ddots& \vdots & \vdots \\
		{\bm \eta}^{(m)}_{2^m,1} & {\bm \eta}^{(m)}_{2^m,2} &\cdots  & {\bm \eta}^{(m)}_{2^m,2^m-1}  & {\bm \eta}^{(m)}_{2^m,2^m}\\
	\end{array}
	\right].
\end{equation}
For $m=-1$ and $m=n$, we let 
\begin{equation*}
	\cC_{-1} = 2^{-\frac{n+3}{2}} (\bm r^{+}_{n+1})^\top\qquad\text{and}\qquad 		\cC_{n} = \underbrace{\left[\begin{array}{c c c c c}
			\bm v & \bm{0}_{1\times 2} &\cdots  & \bm{0}_{1\times 2}  & \bm{0}_{1\times 2}\\
			\bm u & \bm v &\cdots  & \bm{0}_{1\times 2}  & \bm{0}_{1\times 2}\\
			\vdots & \vdots & \ddots& \vdots & \vdots \\
			\bm u & \bm u &\cdots  & \bm v  & \bm{0}_{1\times 2}\\
			\bm u & \bm u &\cdots  & \bm u  & \bm v\\
		\end{array}\right]}_{2^n \text{ times}},
\end{equation*}
so that $\cC_m\in\bR^{2^{m \vee 0}\times 2^{n+1}}$ for $m=-1,\dots, n$.

\begin{lemma}\label{Lemma Pn} 
	The matrix $P_n$ is of the following form, 
	\begin{equation}\label{eq Pn}
		P_n = \left((\cC_{-1})^\top, (\cC_{0})^\top, \cdots, (\cC_{n})^\top\right)^\top.
	\end{equation}
\end{lemma}

To prove \Cref{Lemma Pn}, the following lemma is needed. For $n \in \bN_0$ and $-1 \le m \le n-1$, we define the column vector $\bm w^+_{(m,n)} \in \bR^{2^{n-m}}$ and its reverse vector $\bm w^-_{(m,n)}\in\bR^{2^{n-m}}$,
\begin{equation}\label{eq bmw}
	\bm w^+_{(m,n)} = \left(\frac{2k-1}{2^{(n-1-m \vee 0)/2}}\right)_{k = 1,2,\cdots, 2^{n-m}} \quad \text{and} \quad \bm w^-_{(m,n)} = \left(\frac{2k-1}{2^{(n-1-m \vee 0)/2}}\right)_{k = 2^{n-m},\cdots,2,1}.
\end{equation}

\begin{lemma}\label{Lemma An}
	For given $n \in \bN_0$ and $0 \le m \le n$, we define the column vectors 
	\begin{equation*}
		\bm \alpha^{(m)}_{i,j} = \begin{cases}
			\big((\bm w^+_{(m,n)})^\top,(\bm w^-_{(m,n)})^\top\big)^\top &\quad \text{for} \quad 1 \le i = j \le 2^m,\\
			\bm 0_{2^{n-m+1}\times 1} &\quad \text{for} \quad 1 \le i \neq j \le 2^m.
		\end{cases}
	\end{equation*}
	Then we define $\cA_{-1}, \cA_{0}, \cdots, \cA_{n}$ as follows: For $0 \le m \le n$, we let 
	\begin{equation}\label{eq Am}
		\cA_{m}:= \left[
		\begin{array}{c c c c c}
			{\bm \alpha}^{(m)}_{1,1} & {\bm \alpha}^{(m)}_{1,2} &\cdots  & {\bm \alpha}^{(m)}_{1,2^m-1}  & {\bm \alpha}^{(m)}_{1,2^m}\\
			{\bm \alpha}^{(m)}_{2,1} & {\bm \alpha}^{(m)}_{2,2} &\cdots  & {\bm \alpha}^{(m)}_{2,2^m-1}  & {\bm \alpha}^{(m)}_{2,2^m}\\
			\vdots & \vdots & \ddots& \vdots & \vdots \\
			{\bm \alpha}^{(m)}_{2^m,1} & {\bm \alpha}^{(m)}_{2^m,2} &\cdots  & {\bm \alpha}^{(m)}_{2^m,2^m-1}  & {\bm \alpha}^{(m)}_{2^m,2^m}\\
		\end{array}
		\right] \in \bR^{2^{n+1}\times2^m}.
	\end{equation}
	For $m = -1$, we set $\cA_{-1} = \bm \omega^{+}_{(-1,n)}$, so that $\cA_m\in\bR^{2^{m \vee 0}\times 2^{n+1}}$ for $m=-1,\dots, n$. Then the matrix $A_n:= Q^{-1}_{(n+1,n+1)}\cdot\Psi_{n+1}$ is well-defined and of the following form
	\begin{equation*}
		A_n = \big(\cA_{-1}, \cA_{0}, \cdots, \cA_{n}\big).
	\end{equation*}
\end{lemma}

\begin{proof}
	The proof has two steps. First, we claim and show that $Q_{(n+1,n+1)}$ is invertible and
	\begin{equation}\label{eq Q inverse}
		Q^{-1}_{(n+1,n+1)} = 2^{\frac{3(n+1)}{2}+2}\left[\begin{array}{c c c c c }
			1 & 0 &\cdots &0 &0 \\ 
			-1 & 1  &\cdots &0 &0 \\ 
			0 & -1  &\cdots &0 &0 \\ 
			\vdots &  \vdots & \ddots & \vdots & \vdots \\
			0 & 0  &\cdots & -1 & 1 \\
		\end{array}\right].
	\end{equation}
	Taking $m = n+1$ in \eqref{eq Q m n} implies that $Q_{(n+1,n+1)}$ is of the form
	\begin{equation*}
		Q_{(n+1,n+1)} = 2^{-\frac{3(n+1)}{2}-2}\left[\begin{array}{c c c c c }
			1 & 0 &\cdots &0 &0 \\ 
			1 & 1  &\cdots &0 &0 \\ 
			\vdots &  \vdots & \ddots & \vdots & \vdots \\
			1 & 1  &\cdots &1 &0 \\ 
			1 & 1  &\cdots & 1 & 1 \\
		\end{array}\right],
	\end{equation*}
	and thus, we only need to show that 
	\begin{equation*}
		\left[\begin{array}{c c c c c }
			1 & 0 &\cdots &0 &0 \\ 
			-1 & 1  &\cdots &0 &0 \\ 
			0 & -1  &\cdots &0 &0 \\ 
			\vdots &  \vdots & \ddots & \vdots & \vdots \\
			0 & 0  &\cdots & -1 & 1 \\
		\end{array}\right]\left[\begin{array}{c c c c c }
			1 & 0 &\cdots &0 &0 \\ 
			1 & 1  &\cdots &0 &0 \\ 
			\vdots &  \vdots & \ddots & \vdots & \vdots \\
			1 & 1  &\cdots &1 &0 \\ 
			1 & 1  &\cdots & 1 & 1 \\
		\end{array}\right] = \bm I_{2^{n+1}\times 2^{n+1}},
	\end{equation*}
	The above identity indeed holds for the following reason: The first matrix can be regarded as a row operation acting on the subsequent lower triangular matrix by subtracting preceding rows. To be more specific, for $n \in \bN_0$, the $n^\text{th}$ row of the second matrix consists of ones at its first $n$ positions and zeros at the remaining position. The difference of the $n^\text{th}$ and the $(n-1)^\text{th}$ row then only has a one at the $n^\text{th}$ position and zeros elsewhere. This then leads to 
	\begin{equation*}
		A_n = Q^{-1}_{(n+1,n+1)}\Psi_{(n,n+1)} = \Big(Q^{-1}_{(n+1,n+1)}Q_{(-1,n+1)},Q^{-1}_{(n+1,n+1)}Q_{(0,n+1)}, \cdots, Q^{-1}_{(n+1,n+1)}Q_{(n,n+1)}\Big).
	\end{equation*}
	Again, the matrix $Q^{-1}_{(n+1,n+1)}$ acts on each matrix $Q_{(m,n+1)}$ by subtracting preceding rows after proper scaling. For instance, the first column of the matrix $Q^{-1}_{(n+1,n+1)}Q_{(m,n+1)}$ becomes 
	\begin{equation}\label{eq first row}
		2^{3(n+1)/2+2}\left(\psi_{m,0}\Big(\frac{1}{2^{n+1}}\Big),\psi_{m,0}\Big(\frac{2}{2^{n+1}}\Big)-\psi_{m,0}\Big(\frac{1}{2^{n+1}}\Big),\cdots,\psi_{m,0}(1)-\psi_{m,0}\Big(\frac{2^{n+1}-1}{2^{n+1}}\Big)\right)^\top.
	\end{equation}
	Note that 
	\begin{equation}\label{eq psi difference}
		\begin{split}
			\psi_{m,0}\Big(\frac{k}{2^{n+1}}\Big)&-\psi_{m,0}\Big(\frac{k-1}{2^{n+1}}\Big) = 2^{-3m/2}\left(\psi_{0,0}\Big(\frac{k}{2^{n-m+1}}\Big)-\psi_{0,0}\Big(\frac{k-1}{2^{n-m+1}}\Big)\right)
			\\&=\begin{cases}
				2^{m/2-2n-3}(2k-1) &\quad \text{for $1 \le k \le 2^{n-m}$,}\\
				2^{m/2-2n-3}\left(2^{n-m+2}-(2k-1)\right) &\quad \text{for $2^{n-m}+1 \le k \le 2^{n-m+1}$,}\\
				0	&\quad \text{otherwise.}
			\end{cases}
		\end{split}
	\end{equation}
	Hence, \eqref{eq first row} and \eqref{eq psi difference} demonstrate that the first column fits into the form \eqref{eq Am}. Furthermore, by translation, this assertion carries over to all columns of the matrix $Q^{-1}_{(n+1,n+1)}Q_{(m,n+1)}$ for $m \ge 0$. For the case $m = -1$, we have 
	\begin{equation*}
		Q^{-1}_{(n+1,n+1)}Q_{(-1,n+1)} = 2^{3(n+1)/2+1}\left(\Big(\frac{1}{2^{n+1}}\Big),\Big(\frac{2}{2^{n+1}}\Big)^2-\Big(\frac{1}{2^{n+1}}\Big)^2,\cdots,1-\Big(\frac{2^{n+1}-1}{2^{n+1}}\Big)^2\right)^\top.
	\end{equation*}
	Since for any $1 \le k \le 2^{n+1}$, we have 
	\begin{equation*}
		2^{3(n+1)/2+1}\left(\frac{(k-1)^2}{2^{2n+2}}-\frac{k^2}{2^{2n+2}} \right) = 2^{-n/2+1/2}(2k-1), 
	\end{equation*}
	thus, $Q^{-1}_{(n+1,n+1)}Q_{(-1,n+1)} = \bm \omega^{+}_{(-1,n)}$. This completes the proof.
\end{proof}

Now we can proceed to prove \Cref{Lemma Pn}.

\begin{proof}[Proof of \Cref{Lemma Pn}]
For $n \in \bN_0$, the matrix $A_n = Q^{-1}_{(n+1,n+1)}\cdot \Psi_{n+1}$ is of full-rank, since both matrices $Q_{(n+1,n+1)}$ and $\Psi_{n+1}$ are of full-rank. It then suffices to show that $L_n:=P_nA_n$ is equal to $\bm I_{2^{n+1}\times 2^{n+1}}$. Applying the representation of $P_n$ and $A_n$ in their corresponding forms in \Cref{Lemma Pn} and \Cref{Lemma An} yields
	\begin{equation*}
	\begin{split}
		L_n &=
		\left[\begin{array}{c c c c}
			\cC_{-1}\cA_{-1} & \cC_{-1}\cA_{0}  & \cdots & \cC_{-1}\cA_{n}\\ 
			\cC_{0}\cA_{-1} & \cC_{0}\cA_{0}  & \cdots & \cC_{0}\cA_{n}\\ 
			\vdots & \vdots & \ddots  & \vdots\\
			\cC_{n}\cA_{-1} & \cC_{n}\cA_{0} & \cdots & \cC_{n}\cA_{n}\\ 
		\end{array}\right] =:  \left[\begin{array}{c c c c }
			\cL_{(-1,-1)} & \cL_{(-1,0)}  & \cdots & \cL_{(-1,n)}\\ 
			\cL_{(0,-1)} & \cL_{(0,0)}  & \cdots & \cL_{(0,n)}\\ 
			\vdots & \vdots & \ddots  & \vdots\\
			\cL_{(n,-1)} & \cL_{(n,0)} & \cdots & \cL_{(n,n)}\\ 
		\end{array}\right],\end{split}
\end{equation*}
where $\cL_{(m,k)} := \cC_m\cA_k \in \bR^{2^{m\vee 0}\times2^{k\vee 0}}$ for $-1 \le m,k \le n$. Now, we only need to show that $\cL_{(m,k)} = \bm{0}_{2^{m\vee 0}\times2^{k\vee 0}}$ and $\cL_{(m,m)} = \bm I_{2^{m\vee 0} \times 2^{m\vee 0}}$ for $-1 \le m \neq k \le n$.

First, we show that that $\cL_{(m,m)} = \bm I_{2^{m} \times 2^{m}}$ for $0 \le m \le n-1$. To this end,
\begin{equation*}
	\begin{split}
		\cL_{(m,m)} &= \left[
		\begin{array}{c c c c}
			{\bm \eta}^{(m)}_{1,1} & {\bm \eta}^{(m)}_{1,2} &\cdots  &  {\bm \eta}^{(m)}_{1,2^m}\\
			{\bm \eta}^{(m)}_{2,1} & {\bm \eta}^{(m)}_{2,2} &\cdots  &  {\bm \eta}^{(m)}_{2,2^m}\\
			\vdots & \vdots & \ddots&  \vdots \\
			{\bm \eta}^{(m)}_{2^m,1} & {\bm \eta}^{(m)}_{2^m,2} &\cdots  &  {\bm \eta}^{(m)}_{2^m,2^m}\\
		\end{array}
		\right]\left[
		\begin{array}{c c c c }
			{\bm \alpha}^{(m)}_{1,1} & {\bm \alpha}^{(m)}_{1,2} &\cdots   & {\bm \alpha}^{(m)}_{1,2^m}\\
			{\bm \alpha}^{(m)}_{2,1} & {\bm \alpha}^{(m)}_{2,2} &\cdots    & {\bm \alpha}^{(m)}_{2,2^m}\\
			\vdots & \vdots & \ddots&  \vdots \\
			{\bm \alpha}^{(m)}_{2^m,1} & {\bm \alpha}^{(m)}_{2^m,2} &\cdots  & {\bm \alpha}^{(m)}_{2^m,2^m}\\
		\end{array}
		\right] =: \left[
		\begin{array}{c c c  c}
			\ell^{(m)}_{1,1} & \ell^{(m)}_{1,2} &\cdots    & \ell^{(m)}_{1,2^m}\\
			\ell^{(m)}_{2,1} & \ell^{(m)}_{2,2} &\cdots   & \ell^{(m)}_{2,2^m}\\
			\vdots & \vdots & \ddots&  \vdots \\
			\ell^{(m)}_{2^m,1} & \ell^{(m)}_{2^m,2} &\cdots    & \ell^{(m)}_{2^m,2^m}\\
		\end{array}
		\right],
	\end{split}
\end{equation*}
where $\ell^{(m)}_{i,j} = \sum_{k = 1}^{2^m} {\bm \eta}^{(m)}_{i,k}{\bm \alpha}^{(m)}_{k,j}$. Since ${\bm \eta}^{(m)}_{i,j} = \bm 0_{1 \times 2^{n+1-m}}$ and  ${\bm \alpha}^{(m)}_{i,j} = \bm 0_{2^{n+1-m} \times 1}$ for $i \neq j$, therefore $\ell^{(m)}_{i,j} = 0$ for $i \neq j$. Moreover, we get in the same way that
\begin{equation*}
	\begin{split}
		\ell^{(m)}_{i,i} &= {\bm \eta}^{(m)}_{i,i}{\bm \alpha}^{(m)}_{i,i} = 2^{(m-n-3)/2}\left((\bm r^+_{n-m})^\top \bm w^+_{(m,n)} + (\bm r^-_{n-m})^\top \bm w^-_{(m,n)}\right)\\&= 2^{m-n-1}\left(\sum_{k = 1}^{2^{n-m}}(-1)^k(2k-1) + \sum_{k = 2^{n-m+1}}^{2^{n-m+1}}(-1)^{k+1}\left(2^{n-m+2}-(2k-1)\right)\right)\\&= 2^{m-n}\sum_{k = 1}^{2^{n-m}}(-1)^{k}(2k-1) = 2^{m-n}\sum_{k = 1}^{2^{n-m-1}}2 = 1. 
	\end{split}
\end{equation*}
Next, for the case $m = -1$, we prove analogously that
\begin{equation*}
	\cC_{-1}\cA_{-1} = 2^{-\frac{n+3}{2}}(\bm r^+_{n+1})^\top \bm w^{+}_{(-1,n)} =2^{-n-1}\left(\sum_{k = 1}^{2^{n+1}}(-1)^k(2k-1) \right)= 1 .
\end{equation*}
Last, it follows that
\begin{equation*}
	\cC_{n}\cA_{n} = 2^{1/2}\left[\begin{array}{ccccccc}
		\frac{3}{2 \sqrt{2}} & -\frac{1}{2 \sqrt{2}} & 0 & 0 & \cdots & 0 & 0 \\
		\sqrt{2} & -\sqrt{2} & \frac{3}{2 \sqrt{2}} & -\frac{1}{2 \sqrt{2}}  & \cdots & 0 & 0 \\
		\vdots & \vdots & \vdots & \vdots & \ddots& \vdots & \vdots \\
		\sqrt{2} & -\sqrt{2} & \sqrt{2} & -\sqrt{2}  & \cdots& \frac{3}{2 \sqrt{2}} & -\frac{1}{2 \sqrt{2}} \\
	\end{array}\right]\left[\begin{array}{cccc}
		1 & 0 & \cdots & 0\\
		1 & 0 & \cdots & 0\\
		0 & 1 & \cdots & 0\\
		0 & 1 & \cdots & 0\\
		\vdots & \vdots & \vdots & \vdots\\
		0 & 0 & \cdots & 1\\
		0 & 0 & \cdots & 1\\
	\end{array}\right] = \bm{I}_{2^n \times 2^n}.
\end{equation*}
The above identity holds, because the second matrix can be regarded as a column operation acting on the first matrix by adding subsequent columns. To be more specific, for $1 \le m \le 2^n$, the $m^\text{th}$ column of the matrix $\cL_{(n,n)}$ equals the sum of the $(2m-1)^\text{th}$ and the $(2m)^\text{th}$ columns of $\cC_n$. Prior to the $m^\text{th}$ entry, each entry of the $(2m-1)^\text{th}$ and the $(2m)^\text{th}$ column of $\cC_n$ equals zero and all entries are equal to $\pm\sqrt{2}$ afterwards respectively. Thus, it is clear that $\cC_{n}\cA_{n} = \bm I_{2^n \times 2^n}$.

Next, we are going to show that $\cL_{(m,k)} = \bm 0_{2^m \times 2^k}$ for $0 \le m \neq k \le n-1$, and we first show this holds with $m < k$. Under this assumption, we have 
\begin{equation*}
	\begin{split}
		\cC_{m}\cA_{k} 
		&= \left[
		\begin{array}{c c c c}
			{\bm \eta}^{(m)}_{1,1} & \bm{0}_{1\times 2^{n+1-m}} &\cdots  &  \bm{0}_{1\times 2^{n+1-m}}\\
			\bm{0}_{1\times 2^{n+1-m}} & {\bm \eta}^{(m)}_{2,2} &\cdots  & \bm{0}_{1\times 2^{n+1-m}}\\
			\vdots & \vdots & \ddots &  \vdots \\
			\bm{0}_{1\times 2^{n+1-m}} & \bm{0}_{1\times 2^{n+1-m}} &\cdots  &  {\bm \eta}^{(m)}_{2^m,2^m}\\
		\end{array}
		\right]\left[\begin{array}{c c c c}
			{\bm \alpha}^{(k)}_{1,1} & \bm{0}_{2^{n+1-k}\times1 } &\cdots    & \bm{0}_{2^{n+1-k}\times1 }\\
			\bm{0}_{2^{n+1-k}\times1 } & {\bm \alpha}^{(k)}_{2,2} &\cdots    & \bm{0}_{2^{n+1-k}\times1 }\\
			\vdots & \vdots & \ddots& \vdots \\
			\bm{0}_{2^{n+1-k}\times1 } & \bm{0}_{2^{n+1-k}\times1 } &\cdots   & {\bm \alpha}^{(k)}_{2^k,2^k}\\
		\end{array}\right]\\&= \left[\begin{array}{c c c c}
			{\bm \eta}^{(m)}_{1,1} {\bm \alpha}^{(m,k)}_{1,1}  & \bm{0}_{1\times 2^{k-m}} &\cdots    & \bm{0}_{1\times 2^{k-m}}\\
			\bm{0}_{1\times 2^{k-m}} & {\bm \eta}^{(m)}_{2,2}{\bm \alpha}^{(m,k)}_{2,2}  &\cdots    & \bm{0}_{1\times 2^{k-m}}\\
			\vdots & \vdots & \ddots& \vdots \\
			\bm{0}_{1\times 2^{k-m}} & \bm{0}_{1\times 2^{k-m}} &\cdots   & {\bm \eta}^{(m)}_{2^m,2^m}{\bm \alpha}^{(m,k)}_{2^m,2^m} \\
		\end{array}\right],
	\end{split}
\end{equation*}
where the $2^{n+m-1} \times 2^{k-m}$ dimensional matrix $
{\bm \alpha}^{(m,k)}_{i,i}$ is defined as 
\begin{equation*}
	{\bm \alpha}^{(m,k)}_{i,i} := \left[\begin{array}{c c c c}
		{\bm \alpha}^{(k)}_{1,1} & \bm{0}_{2^{n+1-k}\times1 } &\cdots    & \bm{0}_{2^{n+1-k}\times1 }\\
		\bm{0}_{2^{n+1-k}\times1 } & {\bm \alpha}^{(k)}_{2,2} &\cdots    & \bm{0}_{2^{n+1-k}\times1 }\\
		\vdots & \vdots & \ddots& \vdots \\
		\bm{0}_{2^{n+1-k}\times1 } & \bm{0}_{2^{n+1-k}\times1 } &\cdots   & {\bm \alpha}^{(k)}_{2^{k-m},2^{k-m}}\\
	\end{array}\right],
\end{equation*}
for any $1 \le i \le 2^m$. Further calculation yields that for each $1 \le i \le 2^m$,
\begin{equation*}
	\begin{split}
		&{\bm \eta}^{(m)}_{i,i}{\bm \alpha}^{(m,k)}_{i,i}=2^{(m-n-3)/2}\left(
		(\bm r^+_{n-m})^\top,
		(\bm r^-_{n-m})^\top
		\right)\left[\begin{array}{c c c c}
			{\bm \alpha}^{(k)}_{1,1} & \bm{0}_{2^{n-k+1}\times1 } &\cdots    & \bm{0}_{2^{n-k+1}\times1 }\\
			\bm{0}_{2^{n-k+1}\times1 } & {\bm \alpha}^{(k)}_{2,2} &\cdots    & \bm{0}_{2^{n-k+1}\times1 }\\
			\vdots & \vdots & \vdots& \vdots \\
			\bm{0}_{2^{n-k+1}\times1 } & \bm{0}_{2^{n-k+1}\times1 } &\cdots   & {\bm \alpha}^{(k)}_{2^{k-m},2^{k-m}}\\
		\end{array}\right]\\&= 2^{(m-n-3)/2}\big(\underbrace{(\bm r_{n-k}^+)^\top,\cdots,(\bm r_{n-k}^+)^\top}_{2^{k-m} \text{ times}},\,\, \underbrace{(\bm r_{n-k}^-)^\top,\cdots,(\bm r_{n-k}^-)^\top}_{2^{k-m} \text{ times}}\big)\left[\begin{array}{c c c c}
			\bm w^+_{(k,n)} & \bm{0}_{2^{n-k}\times1 } &\cdots    & \bm{0}_{2^{n-k}\times1 }\\
			\bm w^-_{(k,n)} & \bm{0}_{2^{n-k}\times1 } &\cdots    & \bm{0}_{2^{n-k}\times1 }\\
			\bm{0}_{2^{n-k}\times1 } & \bm w^+_{(k,n)} &\cdots    & \bm{0}_{2^{n-k}\times1 }\\
			\bm{0}_{2^{n-k}\times1 } & \bm w^-_{(k,n)} &\cdots    & \bm{0}_{2^{n-k}\times1 }\\
			\vdots & \vdots & \vdots& \vdots \\
			\bm{0}_{2^{n-k}\times1 } & \bm{0}_{2^{n-k}\times1 } &\cdots   & \bm w^+_{(k,n)}\\
			\bm{0}_{2^{n-k}\times1 } & \bm{0}_{2^{n-k}\times1 } &\cdots   & \bm w^-_{(k,n)}\\
		\end{array}\right]\\
		&= 2^{(m-n-3)/2}\big(\underbrace{(\bm r_{n-k}^+)^\top(\bm w^+_{(k,n)}+ \bm w^-_{(k,n)}),\cdots}_{2^{k-m-1} \text{ times}},\,\, \underbrace{(\bm r_{n-k}^-)^\top(\bm w^+_{(k,n)}+ \bm w^-_{(k,n)}),\cdots}_{2^{k-m-1} \text{ times}}\big).
	\end{split}
\end{equation*}
Since $\bm w^+_{(k,n)}+ \bm w^-_{(k,n)} = 2^{\frac{m-n+1}{2}}\big((2^{n-m+1}-1)+1\big)\cdot\bm 1_{2^{n-k}\times 1} = 2^{\frac{n-m+3}{2}}\cdot \bm1_{2^{n-k}\times 1}$, it follows that
\begin{equation}\label{eq rw prod}
	(\bm r_{n-k}^+)^\top(\bm w^+_{(k,n)}+ \bm w^-_{(k,n)}) = (\bm r_{n-k}^-)^\top(\bm w^+_{(k,n)}+ \bm w^-_{(k,n)}) = 0.
\end{equation} 
Thus, we have ${\bm \eta}^{(m)}_{i,i}{\bm \alpha}^{(m,k)}_{i,i} = \bm 0_{1 \times 2^{k-m}}$, and hence, $\cL_{(m,k)} = \bm 0_{2^m \times 2^k}$. Since the matrices $A_n$ and $P_n$ are symmetric by construction, we have $\cL_{(m,k)} = \cL_{(k,m)}^\top =  \bm 0_{2^k \times 2^m}$ for $m > k$. This now proves our claim $\cL_{(m,k)} = \bm 0_{2^m \times 2^k}$ for $0 \le m \neq k \le n-1$. 

For the case $m = -1$, we get in the same manner that for $-1 \le k \le n-1$,
\begin{equation*}
	\cC_{-1}\cA_k = 2^{-\frac{n+3}{2}}\big(\underbrace{(\bm r_{n-k}^+)^\top(\bm w^+_{(k,n)}+ \bm w^-_{(k,n)}),\cdots,(\bm r_{n-k}^+)^\top(\bm w^+_{(k,n)}+ \bm w^-_{(k,n)}}_{2^{k} \text{ times}}\big) = \bm 0_{1 \times 2^k}.
\end{equation*}

Finally, for $-1 \le m \le n-1$, recall that we denote the $(i,j)$-entry of $\cL_{(n,m)}$ by $\ell^{(n,m)}_{i,j}$. A simple matrix calculation implies that
\begin{equation*}
	\ell^{(n,m)}_{i,j} = (\underbrace{\bm u,\cdots, \bm u}_{i \,\, \text{times}} ,\, \bm v ,\,  \underbrace{\bm{0}_{1\times 2},\, \cdots,\,\bm{0}_{1\times 2}}_{2^{n} - i -1 \,\, \text{times}})\big(\underbrace{\bm 0_{1\times 2^{n-m}}, \cdots}_{2j \,\, \text{times}},  (\bm w^{+}_{(m,n)})^\top,(\bm w^{-}_{(m,n)})^\top, \cdots, \underbrace{\bm 0_{1\times 2^{n-m}}, \cdots}_{2^{m+1}-2j-2 \,\, \text{times}}  \big)^\top.
\end{equation*}
The fact that $$\underbrace{(\bm u,\cdots, \bm u)}_{2^{n-m-1} \,\, \text{times}}\bm w^{\pm}_{(m,n)} = \sqrt{2}\cdot  (\bm r^+_{n-m})^\top \bm w^{\pm}_{(m,n)} = 0.$$ directly yields that $\ell^{(n,m)}_{i,j} = 0$ for $i \ge 2^{n-m}(j+1)$. Otherwise, for $i \le 2^{n-m}j-1$, we clearly have $\ell_{i,j}^{(n,m)} = 0$, as every non-zero vector must be multiplied by a zero vector. For $2^{n-m}j \le i \le 2^{n-m-1}(2j+1)-1$, we denote $\tau = i-2^{n-m}j$. Then, a matrix calculation gives 
\begin{equation*}
	\begin{split}
		\ell^{(n,m)}_{i,j} &= \frac{\sqrt{2}}{2^{(n+3-m)/2}}\left(\sum_{k = 1}^{\tau}(-1)^k(2k-1) + \frac{3}{4}(2\tau+1)-\frac{1}{4}(2\tau +3) \right) \\&=  \frac{\sqrt{2}}{2^{2n-m/2-1}}(\tau - \frac{3}{4}(2\tau+1)-\frac{1}{4}(2\tau +3)) = 0.
	\end{split}
\end{equation*}
The case $2^{n-m-1}(2j+1)\le i \le 2^{n-m}(j+1)-1$ can be proved analogously using the symmetry property of $\bm{w}^{\pm}_{(m,n)}$ for $m \ge 0$. For the case $m = -1$, one can easily apply the same method to obtain $\cL_{(n,-1)} = \bm 0_{2^{n+1} \times 1}$. This completes the proof. 
\end{proof}

Next, we shall apply \Cref{Lemma Pn} to obtain the matrix norms of $\bar{R}^{(n)}_mP_n$ and $R^{(n)}_mP_n$ as stated in \Cref{thm V norm}. We give separate proofs for the cases $p=\infty$, $p=1$, and $p=2$. 
Recall that for an $m\times n$ matrix $A=(a_{ij})$, we have $\norm{A}_1 = \max_{1 \le j \le m}\sum_{i = 1}^{n}|a_{ij}|$, which is simply the maximum absolute column sum of the matrix, and $\norm{A}_\infty= \max_{1 \le i \le n}\sum_{j = 1}^{m}|a_{ij}|$ is the maximum absolute row sum of the matrix. For the special case $p = 2$, $\norm{A}_2 = \sigma_{\max}(A)$ is the largest singular value of the matrix. 

\begin{proof}[Proof of \Cref{thm V norm} for $p=\infty$]
	We start with computing the norm $\norm{\bar{R}_m^{(n)}P_n}_\infty$ for $-1 \le m \le n-1$. Note that $\bar{R}^{(n)}_m P_n = \cC_m$, and since ${\bm \eta}^{(m)}_{i,j}$ is a zero vector for $i \neq j$, then
	\begin{equation*}
		\begin{split}
			\norm{\cC_m}_{\infty} &= \max_{1 \le i \le 2^m}\Bnorm{\big({\bm \eta}^{(m)}_{i,1}, {\bm \eta}^{(m)}_{i,2} ,\cdots  , {\bm \eta}^{(m)}_{i,2^m-1}  , {\bm \eta}^{(m)}_{i,2^m}\big)}_{\infty} \\&= \max_{1 \le i \le 2^m}\Bnorm{\big({\bm \eta}^{(m)}_{i,1}, {\bm \eta}^{(m)}_{i,2} ,\cdots  , {\bm \eta}^{(m)}_{i,2^m-1}  , {\bm \eta}^{(m)}_{i,2^m}\big)^\top}_{\ell_1} = \bnorm{\big({\bm \eta}^{(m)}_{i,i}\big)^\top}_{\ell_1}\\&=\begin{cases}
				2^{\frac{m -n -3}{2}}\norm{(\bm r_{n-m}^+, \bm r_{n-m}^-)}_{\ell_1} &\quad \text{for} \quad m \ge 0 \\
				2^{\frac{-n -3}{2}}\norm{\bm r^{+}_{n+1}}_{\ell_1}  &\quad \text{for} \quad m =  -1 
			\end{cases} \\&= 2^{\frac{m \vee 0 -n-3}{2}}\cdot 2^{n+1-m\vee0} = 2^{\frac{n-1-m \vee 0}{2}}.
		\end{split}
	\end{equation*}
	Furthermore, for $-1 \le m \le n-1$, we get 
	\begin{equation*}
		\norm{R^{(n)}_m P_n}_{\infty} = \bnorm{\big(\cC^\top_{-1}, \cC^\top_0, \cdots, \cC^\top_{m}\big)^\top}_{\infty} = \max_{-1 \le i \le m}\norm{\cC_i}_\infty = 2^{\frac{n-1}{2}}.
	\end{equation*}
	Last, for $m = n$, we have 
	\begin{equation*}
		\begin{split}
			\norm{\bar{R}_n^{(n)}P_n}_\infty &= \left(2^{n}-1\right)\norm{\bm u}_{\infty} + \norm{\bm v}_{\infty} = \left(2^{n}-1\right)\norm{\bm u}_{\ell_1} + \norm{\bm v}_{\ell_1} \\&= \left(2^n-1\right)\cdot 2\sqrt{2} + \frac{1}{2\sqrt{2}} + \frac{3}{2\sqrt{2}} = 2^{n+\frac{3}{2}} - \sqrt{2}.
		\end{split}
	\end{equation*}
	This concludes the proof of \Cref{thm V norm} for $p=\infty$. 
	\end{proof}
	
	\begin{proof}[Proof of \Cref{thm V norm} for $p=1$]
		By utilizing the identity $\bar{R}^{(n)}_m P_n = \cC_m$, we get 
	\begin{equation*}
		\norm{\cC_m}_1 = \max_{1 \le i \le m}\Bnorm{\left(\big({\bm \eta}^{(m)}_{1,i}\big)^\top, \big({\bm \eta}^{(m)}_{2,i}\big)^\top ,\cdots  , \big({\bm \eta}^{(m)}_{2^m-1,i}\big)^\top  , \big({\bm \eta}^{(m)}_{2^m,i}\big)^\top\right)^\top}_1 = \max_{1 \le i \le m}\norm{\bm \eta^{(m)}_{i,i}}_1.
	\end{equation*}
	Since, for every $1 \le i \le 2^m$,  the entries of the vector $\bm \eta^{(m)}_{i,i}$ are either $2^{(m \vee 0 - n -3)/2}$ or $-2^{(m \vee 0 - n -3)/2}$,  the vector $\bm \eta^{(m)}_{i,i}$, regarded as a linear functional, has the $\ell_1$-induced norm $\norm{\bm \eta^{(m)}_{i,i}}_1 = 2^{(m \vee 0 - n -3)/2}$. Next, we have 
	\begin{equation*}
		\begin{split}
			\norm{R^{(n)}_m P_n}_{1} &= \bnorm{\left(\cC^\top_{-1}, \cC^\top_0, \cdots, \cC^\top_{m}\right)^\top}_{1} = \sum_{i = -1}^{m}\norm{\cC_i}_1 \\&= \sum_{i = -1}^{m}2^{(i \vee 0 - n -3)/2} = \begin{cases}
				2^{-\frac{n+3}{2}} & \quad \text{for} \quad  m = -1,\\
				2^{-\frac{n+3}{2}}\left(\frac{2^{\frac{m+1}{2}}-1}{\sqrt{2}-1}+1\right) &\quad \text{for} \quad  0 \le m \le n-1,
			\end{cases}
			\\&= 2^{-\frac{n+3}{2}} \left(1 + \frac{2^{\frac{m+1}{2}}-1}{\sqrt{2}-1}\mathbbm{1}_{\{m \ge 0\}}\right).
		\end{split}
	\end{equation*}
	Last, for the case $m = n$, it follows that
	\begin{equation*}
		\norm{\bar{R}_n^{(n)}P_n}_1 = \left(2^{n}-1\right)\norm{\bm u}_1 + \norm{\bm v}_1 = \left(2^n-1\right)\sqrt{2} + \frac{3}{2\sqrt{2}} = 2^{n+\frac{1}{2}} - \frac{\sqrt{2}}{4}.
	\end{equation*}
	This completes the proof. 
	\end{proof}

The proof of \Cref{thm V norm} for $p=2$ is more complicated than the ones for $p=1$ and $p=\infty$ and requires additional preparation. 
We start with a lemma concerning the exact form of product $\cC_m \cC^\top_k$ for $-1 \le m, k \le n-1$.

\begin{lemma}\label{lemma cmk}
For $-1 \le m \le n$, let $\cC_m$ be as in the context of \Cref{Lemma Pn}. Then,
\begin{equation*}
	\cC_{m}\cC^\top_k = \begin{cases}
		\frac{1}{4}\bm I_{2^{m \vee 0} \times 2^{m \vee 0}} &\quad \text{for} \quad 1 \le m = k \le n-1,\\
		\bm 0_{2^{m \vee 0} \times 2^{k \vee 0}} &\quad \text{for} \quad 1 \le  m \neq k \le n-1.
	\end{cases}
\end{equation*}
\end{lemma}

\begin{proof}
	Let us denote $\Phi_{(m,k)} = (\phi^{(m,k)}_{i,j}):= \cC_m \cC_k^\top$ and prove the case for $m = k$. Since
	\begin{equation*}
		\Phi_{(m,m)} = \left[
		\begin{array}{c c c c}
			{\bm \eta}^{(m)}_{1,1} & {\bm \eta}^{(m)}_{1,2} &\cdots  &  {\bm \eta}^{(m)}_{1,2^m}\\
			{\bm \eta}^{(m)}_{2,1} & {\bm \eta}^{(m)}_{2,2} &\cdots  &  {\bm \eta}^{(m)}_{2,2^m}\\
			\vdots & \vdots & \ddots&  \vdots \\
			{\bm \eta}^{(m)}_{2^m,1} & {\bm \eta}^{(m)}_{2^m,2} &\cdots  &  {\bm \eta}^{(m)}_{2^m,2^m}\\
		\end{array}
		\right]\left[
		\begin{array}{c c c c}
			\big({\bm \eta}^{(m)}_{1,1}\big)^\top & \big({\bm \eta}^{(m)}_{2,1}\big)^\top &\cdots  &  \big({\bm \eta}^{(m)}_{2^m,1}\big)^\top\\
			\big({\bm \eta}^{(m)}_{1,2}\big)^\top & \big({\bm \eta}^{(m)}_{2,2}\big)^\top &\cdots  &  \big({\bm \eta}^{(m)}_{2^m,2}\big)^\top\\
			\vdots & \vdots & \ddots&  \vdots \\
			\big({\bm \eta}^{(m)}_{1,2^m}\big)^\top & \big({\bm \eta}^{(m)}_{2,2^m}\big)^\top &\cdots  & \big( {\bm \eta}^{(m)}_{2^m,2^m}\big)^\top\\
		\end{array}
		\right],
	\end{equation*}
	then $\phi^{(m,m)}_{i,j}  = \sum_{k = 1}^{2^m} {\bm \eta}^{(m)}_{i,k}\big({\bm \eta}^{(m)}_{j,k}\big)^\top$. As $\bm \eta^{(m)}_{i,k} = \bm 0_{1 \times 2^{n-m+1}}$ for $i \neq k$, then $\phi^{(m,m)}_{i,j} = 0$ for all $i \neq j$. Furthermore, it again follows from \Cref{Lemma Pn} that for $1 \le i \le 2^m$,
	\begin{equation*}
		\phi^{(m,m)}_{i,i} = {\bm \eta}^{(m)}_{i,i}\big({\bm \eta}^{(m)}_{i,i}\big)^\top = 2^{m \vee 0 -n -3} \left(\norm{\bm r^{+}_{n-m}}^2_{\ell_2}+\norm{\bm r^{-}_{n-m}}^2_{\ell_2}\right) = 2^{m \vee 0 -n -3}\cdot 2^{n+1 - m\vee 0} = \frac{1}{4}.
	\end{equation*}
	This shows that $\Phi_{(m,m)} = \frac{1}{4}\bm I_{2^{m \vee 0} \times 2^{m \vee 0}}$. Next, for $0 \le m < k \le n-1$, we have 
	\begin{equation*}
		\begin{split}
			\cC_{m}\cC^\top_{k} 
			&= \left[
			\begin{array}{c c c c}
				{\bm \eta}^{(m)}_{1,1} & \bm{0}_{1\times 2^{n+1-m}} &\cdots  &  \bm{0}_{1\times 2^{n+1-m}}\\
				\bm{0}_{1\times 2^{n+1-m}} & {\bm \eta}^{(m)}_{2,2} &\cdots  & \bm{0}_{1\times 2^{n+1-m}}\\
				\vdots & \vdots & \ddots &  \vdots \\
				\bm{0}_{1\times 2^{n+1-m}} & \bm{0}_{1\times 2^{n+1-m}} &\cdots  &  {\bm \eta}^{(m)}_{2^m,2^m}\\
			\end{array}
			\right]\left[\begin{array}{c c c c}
				\big({\bm \eta}^{(k)}_{1,1}\big)^\top & \bm{0}_{2^{n+1-k}\times1 } &\cdots    & \bm{0}_{2^{n+1-k}\times1 }\\
				\bm{0}_{2^{n+1-k}\times1 } & \big({\bm \eta}^{(k)}_{2,2}\big)^\top &\cdots    & \bm{0}_{2^{n+1-k}\times1 }\\
				\vdots & \vdots & \ddots& \vdots \\
				\bm{0}_{2^{n+1-k}\times1 } & \bm{0}_{2^{n+1-k}\times1 } &\cdots   & \big({\bm \eta}^{(k)}_{2^k,2^k}\big)^\top\\
			\end{array}\right]\\&= \left[\begin{array}{c c c c}
				{\bm \eta}^{(m)}_{1,1} {\bm \eta}^{(m,k)}_{1,1}  & \bm{0}_{1\times 2^{k-m}} &\cdots    & \bm{0}_{1\times 2^{k-m}}\\
				\bm{0}_{1\times 2^{k-m}} & {\bm \eta}^{(m)}_{2,2}{\bm \eta}^{(m,k)}_{2,2} &\cdots    & \bm{0}_{1\times 2^{k-m}}\\
				\vdots & \vdots & \ddots& \vdots \\
				\bm{0}_{1\times 2^{k-m}} & \bm{0}_{1\times 2^{k-m}} &\cdots   & {\bm \eta}^{(m)}_{2^m,2^m}{\bm \eta}^{(m,k)}_{2^m,2^m} \\
			\end{array}\right],
		\end{split}
	\end{equation*}
	where
	\begin{equation*}
		{\bm \eta}^{(m,k)}_{i,i} := \left[\begin{array}{c c c c}
			\big({\bm \eta}^{(k)}_{1,1}\big)^\top & \bm{0}_{2^{n+1-k}\times1 } &\cdots    & \bm{0}_{2^{n+1-k}\times1 }\\
			\bm{0}_{2^{n+1-k}\times1 } & \big({\bm \eta}^{(k)}_{2,2}\big)^\top &\cdots    & \bm{0}_{2^{n+1-k}\times1 }\\
			\vdots & \vdots & \ddots& \vdots \\
			\bm{0}_{2^{n+1-k}\times1 } & \bm{0}_{2^{n+1-k}\times1 } &\cdots   & \big({\bm \eta}^{(k)}_{2^{k-m},2^{k-m}}\big)^\top\\
		\end{array}\right],
	\end{equation*}
	for any $1 \le i \le 2^m$. Note that for $0 \le m \le n-1$ and $1 \le i \le 2^m$,
	\begin{equation*}
		\begin{split}
			&{\bm \eta}^{(m)}_{i,i}{\bm \eta}^{(m,k)}_{i,i}=2^{(m-n-3)/2}\left(
			(\bm r^+_{n-m})^\top,
			(\bm r^-_{n-m})^\top
			\right)\left[\begin{array}{c c c c}
				\big({\bm \eta}^{(k)}_{1,1}\big)^\top & \bm{0}_{2^{n+1-k}\times1 } &\cdots    & \bm{0}_{2^{n+1-k}\times1 }\\
				\bm{0}_{2^{n+1-k}\times1 } & \big({\bm \eta}^{(k)}_{2,2}\big)^\top &\cdots    & \bm{0}_{2^{n+1-k}\times1 }\\
				\vdots & \vdots & \ddots& \vdots \\
				\bm{0}_{2^{n+1-k}\times1 } & \bm{0}_{2^{n+1-k}\times1 } &\cdots   & \big({\bm \eta}^{(k)}_{2^{k-m},2^{k-m}}\big)^\top\\
			\end{array}\right]\\&= 2^{m-n-3}\big(\underbrace{(\bm r_{n-k}^+)^\top,\cdots,(\bm r_{n-k}^+)^\top}_{2^{k-m} \text{ times}},\,\, \underbrace{(\bm r_{n-k}^-)^\top,\cdots,(\bm r_{n-k}^-)^\top}_{2^{k-m} \text{ times}}\big)\left[\begin{array}{c c c c}
				\bm r^+_{n-k} & \bm{0}_{2^{n-k}\times1 } &\cdots    & \bm{0}_{2^{n-k}\times1 }\\
				\bm r^-_{n-k} & \bm{0}_{2^{n-k}\times1 } &\cdots    & \bm{0}_{2^{n-k}\times1 }\\
				\bm{0}_{2^{n-k}\times1 } & \bm r^+_{n-k} &\cdots    & \bm{0}_{2^{n-k}\times1 }\\
				\bm{0}_{2^{n-k}\times1 } & \bm r^-_{n-k} &\cdots    & \bm{0}_{2^{n-k}\times1 }\\
				\vdots & \vdots & \vdots& \vdots \\
				\bm{0}_{2^{n-k}\times1 } & \bm{0}_{2^{n-k}\times1 } &\cdots   & \bm r^+_{n-k}\\
				\bm{0}_{2^{n-k}\times1 } & \bm{0}_{2^{n-k}\times1 } &\cdots   & \bm r^-_{n-k}\\
			\end{array}\right]\\
			&= 2^{m - n -3}\big(\underbrace{(\bm r_{n-k}^+)^\top(\bm r_{n-k}^+ + \bm r_{n-k}^-),\cdots}_{2^{k-m-1} \text{ times}},\,\, \underbrace{(\bm r_{n-k}^-)^\top(\bm r_{n-k}^+ + \bm r_{n-k}^-),\cdots}_{2^{k-m-1} \text{ times}}\big) = \bm 0_{1 \times 2^{k-m}}.
		\end{split}
	\end{equation*}
	Hence, $\cC_{m}\cC^\top_{k}  = \bm 0_{2^m \times 2^k}$ for $0 \le m < k \le n-1$, and for $0 \le k < m \le n-1$, one has $\cC_{m}\cC^\top_{k} = \big(\cC_{k}\cC^\top_{m}\big)^\top = \bm 0^\top_{2^k \times 2^m} = \bm 0_{2^m \times 2^k}$. Therefore, for $0 \le m \neq k \le n-1$, we have $\cC_{m}\cC^\top_{k}  = \bm 0_{2^m \times 2^k}$. Last, for the special case $m = -1$ and $0 \le k \le n-1$, we have  
	\begin{equation*}
		\cC_{-1}\cC_k = 2^{-n-3 + k/2} \big(\underbrace{(\bm r_{n-k}^+)^\top(\bm r_{n-k}^+ + \bm r_{n-k}^-),\cdots,(\bm r_{n-k}^+)^\top(\bm r_{n-k}^+ + \bm r_{n-k}^-)}_{2^{k} \text{ times}}\big) = \bm 0_{1 \times 2^k},
	\end{equation*}
	and by a similar argument as above, we get $\cC_{k}\cC_{-1} = \bm 0_{2^k \times 1}$. This completes the proof.
\end{proof}

Let us next consider the following $n \times n$ dimensional matrix $D_{n}$, where the $(i,j)^\text{th}$-entry $d_{i,j} = i \vee j -1/2$ for $1 \le i,j \le n$. The next lemma computes the $\ell_2$-induced norm of $D_n$.

\begin{lemma}\label{Lemma Dn Inverse}
	For $n \ge 3$, we have 
	$
	\norm{D_n}_2 = \frac{1}{2}(1- \cos\frac{\pi}{2n})^{-1}.
	$\end{lemma}
	
\begin{proof}
	In the first step, we first derive the inverse of $D_n$ and then obtain $\norm{D_n}_2$. To this end, let us consider the $n \times n$ dimensional tridiagonal matrix $E_n$ which is defined as 
	\begin{equation*}
		E_n := \left[
		\begin{array}{ccccc}
			3 & -1 & 0 & 0 & 0 \\
			-1 & 2 & -1 & 0 & \ddots \\
			0 & -1 & 2 & \ddots & 0 \\
			0 & 0 & \ddots & \ddots & -1 \\
			0 & \ddots & 0 & -1 & 1 \\
		\end{array}
		\right],
	\end{equation*}
	and we are about to demonstrate $E_n = D^{-1}_n$ for $n \ge 3$. Since $D_n$ and $E_n$ are symmetric, it suffices to show that the product matrix $F_n:= E_nD_n = \bm I_{n \times n}$. Denote the $i^\text{th}$ row of the matrix $D_n$ by $\bm d_i$ for $1 \le i \le n$, and $D_n = (\bm d_1^\top, \bm d_2^\top, \cdots, \bm d_n^\top)^\top$. Writing the $i^\text{th}$ row of the matrix $F_n$ by $\bm f_i$ and regarding matrix as a row operation acting on the matrix $D_n$ leads to
	\begin{equation*}
		F_n = \left(\begin{array}{c}
			\bm f_1\\ \bm f_2 \\ \vdots\\ \bm f_{n-1}\\ \bm f_n 
		\end{array}\right) =   E_n \left(\begin{array}{c}
			\bm d_1\\ \bm d_2 \\ \vdots\\ \bm d_{n-1}\\ \bm d_n 
		\end{array}\right) = \left(\begin{array}{c}
			3\bm d_1 - \bm d_2\\2 \bm d_2 -\bm d_1 - \bm d_3\\ \vdots\\ 2\bm d_{n-1} - \bm d_{n-2} - \bm d_{n}\\ \bm d_n - \bm d_{n-1}
		\end{array}\right).
	\end{equation*}
	It remains to verify the form of $\bm f_i$ term by term. First, we have 
	\begin{equation*}
		\bm f_1 = 3\bm d_1 - \bm d_2 = 3 \cdot \left(\frac{1}{2}, \frac{1}{2}, \cdots,\frac{1}{2} \right) - \left(\frac{1}{2}, \frac{3}{2}, \cdots,\frac{3}{2} \right) = (1,0,\cdots,0).
	\end{equation*}
	Next, for $2 \le i \le n-1$ and $1 \le j \le n$, we have 
	\begin{equation*}
		t_{i,j} = 2 (i \vee j) -(i-1) \vee j - (i+1) \vee j = \begin{cases}
			1 \quad \text{for} \quad i = j,\\
			0 \quad \text{otherwise}.
		\end{cases}
	\end{equation*} 
	Thus, for $2 \le i \le n-1$, we have
	\begin{equation*}
		\bm f_i = 2\bm d_i - \bm d_{i-1}- \bm d_{i+1} = \big(\underbrace{0,\cdots, 0}_{i-1 \, \text{times}},1,\underbrace{0, \cdots, 0}_{ n-1 \, \text{times}}\big).
	\end{equation*}
	Last, we have 
	\begin{equation*}
		\bm f_n = \bm d_n - \bm d_{n-1} = \Big(\frac{1}{2}, \frac{3}{2}, \frac{5}{2}, \cdots, \frac{2n-3}{2},\frac{2n-1}{2} \Big)-\Big(\frac{1}{2}, \frac{3}{2}, \frac{5}{2}, \cdots, \frac{2n-3}{2},\frac{2n-3}{2} \Big) = \big(0,0, \cdots, 0,1\big).
	\end{equation*}
	Hence, the above calculation establishes that $E_n = D^{-1}_n$.

	For the second step of the proof, a theorem by Losonczi \cite{Losonczi} states that the determinant of $E_n - \lambda \bm I_{n \times n}$ admits the following representation:
	\begin{equation*}
		\det(E_n - \lambda \bm I_{n \times n}) = \prod_{j = 0}^{n-1}\left(2-\lambda - 2\cos\frac{(j+1)\pi}{2n}\right),
	\end{equation*}
	and this leads to $\lambda_{\min}(E_n) = 2 - 2\cos\frac{\pi}{2n}$.
	Thus, we have 
	\begin{equation*}
		\norm{D_n}_2 = \lambda_{\max}(D_n) = \frac{1}{\lambda_{\min}(E_n)} = \frac{1}{2}\left(1- \cos\frac{\pi}{2n}\right)^{-1},
	\end{equation*}
	and this completes the proof.
\end{proof}

We are now ready to complete the proof of  \Cref{thm V norm}.

\begin{proof}[Proof of \Cref{thm V norm} for $p=2$]
	By the definition of the spectral norm, we get 
	\begin{equation}\label{eq bar R_n norm}
		\norm{\bar{R}^{(n)}_m P_n}_2 = \sqrt{\lambda_{\max}\big(\bar{R}^{(n)}_m P_n P_n^\top \big(\bar{R}^{(n)}_m\big)^\top\big)} = \sqrt{\lambda_{\max}\big(\cC_m\cC_m^\top\big)} = \frac{1}{2}\sqrt{\lambda_{\max}\big(\bm I_{2^m \times 2^m}\big)} = \frac{1}{2},
	\end{equation}
	where the second last identity follows from \Cref{lemma cmk}. To derive the $\ell_2$-induced norm of matrices $R^{(n)}_m P_n$, note that 
	\begin{equation*}
		R^{(n)}_m P_n P_n^\top \big(R^{(n)}_m\big)^\top = \left[\begin{array}{c c c c}
			\cC_{-1}\cC^\top_{-1} & \cC_{-1}\cC^\top_{0}  & \cdots & \cC_{-1}\cC^\top_{m}\\ 
			\cC_{0}\cC^\top_{-1} & \cC_{0}\cC^\top_{0}  & \cdots & \cC_{0}\cC^\top_{m}\\ 
			\vdots & \vdots & \ddots  & \vdots\\
			\cC_{m}\cC^\top_{-1} & \cC_{m}\cC^\top_{0} & \cdots & \cC_{m}\cC^\top_{m}\\ 
		\end{array}\right] = \frac{\bm I_{2^{m+1}\times 2^{m+1}}}{4}.
	\end{equation*}
	An analogous argument as in \eqref{eq bar R_n norm} yields that $\norm{\bar{R}^{(n)}_m P_n}_2 = \frac{1}{2}$. Last, it remains to obtain the spectral norm of $\bar{R}^{(n)}_n P_n$, or equivalently, the largest singular value of $\cC_n\cC_n^\top$. To this end, let us first of all derive that exact form of the matrix $G_n = (g^{(n)}_{i,j}) = \cC_n\cC_n^\top$. Next, we shall claim and show that 
	\begin{equation}\label{eq matrix G form}
		g^{(n)}_{i,j}  = \begin{cases}
			4i - \frac{11}{4} &\qquad \text{for} \quad 1 \le i = j \le 2^n,\\
			4(i \wedge j)-2 &\qquad \text{for} \quad 1 \le i \neq j \le 2^n.
		\end{cases}
	\end{equation} 
	First, one has 
	\begin{equation*}
		g^{(n)}_{i,i} = \norm{\big(\underbrace{\bm u, \cdots, \bm u}_{i-1 \, \text{times}}, \bm v, \underbrace{\bm 0_{1 \times 2}, \cdots, \bm 0_{1 \times 2}}_{2^n - i \, \text{times}}\big)}^2_{\ell_2} = 4(i-1) + \left(\Big(\frac{3}{2\sqrt{2}}\Big)^2 +\Big(\frac{1}{2\sqrt{2}}\Big)^2 \right) = \frac{5}{4} + 4(i-1) = 4i - \frac{11}{4}.
	\end{equation*}
	On the other hand, we assume that $i \le j -1$ with loss of generality, then 
	\begin{equation*}
		\begin{split}
			g^{(n)}_{i,j} &= \big(\underbrace{\bm u, \cdots, \bm u}_{i-1 \, \text{times}}, \bm v, \underbrace{\bm 0_{1 \times 2}, \cdots, \bm 0_{1 \times 2}}_{2^n - i \, \text{times}}\big)\big(\underbrace{\bm u^\top, \cdots, \bm u^\top}_{j-1 \, \text{times}}, \bm v^\top, \underbrace{\bm 0_{2 \times 1}, \cdots, \bm 0_{2 \times 1}}_{2^n - j \, \text{times}}\big)^\top \\&= 4(i-1) + \sqrt{2}\left(\frac{3}{2\sqrt{2}} + \frac{1}{2\sqrt{2}}\right) = 4i -2.
		\end{split}
	\end{equation*}
	This demonstrates the exact form of the matrix $G_n$ as shown in \eqref{eq matrix G form}. Since $G_n = 4 D_{2^n} - \frac{3}{4}\bm I_{2^n \times 2^n}$, thus 
	\begin{equation*}
		\norm{G_n}_2 = \Bnorm{4 D_{2^n } - \frac{3}{4}\bm I_{2^n \times 2^n}}_2 \le 4\norm{D_{2^n}}_2+ \frac{3}{4} \le 2\left(1- \cos\frac{\pi}{2^{n+1}}\right)^{-1} + \frac{3}{4},
	\end{equation*}
	where the last inequality follows from \Cref{Lemma Dn Inverse}. Similarly, we have 
	\begin{equation*}
		\norm{G_n}_2 = \Bnorm{4 D_{2^n} - \frac{3}{4}\bm I_{2^n \times 2^n}}_2  \ge \left|4\norm{D_{2^n}}_2 - \frac{3}{4}\right| = 2\left(1- \cos\frac{\pi}{2^{n+1}}\right)^{-1} - \frac{3}{4}.
	\end{equation*}
	Thus, we have $\sqrt{2\left(1 - \cos\frac{\pi}{2^{n+1}}\right)^{-1} - \frac{3}{4}} \le \bnorm{\bar{R}^{(n)}_nP_n}_{2} = \le \sqrt{2\left(1 - \cos\frac{\pi}{2^{n+1}}\right)^{-1} + \frac{3}{4}}$. Now, it remains to show that $\bnorm{\bar{R}^{(n)}_nP_n}_{2} = \bnorm{R^{(n)}_nP_n}_{2}$. To this end, note that $R^{(n)}_n = \bm I_{2^{n+1} \times 2^{n+1}}$. Hence, we have 
	\begin{equation*}
		R^{(n)}_n P_n P_n^\top \big(R^{(n)}_n\big)^\top = P_nP_n^\top =  \left[\begin{array}{c c c c}
			\cC_{-1}\cC^\top_{-1} & \cC_{-1}\cC^\top_{0}  & \cdots & \cC_{-1}\cC^\top_{n}\\ 
			\cC_{0}\cC^\top_{-1} & \cC_{0}\cC^\top_{0}  & \cdots & \cC_{0}\cC^\top_{n}\\ 
			\vdots & \vdots & \ddots  & \vdots\\
			\cC_{n}\cC^\top_{-1} & \cC_{n}\cC^\top_{0} & \cdots & \cC_{n}\cC^\top_{n}\\ 
		\end{array}\right] = \left[\begin{array}{c c}
			\frac{\bm I_{2^{n}\times 2^{n}}}{4} & \bm 0_{2^n \times 2^n}\\
			\bm 0_{2^n \times 2^n} &G_n
		\end{array}\right].
	\end{equation*}
	Thus, we have 
	\begin{equation*}
		\begin{split}
			\bnorm{R^{(n)}_nP_n}_{2} &= \sqrt{\lambda_{\max}(R^{(n)}_n P_n P_n^\top \big(R^{(n)}_n\big)^\top)} = \sqrt{\lambda_{\max}\left(\frac{\bm I_{2^{n}\times 2^{n}}}{4}\right) \vee \lambda_{\max}\left(G_n\right)} \\&= \sqrt{\lambda_{\max}\left(G_n\right)} = \sqrt{\lambda_{\max}(\bar{R}^{(n)}_n P_n P_n^\top \big(\bar{R}^{(n)}_n\big)^\top)} = \bnorm{\bar{R}^{(n)}_nP_n}_{2}.
		\end{split}
	\end{equation*}
	This completes the proof.
\end{proof}

\section{Proofs for the results in \Cref{Section Main}}
\label{Section 1 proofs section}

\begin{proof}[Proof of \Cref{thm represent}]
	We consider first the case in which $\hat f_0=0$. 
	Let us consider the $2^{n+1}$ dimensional column vector $\bm \omega_{n+1} = (\omega^{(n+1)}_{i} )_{i = 1}^{2^{n+1}}$, where
	$
	\omega^{(n+1)}_{i} = F(i{2^{-n-1}}) - F((i-1){2^{-n-1}}).
	$	Recall the explicit form of $Q^{-1}_{(n+1,n+1)}$ in \eqref{eq Q inverse}, which can be regarded as a row operation acting on a vector by subtracting preceding rows after properly scaling. Therefore, we have $Q^{-1}_{(n+1,n+1)}\bm y_{n+1} = 2^{3(n+1)/2+2}\bm \omega_{n+1}$. Thus,
	\begin{equation*}
		\bm \vartheta^{(n)}_{n-1} = R^{(n)}_{n-1}P_nQ^{-1}_{(n+1,n+1)}\bm y_{n+1} = 2^{3(n+1)/2+2}R^{(n)}_{n-1}P_n\bm \omega_{n+1}.
	\end{equation*}
	Analogously, for $0 \le m \le n-1$, we partition $\bm \omega_{n+1}$ as follows,
	\begin{equation*}
		\bm \omega_{n+1} = \left((\bm \omega^{(m,n+1)}_{1})^\top, (\bm \omega^{(m,n+1)}_{2})^\top, \cdots,  (\bm \omega^{(m,n+1)}_{2^{m}})^\top\right)^\top,
	\end{equation*}
	where each vector $\bm \omega^{(m,n+1)}_{k}$ is a $2^{n+1-m}$ dimensional column vector. Then, for a similar reason, we have 
	\begin{equation}\label{eq mu nmk}
		\begin{split}
			\vartheta^{(n)}_{m,k} &= 2^{3(n+1)/2+2}\left({\bm \eta}^{(m)}_{k+1,1},{\bm \eta}^{(m)}_{k+1,2},\cdots, {\bm \eta}^{(m)}_{k+1,2^m}\right)\cdot \bm \omega_{n+1} = 2^{3(n+1)/2+2}\bm \eta^{(m)}_{k+1,k+1}\bm \omega^{(m,n+1)}_{k+1}\\&= 2^{3(n+1)/2+(m-n-3)/2+2}\left((\bm r^+_{n-m})^\top,(\bm r^-_{n-m})^\top \right)\left((\bm \omega^{(m+1,n+1)}_{2k+1})^\top,(\bm \omega^{(m+1,n+1)}_{2k+2})^\top\right)^\top\\&= 2^{n+m/2+2}\left((\bm r^+_{n-m})^\top\bm \omega^{(m+1,n+1)}_{2k+1}+(\bm r^-_{n-m})^\top\bm \omega^{(m+1,n+1)}_{2k+2}\right).
		\end{split}
	\end{equation}
	A routine calculation gives that 
	\begin{equation*}
		(\bm r^+_{n-m})^\top\bm \omega^{(m+1,n+1)}_{2k+1} = \sum_{j = 1}^{2^{n-m}}(-1)^j\omega_{2^{n-m+1}k+j}^{(n+1)}  =  \sum_{j = 1}^{2^{n-m}}(-1)^j\left(F\Big(\frac{k}{2^m}+\frac{j}{2^{n+1}}\Big)-F\Big(\frac{k}{2^m}+\frac{j-1}{2^{n+1}}\Big)\right),
	\end{equation*}
	and we similarly have 
	\begin{equation*}
		\begin{split}
			&(\bm r^-_{n-m})^\top\bm \omega^{(m+1,n+1)}_{2k+2} = \sum_{j = 1}^{2^{n-m}}(-1)^{j+1}\omega_{2^{n-m}(2k+1)+j}^{(n+1)} = \sum_{j = 0}^{2^{n-m}-1}(-1)^{j+1}\omega_{2^{n-m+1}(k+1)-j}^{(n+1)}  \\&= \sum_{j = 1}^{2^{n-m}}(-1)^{j}\omega_{2^{n-m+1}(k+1)-(j-1)}^{(n+1)}=  \sum_{j = 1}^{2^{n-m}}(-1)^j\left(F\Big(\frac{k+1}{2^m}-\frac{j-1}{2^{n+1}}\Big)-F\Big(\frac{k+1}{2^m}-\frac{j}{2^{n+1}}\Big)\right).
		\end{split}
	\end{equation*}
	Plugging the above two identities back to \eqref{eq mu nmk} gives \eqref{eq zeta}. Last, for $m = -1$, we  have
	\begin{equation*}
		\begin{split}
			\vartheta^{(n)}_{-1,0} &= 2^{3(n+1)/2+2}\cC_{-1}\bm \omega_{n+1}= 2^{n+2} \bm r^{+}_{n+1} \bm \omega_{n+1} = 2^{n+2}\sum_{j = 1}^{2^{n+1}}(-1)^j \omega^{(n+1)}_{j}\\& = 2^{n+2}\sum_{j = 1}^{2^{n+1}}(-1)^j\left(F\Big(\frac{j}{2^{n+1}}\Big) - F\Big(\frac{j-1}{2^{n+1}}\Big)\right).
		\end{split}
	\end{equation*}
	Last, note that $\bar{\bm \vartheta}^{(n)}_n =  2^{3(n+1)/2+2}\bar{R}^{(n)}_nP_n\bm \omega_{n+1} = 2^{3(n+1)/2+2}\cC_{n}\bm \omega_{n+1}$, and therefore,
	\begin{equation*}
		\begin{split}
			\vartheta^{(n)}_{n,k} &= 2^{3(n+1)/2+2} \big(\underbrace{\bm u,\cdots, \bm u}_{k \,\, \text{times}} ,\, \bm v ,\,  \underbrace{\bm{0}_{1\times 2},\, \cdots,\,\bm{0}_{1\times 2}}_{2^{n} - k -1 \,\, \text{times}}\big) \bm \omega_{n+1}\\&= 2^{3(n+1)/2+2} \big(\underbrace{\bm u,\cdots, \bm u}_{k \,\, \text{times}} ,\, \bm v ,\,  \underbrace{\bm{0}_{1\times 2},\, \cdots,\,\bm{0}_{1\times 2}}_{2^{n} - k -1 \,\, \text{times}}\big) \left((\bm \omega^{(n,n+1)}_{1})^\top, (\bm \omega^{(n,n+1)}_{2})^\top, \cdots,  (\bm \omega^{(n,n+1)}_{2^{n}})^\top\right)^\top \\&= 2^{3(n+1)/2+2}\left(\sum_{j = 1}^{k} \bm u \cdot \bm \omega^{(n,n+1)}_j + \bm v \cdot \bm \omega^{(n,n+1)}_{k+1}\right) \\&= 2^{3(n+1)/2 + 2} \left(\sqrt{2}\sum_{j = 1}^{k}(\omega^{(n+1)}_{2j-1} - \omega^{(n+1)}_{2j}) + \frac{3}{2\sqrt{2}}\omega^{(n+1)}_{2k+1} - \frac{1}{2\sqrt{2}}\omega^{(n+1)}_{2k+2}\right)\\&= 2^{3(n+1)/2 + 5/2}\sum_{j = 1}^{k}\left(2F\left(\frac{2j-1}{2^{n+1}}\right)-F\left(\frac{2j-2}{2^{n+1}}\right)-F\left(\frac{2j}{2^{n+1}}\right)\right) \\&+ 3\cdot2^{3(n+1)/2+1}\left(F\left(\frac{2k+1}{2^{n+1}}\right)-F\left(\frac{2k}{2^{n+1}}\right)\right) - 2^{3(n+1)/2+1}\left(F\left(\frac{2k+2}{2^{n+1}}\right)-F\left(\frac{2k+1}{2^{n+1}}\right)\right).
		\end{split}
	\end{equation*}
	This completes the proof for the case  $\hat f_0=0$.
	
	For  $\hat f_0\neq0$, we let $\wt\vartheta^{(n)}_{m,k}$ be the estimated Faber--Schauder coefficients for $\hat f_0=0$, as in the first part of this proof.
	Then the functions
	$$\wt f_n:=\wt\vartheta^{(n)}_{-1,0}e_{-1,0} + \sum_{m = 0}^{n}\sum_{k = 0}^{2^m-1}\wt\vartheta^{(n)}_{m,k}e_{m,k}\quad\text{and}\quad \wt F(t):=F(0)+\int_0^t\wt f_n(s)\,ds
	$$	
	are such that $
	\wt F(t)=F(t)$ for all $t\in\bT_{n+1}$.	Now we define
	$\vartheta^{(n)}_{m,k}:=\wt\vartheta^{(n)}_{m,k}$ for $m<n$ and $\vartheta^{(n)}_{n,k}:=\wt\vartheta^{(n)}_{n,k}-2^{-n/2+2}\hat f_0$. Then we let
	$$ \hat{f}_{n} := \hat f_0+\vartheta^{(n)}_{-1,0}e_{-1,0} + \sum_{m = 0}^{n}\sum_{k = 0}^{2^m-1}\vartheta^{(n)}_{m,k}e_{m,k}, \quad\text{and}\quad \hat{F}_{n}(t) := F(0)+\int_{0}^{t}\hat{f}_{n}(s)\,ds.
	$$ 
	With these definition, we have that 	
	\begin{align*}
		\hat F_n(t)=\wt F(t)+\hat f_0\cdot t-2^{-n/2+2}\hat f_0\int_0^t\sum_{k=0}^{2^n-1}e_{n,k}(s)\,ds=\wt F(t)+\hat f_0\cdot t-2^{-n/2+2}\hat f_0\sum_{k=0}^{2^n-1}\psi_{n,k}(t).
	\end{align*}
	By using \eqref{psi def eq} and considering separately the case $t\in\bT_n$ and $t\in\bT_{n+1}\setminus\bT_n$, one easily checks that $2^{-n/2+2}\sum_{k=0}^{2^n-1}\psi_{n,k}(t)=t$ for $t\in\bT_{n+1}$. Hence,  $\hat F(t)=\wt F(t)=F(t)$ for all $t\in\bT_{n+1}$, and so the coefficients $\vartheta^{(n)}_{m,k}$ are as desired. To show uniqueness, one argues in a similar way, transforming any solution to \eqref{eq inverse Faber} back to the case $\hat f_0=0$ and using the already established uniqueness there.
\end{proof}

\begin{proof}[Proof of the formulas in \Cref{Example Takagi}]
	Consider the following partition of $\bm z_{n+1}$,
	\begin{equation*}
		\bm z_{n+1} = \left((\bm z^{(m,n+1)}_{1})^\top, (\bm z^{(m,n+1)}_{2})^\top, \cdots,  (\bm z^{(m,n+1)}_{2^{m}})^\top\right)^\top,
	\end{equation*}
	where each vector $\bm z^{(m,n+1)}_{k}$ is a $2^{n+1-m}$ dimensional column vector for $0 \le m \le n+1$ and $1 \le k \le 2^{n+1-m}$. Moreover, for each $1 \le i \le 2^{n+1}$, we have 
	\begin{equation}\label{eq zH}
		\begin{split}
			z^{(n+1)}_{i} &= 2^{3(n+1)/2}\sum_{m = n+1}^{\infty}2^{-3m/2}\sum_{k = 2^{m-n-1}(i-1)}^{2^{m-n-1}i-1}\theta_{m,k} = 2^{3(n+1)/2}\sum_{m = n+1}^{\infty}2^{-3m/2}\sum_{k = 2^{m-n-1}(i-1)}^{2^{m-n-1}i-1}2^{m/2}c_m \\&= 2^{3(n+1)/2}\sum_{m = n+1}^{\infty}2^{-3m/2 + m-n-1 + m/2}c_m= 2^{(n+1)/2} \sum_{m = n+1}^\infty c_m =: \kappa_{n+1} < \infty,
		\end{split}
	\end{equation}
	where the finiteness of $\kappa_{n+1}$ follows from the absolute summability of the sequence $(c_m)$. Therefore, for every $1 \le k \le 2^m$, we have 
	\begin{equation*}
		\bm z^{(m,n+1)}_{k} = (\underbrace{\kappa_{n+1}, \kappa_{n+1}, \cdots, \kappa_{n+1}}_{2^{n+1-m} \text{ times}} ).
	\end{equation*}
	Furthermore, it follows from \Cref{Lemma Pn} that ${\bm \eta}^{(m)}_{i,j} = \bm{0}_{1 \times 2^{n-m+1}}$ for $i \neq j$, then for $0 \le m \le n-1$ and $0 \le k \le 2^m-1$
	\begin{equation*}
		\begin{split}
			\vartheta^{(n)}_{m,k} - \theta_{m,k} &= \left({\bm \eta}^{(m)}_{k+1,1}, {\bm \eta}^{(m)}_{k+1,2}, \cdots,  {\bm \eta}^{(m)}_{k+1,2^m}\right) \bm z_{n+1} = \bm \eta^{(m)}_{k+1,k+1}\bm z^{(m,n+1)}_{k+1} \\&= 2^{\frac{m-n-3}{2}}\left((\bm r^+_{n-m})^\top,(\bm r^-_{n-m})^\top \right)\left((\bm z^{(m,n)}_{1})^\top,(\bm z^{(m,n)}_{1})^\top\right)^\top \\&= 2^{\frac{m-n-3}{2}}\left(\bm r^+_{n-m} + \bm r^-_{n-m}\right)^\top\bm z^{(m,n)}_{1} = 0.
		\end{split}
	\end{equation*}
	The proof for the case $m = -1$ is analogous and we next prove the case $m = n$. Let us use the shorthand notation $\bm \kappa_{n+1} = (\kappa_{n+1}, \kappa_{n+1})$, and an analogous calculation yields that
	\begin{equation*}
		\begin{split}
			\vartheta^{(n)}_{n,k} - \theta_{n,k} &= \big(\underbrace{\bm u,\cdots, \bm u}_{k \,\, \text{times}} ,\, \bm v ,\,  \underbrace{\bm{0}_{1\times 2},\, \cdots,\,\bm{0}_{1\times 2}}_{2^{n} - k -1 \,\, \text{times}}\big) \bm z_{n+1} \\& = \big(\underbrace{\bm u,\cdots, \bm u}_{k \,\, \text{times}} ,\, \bm v ,\,  \underbrace{\bm{0}_{1\times 2},\, \cdots,\,\bm{0}_{1\times 2}}_{2^{n} - k -1 \,\, \text{times}}\big)\big(\underbrace{\bm \kappa_{n+1}, \bm \kappa_{n+1}, \cdots, \bm \kappa_{n+1}}_{2^n \, \text{times}} \big)^\top\\& = k \big(\bm u\cdot \bm \kappa^\top_{n+1}\big) + \bm v \cdot \bm \kappa^\top_{n+1}  = \Big(\frac{3}{2\sqrt{2}}, - \frac{1}{2\sqrt{2}}\Big) \cdot (\kappa_{n+1},\kappa_{n+1})^\top = \frac{\sqrt{2}}{2} \kappa_{n+1} = 2^{n/2} \sum_{m = n+1}^\infty c_m.
		\end{split}
	\end{equation*}
	Thus, we have 
	\begin{equation*}
		\vartheta^{(n)}_{n,k} = 2^{n/2}c_n + 2^{n/2} \sum_{m = n+1}^\infty c_m = 2^{n/2} \sum_{m = n}^\infty c_m.
	\end{equation*}
	This completes the proof.
\end{proof}

\begin{proof}[Proof of Proof of \Cref{zn+1 proposition}]
	{Let us first prove the case $k = 0$}. Recall that the first-order modulus of continuity is defined as 
	\begin{equation*}
		\omega_1(f,t) := \sup_{\substack{\tau_1,\tau_2 \in [0,1]\\|\tau_1 - \tau_2|\le t}}\left|f(\tau_1)-f(\tau_2)\right|.
	\end{equation*}
	Following from the triangular inequality and the H\"older continuity of $f$, there exists $c \in [0,1]$ such that $\omega_2(f,t) \le 2\omega_1(f,t/2) \le c|\tau_1 - \tau_2|^{\alpha} \le c\cdot t^{\alpha}$. Note that \eqref{eq interpolation} implies $|\theta_{m,k}|\le 2^{m/2}\omega_2(f,2^{-m})$, thus
	\begin{equation*}
		\begin{split}
			|z^{(n+1)}_i |&= \sum_{m = n+1}^{\infty}2^{-3(m-n-1)/2}\sum_{k = 2^{m-n-1}(i-1)}^{2^{m-n-1}i-1}|\theta_{m,k}| \le \sum_{m = n+1}^{\infty}2^{-3(m-n-1)/2}\cdot\Big(2^{3m/2-n-1}\omega_2(f,2^{-m})\Big) \\
			&= 2^{(n+1)/2}\sum_{m = n+1}^{\infty} \omega_2(f,2^{-m})\le  c \cdot 2^{(n+1)/2} \sum_{m = n+1}^{\infty}2^{-m\alpha} \le \frac{c\cdot 2^{-(n+1)(\alpha-\frac{1}{2})}}{1-2^{-\alpha}}.
		\end{split}
	\end{equation*}
	From here, the assertion follows easily. {To prove the case $k = 1$}, by assumption, there exists $c \in [0,1]$ such that 
	\begin{equation}\label{eq omega bound}
		\begin{split}
			\omega_2(f,t) &= \sup_{\substack{\tau_1,\tau_2 \in [0,1]\\|\tau_1 - \tau_2|\le t}}\left\{\left|f(\tau_1)+f(\tau_2)-2f\Big(\frac{\tau_1 + \tau_2}{2}\Big)\right|\right\} \\ &\le |\tau_1 - \tau_2|\sup_{\substack{\tau_1,\tau_2 \in [0,1]\\|\tau_1 - \tau_2| \le t}}\left(\sup_{u_1,u_2\in [\tau_1,\tau_2]} |f'(u_1) - f'(u_2)|\right) \le c|\tau_1 - \tau_2|^{1 + \alpha} \le c\cdot t^{1+\alpha}.
		\end{split}
	\end{equation}
	Note next that  \eqref{eq interpolation} implies $|\theta_{m,k}|\le 2^{m/2}\omega_2(f,2^{-m})$. A similar argument yields 
	\begin{equation}\label{eq def zn 2}
		\begin{split}
			|z^{(n+1)}_i |
			&= 2^{(n+1)/2}\sum_{m = n+1}^{\infty} \omega_2(f,2^{-m})\le  c \cdot 2^{(n+1)/2} \sum_{m = n+1}^{\infty}2^{-m(1+\alpha)} \le \frac{c\cdot 2^{-(n+1)(\alpha+\frac{1}{2})}}{1-2^{-\alpha-1}}.
		\end{split}
	\end{equation}
	{The case $k = 1$} is an immediate consequence of the above inequality. This completes the proof.
\end{proof}

\begin{proof}[Proof of \Cref{thm functional error all ps}]
	We start with the case $p=1$. For $m<n$, \Cref{thm Faber errors ellinf ellp} and 
	the fact that 	$
	\norm{e_{m,k}}_{L_1[0,1]}= 2^{-3m/2-2}
	$  yield that
	\begin{equation}\label{eq error ell 1 1}
		\begin{split}
			\bnorm{f_m - \hat{f}_{n,m}}_{L_1[0,1]} &\le \bbnorm{\sum_{i = -1}^{m}\sum_{k = 0}^{2^{i\vee 0}-1}(\vartheta_{i,k}-\theta_{i,k})e_{i,k}}_{L_1[0,1]} \le \sum_{i = -1}^{m}\bbnorm{\sum_{k = 0}^{2^{i\vee 0}-1}(\vartheta_{i,k}-\theta_{i,k})e_{i,k}}_{L_1[0,1]} \\& = \sum_{i = 0}^{m}2^{-3i/2-2}\bnorm{\bar{\bm \vartheta}^{(n)}_i - \bar{\bm \theta}_i}_{\ell_1} + \frac{1}{2}\bnorm{\bar{\bm \vartheta}^{(n)}_{-1} - \bar{\bm \theta}_{-1}}_{\ell_1} \\& \le \sum_{i = 0}^{m}2^{-i-2 -(n+3)/2}\norm{\bm z_{n+1}}_{\ell_1} + 2^{-(n+5)/2}\norm{\bm z_{n+1}}_{\ell_1} = 2^{-(n+3)/2}\norm{\bm z_{n+1}}_{\ell_1}.
		\end{split}
	\end{equation}
	For the case $m = n$, we get in the same manner $\bnorm{f_n - \hat{f}_{n,n}}_{L_1[0,1]}\le 2^{-(n+1)/2}\norm{\bm z_{n+1}}_{\ell_1}$. 
	To relate these two inequalities with the second-order modulus of continuity, we recall from \eqref{eq def zn 2} that $|z^{(n+1)}_i |\le  2^{(n+1)/2}\sum_{m = n+1}^{\infty} \omega_2(f,2^{-m})$. From here, the proof of the case $p=1$ is easily completed.

	For $p=2$, we have $	\norm{e_{m,k}}_{L_2[0,1]}  = 2^{-m-1}/{\sqrt{3}}$ and  $e_{m,k}\cdot e_{m,\ell} = 0$ for $k\neq\ell$. Thus, 
	\begin{equation}\label{eq fm l2 error}
		\begin{split}
			\bnorm{f_m - \hat{f}_{n,m}}_{L_2[0,1]} &\le \bbnorm{\sum_{i = -1}^{m}\sum_{k = 0}^{2^{i\vee 0}-1}(\vartheta_{i,k}-\theta_{i,k})e_{i,k}}_{L_2[0,1]} \le \sum_{i = -1}^{m}\bbnorm{\sum_{k = 0}^{2^{i\vee 0}-1}(\vartheta_{i,k}-\theta_{i,k})e_{i,k}}_{L_2[0,1]} \\& = \frac{1}{\sqrt{3}}\sum_{i = -1}^{m}2^{-i-1}\bnorm{\bar{\bm \vartheta}^{(n)}_i - \bar{\bm \theta}_i}_{\ell_2} \le \frac{1}{\sqrt{3}}\sum_{i = -1}^{m}2^{-i-2}\norm{\bm z_{n+1}}_{\ell_2} \le \frac{\norm{\bm z_{n+1}}_{\ell_2}}{\sqrt{3}}.
		\end{split}
	\end{equation}
	Next, we can argue as in the proof for $p=1$ to get $\norm{\bm z_{n+1}}_{\ell_2} \le 2^{n+1}\sum_{m = n+1}^{\infty} \omega_2(f,2^{-m})$, which gives the first inequality. The proof of the second one is analogous.

	For $p=\infty$, we have $\norm{e_{-1,0}}_{L_\infty[0,1]} = 1$ and $\norm{e_{m,k}}_{L_\infty[0,1]} = 2^{-m/2-1}$ for $m \in \bN_0$ and $0 \le k \le 2^m-1$. The triangle inequality then yields that for $0 \le m \le n-1$ 
	\begin{equation}\label{eq fm linf error}
		\begin{split}
			\bnorm{f_m - \hat{f}_{n,m}}_{L_\infty[0,1]} &\le \bbnorm{\sum_{i = -1}^{m}\sum_{k = 0}^{2^{i\vee 0}-1}(\vartheta_{i,k}-\theta_{i,k})e_{i,k}}_{L_\infty[0,1]} \le \sum_{i = -1}^{m}\bbnorm{\sum_{k = 0}^{2^{i\vee 0}-1}(\vartheta_{i,k}-\theta_{i,k})e_{i,k}}_{L_\infty[0,1]} \\& = \sum_{i = 0}^{m}2^{-i/2-1}\bnorm{\bar{\bm \vartheta}^{(n)}_i - \bar{\bm \theta}_i}_{\ell_\infty} + \bnorm{\bar{\bm \vartheta}^{(n)}_{-1} - \bar{\bm \theta}_{-1}}_{\ell_\infty} \\& \le \sum_{i = 0}^{m}2^{-i/2-1 + (n-i-1)/2}\norm{\bm z_{n+1}}_{\ell_\infty} + 2^{(n-1)/2}\norm{\bm z_{n+1}}_{\ell_\infty} \\&= \left(2^{(n-3)/2}\sum_{i = 0}^{m} 2^{-i} + 2^{(n-1)/2}\right)\norm{\bm z_{n+1}}_{\ell_\infty} \le 2^{(n+1)/2}\norm{\bm z_{n+1}}_{\ell_\infty},
		\end{split}
	\end{equation}
	where we have used \Cref{thm Faber errors ellinf ellp} in the fourth step. For the case $m = n$, an analogous decomposition leads to 
	\begin{equation}\label{eq norm fnn inf}
		\begin{split}
			\bnorm{f_n - \hat{f}_{n,n}}_{L_\infty[0,1]} &\le\bbnorm{\big(f_{n-1} - \hat{f}_{n,n-1}\big) + \sum_{k = 0}^{2^n-1}\big(\theta_{n,k}-\vartheta^{(n)}_{n,k}\big)e_{n,k}}_{L_\infty[0,1]} \\&\le \bnorm{f_{n-1} - \hat{f}_{n,n-1}}_{L_\infty[0,1]} + 2^{-n/2-1}\bnorm{\bar{\bm \vartheta}^{(n)}_n - \bar{\bm \theta}_n}_{\ell_\infty}\\& \le 2^{(n+1)/2}\norm{\bm z_{n+1}}_{\ell_\infty} + \big(\frac{2^{n + \frac{3}{2}} - \sqrt{2}}{2^{n/2+1}}\big)\norm{\bm z_{n+1}}_{\ell_\infty}  \le 2^{(n+3)/2}\norm{\bm z_{n+1}}_{\ell_\infty}.
		\end{split}
	\end{equation}
	It follows from \eqref{eq def zn 2} that 
	\begin{equation}\label{eq z ell inf}
		\norm{\bm z_{n+1}}_{\ell_\infty} \le 2^{(n+1)/2}\sum_{k = n+1}^{\infty}\omega_{2}(f,2^{-k}),
	\end{equation}
	and plugging this inequality back into \eqref{eq fm linf error} and \eqref{eq norm fnn inf} completes the proof.
\end{proof}

The proof of \Cref{thm functional errors} will follow immediately from the next lemma. Indeed,  the quadratic spline interpolation $\hat{F}_{n,n}$ coincides with $F$ on the dyadic partition $\bT_{n+1}$,  and so $\hat{F}_{n,n} - F$ belongs to the  linear space $\cX_{n+1}$ introduced in \eqref{eq Xn space} below. 

\begin{lemma}\label{Lemma L2 norm}
	Let us consider a vector space $\cX_n$ of functions, such that 
	\begin{equation}\label{eq Xn space}
		\cX_n := \left\{F \in C^1[0,1]\bigg|F\left(\frac{k}{2^n}\right)=0\quad \text{for} \quad 0 \le k \le 2^n\right\} \quad \text{for} \quad n \in \bN_0
	\end{equation}
	Moreover, we also consider the vector space $\cY_n$, where 
	\begin{equation*}
		\cY_n := \left\{f \in C[0,1]\bigg| \int_{0}^{\bm \cdot} f(s)\,ds \in \cX_n\right\} \quad \text{for} \quad n \in \bN_0.
	\end{equation*}
	Furthermore, by $\cX_n^1$, $\cX_n^2$ and $\cX_n^\infty$, we denote the normed linear space $\cX_n$ equipped with the $L_1[0,1]$, $L_2[0,1]$ and $L_\infty[0,1]$ norms respectively, as well as, the analogously defined normed spaces $\cY_n^1$, $\cY_n^2$ and $\cY_n^\infty$. Let $T: \cY_n \rightarrow \cX_n$ denote the integral operator, i.e., $Tf = \int_0^{\bm \cdot}f(s)\,ds$. Then, its operator norms are given by $	\norm{T}_{\cY^p_n \rightarrow\cX^p_n} = 2a_p$, where $a_p$ is as in \Cref{thm functional errors}.
\end{lemma}

\begin{proof}
	First, we consider the case $p=\infty$.   Since $F(2^{-n}k) = 0$ for all $k$,
	\begin{equation}\label{eq Linf upper bound}
		\begin{split}
			\norm{F}_{L_\infty[0,1]} &= \max_{1 \le k \le 2^n} \max_{\frac{k-1}{2^n} \le t \le \frac{k}{2^n}} |F(t)| \le\max_{1 \le k \le 2^n} \max_{\frac{k-1}{2^n} \le t \le \frac{k}{2^n}}\int_{\frac{k-1}{2^n}}^t |f(s)|\, ds\le 2^{-n-1} \norm{f}_{L_\infty[0,1]}.
		\end{split}
	\end{equation}
	This shows that $\norm{T}_{\cY^\infty_n \rightarrow\cX^\infty_n} \le 2^{-n-1}$. To get the converse inequality, we take $\eps\in(0,2^{-n-3})$ and consider the function $f_\eps$ that  is equal to 1 on intervals of the form $[\frac{k}{2^{n}}+\varepsilon,  \frac{2k+1}{2^{n+1}} - \varepsilon]$, equal to zero on intervals of the form $[\frac{2k+1}{2^{n+1}}+\varepsilon, \frac{k+1}{2^{n+1}} - \varepsilon]$ and linearly interpolated everywhere else. 
	Then $f_\varepsilon \in \cX_n$ and we have $\norm{f_\varepsilon}_{L_\infty[0,1]} = 1$ and $\norm{Tf_{\varepsilon}}_{L_\infty[0,1]} = 2^{-n-1} - 2\varepsilon$, and the proof for $p=\infty$ is complete.

	For the  $L_2[0,1]$ norm, let us take any $f\in\cY_n$ with the corresponding $F:=Tf\in\cX_n$ and define $F_k(t) = F(2^{-n}(t + k))$ for $t \in [0,1]$. Then 
	\begin{equation}\label{eq ell 2 F norm}
		\norm{F}^2_{L_2[0,1]} =  \sum_{k = 0}^{2^{n}-1}\int_{2^{-n}k}^{2^{-n}(k+1)}F^2(t) \, dt = 2^{-n}\sum_{k = 0}^{2^n-1}\int_{0}^{1}F^2_k(t)\, dt = 2^{-n}\sum_{k = 0}^{2^n-1}\norm{F_k}^2_{L_2[0,1]}.
	\end{equation}
	Furthermore, let us denote $f_k(t):= F'_k(t)$. Then  $f_k(t) = 2^{-n}f(2^{-n}(t+k))$ and so
	\begin{equation*}
		\sum_{k = 0}^{2^n-1}\norm{f_k}^2_{L_2[0,1]} = \sum_{k = 0}^{2^n-1}\int_{0}^{1}f_k^2(t)\, dt = 2^{-2n}\sum_{k = 0}^{2^n-1}\int_{0}^{1}f^2(2^{-n}(t+k))\, dt = 2^{-n}\int_{0}^{1}f^2(t) \, dt = 2^{-n}\norm{f}^2_{L_2[0,1]}.
	\end{equation*}
	Next, for $0 \le k \le 2^n-1$, in the spirit of Halmos \cite[Problem 188]{Halmos2012HilbertSpace}, the H\"older inequality yields that 
	\begin{equation}\label{eq Tflef in L2}
		\begin{split}
			\norm{F_k}^2_{L_2[0,1]} &= \int_{0}^{1}F^2_k(t)\, dt =  \int_{0}^{1}\left(\int_{0}^{t}f_k(s)\,ds\right)^2\, dt = \int_{0}^{1}\left(\int_{0}^{t}\sqrt{\cos \frac{\pi s}{2}}\frac{f_k(s)}{\sqrt{\cos \frac{\pi s}{2}}}ds\right)^2 \, dt \\ &\le \int_{0}^{1} \left(\int_{0}^{t} \cos \frac{\pi s}{2} \, ds \int_{0}^{t} \frac{f_k^2(s)}{\cos \frac{\pi s}{2}  }\,ds\right)\, dt = \frac{2}{\pi} \int_{0}^{1}\sin \frac{\pi t}{2} \int_{0}^{t} \frac{f_k^2(s)}{\cos \frac{\pi s}{2}}\, dsdt \\&= \frac{2}{\pi} \int_{0}^{1}\frac{f_k^2(s)}{\cos \frac{\pi s}{2}} \int_{s}^{1} \sin \frac{\pi t}{2}\, dtds = \frac{4}{\pi^2}\int_{0}^{1}\frac{f_k^2(s)}{\cos \frac{\pi s}{2}} (-\cos\frac{\pi}{2} + \cos\frac{\pi s}{2})\, ds \\&= \frac{4}{\pi^2}\norm{f_k}^2_{L_2[0,1]}.
		\end{split}
	\end{equation}
	Thus, $\norm{Tf}_{L_2[0,1]} \le \frac{2^{-n+1}}{\pi}\norm{f}_{L_2[0,1]}$.
	
	To prove the converse inequality, note that the inequality \eqref{eq Tflef in L2} becomes an identity for 
	$f_k(t):= \cos \frac{\pi t}{2}$. To see why, note first that $F_k(t) :=Tf_k(t)= \frac{2}{\pi}\sin \frac{\pi t}{2}$ satisfies $	\norm{F_k}^2_{L_2[0,1]} = \frac{4}{\pi^2} \int_{0}^{1}(\sin \frac{\pi t}{2} )^2\, dt = \frac{2}{\pi^2}$ and $
	\norm{f_k}^2_{L_2[0,1]} = \int_{0}^{1}(\cos \frac{\pi t}{2})^2 \, dt =  \frac{1}{2}$.
	In such a way, one sees that the function $f(t) = \cos 2^{n-1}\pi t$ belongs to $\cY_n$, and $\norm{Tf}_{L_2[0,1]} = \frac{2^{-n+1}}{\pi}\norm{f}_{L_2[0,1]}$. This completes the proof for $p = 2$. 
	
	Finally,  the case for $p = 1$  is similar to the case for $p = 2$. We get as in \eqref{eq ell 2 F norm} that 
	$
	\norm{F}_{L_1[0,1]} =   2^{-n}\sum_{k = 0}^{2^n-1}\norm{F_k}_{L_1[0,1]}$, where $F_k$ and $f_k=F_k'$ are as above. In the same way, we get 
	$\sum_{k = 0}^{2^n-1}\norm{f_k}_{L_1[0,1]} = \norm{f}_{L_1[0,1]}$.
	Moreover, for each $0 \le k \le 2^n-1$, 	\begin{equation*}
		\norm{F_k}_{L_1[0,1]} = \int_{0}^{1}|F_k(t)|\, dt \le \int_{0}^{1}\int_{0}^{t}|f_k(s)|\, dt \le \norm{f_k}_{L_1[0,1]}.
	\end{equation*}
	Plugging this inequality back to the above formulae	shows that  $\norm{T}_{\cY^1_n \rightarrow\cX^1_n} \le 2^{-n}$. To show the reverse inequality, we take $\varepsilon < \frac{1}{2}$ and $m \ge 1$ and define the function $g_{m,\varepsilon} \in C[0,1]$ as follows,
	$$g_{m,\varepsilon}=\frac{m t }{\varepsilon^2}\Ind{[0,\frac{\varepsilon}{m}]}+\frac{1}{\varepsilon}  \Ind{(\frac{\varepsilon}{m}, \varepsilon]}-\frac{m}{\varepsilon^2}\Big(t - \frac{(m+1)\varepsilon}{m}\Big)\Ind{(\varepsilon, \frac{(m+1)\varepsilon}{m}]},
	$$
	where $\Ind{A}$ denotes the indicator function of a set $A$.
	In other words, the function $g_{m, \varepsilon}$ defines an isosceles Trapezoid with height ${1}/{\varepsilon}$. Thus, $\norm{g_{m ,\varepsilon}}_{L_1[0,1]} = 1$ for any $\varepsilon < {1}/{2}$ and $m \ge 1$. Clearly, $G_{m,\varepsilon}:=Tg_{m,\varepsilon}$ is given by
	$$G_{m,\varepsilon}(t)= \frac{m t^2 }{2\varepsilon^2}\Ind{[0,\frac{\varepsilon}{m}]}+\bigg(\frac{1}{\varepsilon}\Big(t - \frac{\varepsilon}{m}\Big) + \frac{1}{2m}\bigg)  \Ind{(\frac{\varepsilon}{m}, \varepsilon]}+\bigg(\frac{2m-1}{2m}-\frac{m}{2\varepsilon^2}\big(t - \frac{(m+1)\varepsilon}{m}\big)^2\bigg)\Ind{(\varepsilon, \frac{(m+1)\varepsilon}{m}]}+\Ind{ (\frac{(m+1)\varepsilon}{m},1]}.
	$$
	Thus, we have $\norm{G_{m,\varepsilon}}_{L_1[0,1]} \ge 1 - \frac{(m+1)\varepsilon}{m}$. 
	Note that for any pair of $m$ and $\varepsilon$, we have $g_{m,\varepsilon}(0) = g_{m,\varepsilon}(1) = 0$. Thus, can define a function $f\in C[0,1]$ by $f(t) := 2^ng_{m,\eps} (2^nt – k)$ for $t \in [2^{-n}k, 2^{-n}(k+1)]$ so that  $f_{k} = g_{m,\varepsilon}$ for all $k$. Furthermore, we have 
	\begin{equation*}
		\norm{F}_{L_1[0,1]} = 2^{-n}\sum_{k = 0}^{2^n-1}\norm{F_k}_{L_1[0,1]} \ge 2^{-n}\left(1 - \frac{(m+1)\varepsilon}{m}\right) = 2^{-n}\Big(1 - \frac{(m+1)\varepsilon}{m}\Big)\norm{f}_{L_1[0,1]}.
	\end{equation*}
	Taking $m \ua \infty$ and $\varepsilon \da 0$ shows that $\norm{T}_{\cY^1_n \rightarrow\cX^1_n} = 2^{-n}$. This completes the proof.
\end{proof}

\begin{proof}[Proof of \Cref{lemma Faber antiderivative}]
	We begin with proving \eqref{eq vartheta Faber minus}. Rewriting \eqref{eq zeta minus} yields that
	\begin{equation*}
		\begin{split}
			\vartheta^{(n)}_{-1,0} &= 2^{n+2}\sum_{j = 1}^{2^{n+1}}(-1)^j\left(F\Big(\frac{j}{2^{n+1}}\Big) - F\Big(\frac{j-1}{2^{n+1}}\Big)\right) \\
			&= 2^{n+2}\sum_{j = 0}^{2^n-1}\left(F\Big(\frac{2j}{2^{n+1}}\Big)+F\Big(\frac{2j+2}{2^{n+1}}\Big)-2F\Big(\frac{2j+1}{2^{n+1}}\Big)\right)
			= -2^{n/2+2} \sum_{j = 0}^{2^n-1} \rho_{n,j}.
		\end{split}
	\end{equation*}
	Now, we proceed to prove \eqref{eq vartheta Faber}. To this end, note that
	\begin{equation*}
		\scalemath{0.8}{\begin{split}
				\vartheta^{(n)}_{m,k} &= 2^{n+m/2+2}\sum_{j = 1}^{2^{n-m}}(-1)^j\left(F\Big(\frac{k}{2^m}+\frac{j}{2^{n+1}}\Big)-F\Big(\frac{k}{2^m}+\frac{j-1}{2^{n+1}}\Big)+F\Big(\frac{k+1}{2^m}-\frac{j-1}{2^{n+1}}\Big)-F\Big(\frac{k+1}{2^m}-\frac{j}{2^{n+1}}\Big)\right)\\&=2^{n+m/2+2}\sum_{j = 0}^{2^{n-m-1}-1}\left(F\Big(\frac{k}{2^m}+\frac{2j+2}{2^{n+1}}\Big)+F\Big(\frac{k}{2^m}+\frac{2j}{2^{n+1}}\Big)-2F\Big(\frac{k}{2^m}+\frac{2j+1}{2^{n+1}}\Big)\right)\\&\quad-2^{n+m/2+2}\sum_{j = 1}^{2^{n-m-1}}\left(F\Big(\frac{k+1}{2^m}-\frac{2j}{2^{n+1}}\Big)+F\Big(\frac{k+1}{2^m}+\frac{2j-2}{2^{n+1}}\Big)-2F\Big(\frac{k+1}{2^m}-\frac{2j-1}{2^{n+1}}\Big)\right)\\&= 2^{n+m/2+2}\sum_{j = 0}^{2^{n-m-1}-1}\left(F\Big(\frac{2(j+2^{n-m}k)+2}{2^{n+1}}\Big)+F\Big(\frac{2(j+2^{n-m}k)}{2^{n+1}}\Big)-2F\Big(\frac{2(j+2^{n-m}k)+1}{2^{n+1}}\Big)\right)\\&\quad-2^{n+m/2+2}\sum_{j = 0}^{2^{n-m-1}-1}\left(F\Big(\frac{2(j+2^{n-m}k)+2}{2^{n+1}} + \frac{1}{2^{m+1}}\Big)+F\Big(\frac{2(j+2^{n-m}k)}{2^{n+1}} + \frac{1}{2^{m+1}}\Big)-2F\Big(\frac{2(j+2^{n-m}k)+1}{2^{n+1}}+\frac{1}{2^{m+1}}\Big)\right)\\&=  2^{(n+m)/2 + 2}\sum_{j = 0}^{2^{n-m-1}-1} \left(\rho_{n,2^{n-m}k + 2^{n-m-1} + j} - \rho_{n,2^{n-m}k + j}\right).	\end{split}}
	\end{equation*}
	Finally, we will proceed to prove \eqref{eq integral representation}. Applying the fundamental theorem of calculus to the above equation yields
	\begin{equation*}
		\scalemath{0.8}{\begin{split}
				\vartheta^{(n)}_{m,k} &= 2^{n+m/2+2}\sum_{j = 0}^{2^{n-m-1}-1}\Bigg[2\left(F\Big(\frac{2(2^{n-m}k  + j)+1}{2^{n+1}} + \frac{1}{2^{m+1}}\Big)-F\Big(\frac{2(j+2^{n-m}k)+1}{2^{n+1}}\Big)\right)\\&\quad - \left(F\Big(\frac{2(j+2^{n-m}k)}{2^{n+1}} + \frac{1}{2^{m+1}}\Big) - F\Big(\frac{2(j+2^{n-m}k)}{2^{n+1}}\Big)\right)- \left(F\Big(\frac{2(j+2^{n-m}k)+2}{2^{n+1}} + \frac{1}{2^{m+1}}\Big) - F\Big(\frac{2(j+2^{n-m}k)+2}{2^{n+1}}\Big)\right)\Bigg]\\&= 2^{n-m/2+1}\sum_{j = 0}^{2^{n-m-1}-1}\int_0^1\Bigg[ 2f\left(\frac{2(j+2^{n-m}k)+1}{2^{n+1}} + \frac{s}{2^{m+1}}\right) - f\left(\frac{j+2^{n-m}k}{2^n} + \frac{s}{2^{m+1}}\right) - f\left(\frac{j+2^{n-m}k+1}{2^n} + \frac{s}{2^{m+1}}\right) \Bigg]\,ds\\&= 2^{(n-m)/2+1}\sum_{j = 0}^{2^{n-m-1}-1}\int_0^1 \theta_{n,2^{n-m}k+j}\Big(\frac{s}{2^{m+1}}\Big)\,ds.
		\end{split}}
	\end{equation*}
	This completes the proof.
\end{proof}

\begin{proof}[Proof of \Cref{thm element minus} for $m = -1$]
	We prove \eqref{eq element minus bound holder} by discussing case by case for $i \in \{0,1,2\}$. To this end, applying \eqref{eq vartheta Faber} yields that 
	\begin{equation}\label{eq abs diff vartheta}
		\scalemath{0.9}{\begin{split}
				\vartheta^{(n+1)}_{-1,0} - \vartheta^{(n)}_{-1,0} &= 2^{n/2+2}\sum_{j = 0}^{2^n-1}\rho_{n,j}-2^{(n+1)/2+2}\sum_{j = 0}^{2^{n+1}-1}\rho_{n+1,j}\\&= 2^{n/2+2}\sum_{i = 0}^{2^n-1}\rho_{n,j} - 2^{(n+1)/2+2}\sum_{i = 0}^{2^n-1}\big(\rho_{n,2j} + \rho_{n,2j+1}\big)\\&= 2^{n+2}\sum_{j = 0}^{2^n-1}\left(F\Big(\frac{4j}{2^{n+2}}\Big)-4F\Big(\frac{4j+1}{2^{n+2}}\Big)+6F\Big(\frac{4j+2}{2^{n+2}}\Big)-4F\Big(\frac{4j+3}{2^{n+2}}\Big)+F\Big(\frac{4j+4}{2^{n+2}}\Big)\right).
		\end{split}}
	\end{equation}
	Let us now prove for the case $i = 0$. Applying the fundamental theorem of calculus to the right-hand side of \eqref{eq abs diff vartheta} yields that
	\begin{equation}\label{eq first abs diff vartheta}
		\scalemath{0.85}{
			\begin{split}
				\vartheta^{(n+1)}_{-1,0} - \vartheta^{(n)}_{-1,0} &= \sum_{j = 0}^{2^n-1}\int_0^1 \left(-f\Big(\frac{4j+s}{2^{n+2}}\Big) + 3f\Big(\frac{4j+1+s}{2^{n+2}}\Big) - 3f\Big(\frac{4j+2+s}{2^{n+2}}\Big)+f\Big(\frac{4j+3+s}{2^{n+2}}\Big)\right)\,ds\\&= \sum_{j = 0}^{2^n-1}\Bigg[\int_0^1 \left(f\Big(\frac{4j+1+s}{2^{n+2}}\Big)-f\Big(\frac{4j+s}{2^{n+2}}\Big)\right)\,ds - 2\int_0^1 \left(f\Big(\frac{4j+2+s}{2^{n+2}}\Big)-f\Big(\frac{4j+1+s}{2^{n+2}}\Big)\right)\,ds \\&\quad+ \int_0^1 \left(f\Big(\frac{4j+3+s}{2^{n+2}}\Big)-f\Big(\frac{4j+2+s}{2^{n+2}}\Big)\right)\,ds\Bigg].
		\end{split}}
	\end{equation}
	For the case $i = 0$, we have $f \in C^{0,\alpha}[0,1]$, and there exists $\kappa_0> 0$ such that 
	\begin{equation}\label{eq f holder}
		\left|f\Big(\frac{4j+s+1}{2^{n+2}}\Big) - f\Big(\frac{4j+s}{2^{n+2}}\Big)\right| \le \kappa_0 \cdot 2^{-\alpha(n+2)} \qquad \text{for} \quad 0 \le j \le 2^n-1 \quad \text{and} \quad s \in [0,3].
	\end{equation} 
	Hence, applying \eqref{eq f holder} and the triangular inequality to \eqref{eq first abs diff vartheta} gives
	\begin{equation*}
		\big|\vartheta^{(n+1)}_{-1,0} - \vartheta^{(n)}_{-1,0}\big| \le \kappa_0 \cdot 2^{(1-\alpha)n-2\alpha+2}.	
	\end{equation*}
	We now proceed to prove for the case $i = 1$. Further applying the fundamental theorem of calculus to the right-hand side of \eqref{eq first abs diff vartheta} gives 
	\begin{equation}\label{eq second abs diff vartheta}
		\scalemath{0.85}{\begin{split}
				\vartheta^{(n+1)}_{-1,0} - \vartheta^{(n)}_{-1,0} &= \sum_{j = 0}^{2^n-1}\int_0^1 \left(-f\Big(\frac{4j+s}{2^{n+2}}\Big) + 3f\Big(\frac{4j+1+s}{2^{n+2}}\Big) - 3f\Big(\frac{4j+2+s}{2^{n+2}}\Big)+f\Big(\frac{4j+3+s}{2^{n+2}}\Big)\right)\,ds \\&= 2^{-n-2}\sum_{j = 0}^{2^n-1}\int_0^1\int_0^1 \left(f'\left(\frac{4j+s+t}{2^{n+2}}\right) -2 f'\left(\frac{4j+s+t+1}{2^{n+2}}\right) + f'\left(\frac{4j+s+t+2}{2^{n+2}}\right)\,dsdt\right)\\&= 2^{-n-2}\sum_{j = 0}^{2^n-1}\Bigg[-\int_0^1\int_0^1f'\left(\frac{4j+s+t+1}{2^{n+2}}\right) - f'\left(\frac{4j+s+t}{2^{n+2}}\right)\,dsdt \\&\quad+ \int_0^1\int_0^1f'\left(\frac{4j+s+t+2}{2^{n+2}}\right) - f'\left(\frac{4j+s+t+1}{2^{n+2}}\right)\,dsdt\Bigg].
		\end{split}}
	\end{equation}
	Since $f \in C^{1,\alpha}[0,1]$, we have $f' \in C^{0,\alpha}[0,1]$. Thus, there exists $\kappa_1 > 0$ such that 
	\begin{equation}\label{eq f first holder}
		\left|f'\Big(\frac{4j+s+1}{2^{n+2}}\Big) - f'\Big(\frac{4j+s}{2^{n+2}}\Big)\right| \le \kappa_1 \cdot 2^{-\alpha(n+2)} \qquad \text{for} \quad 0 \le j \le 2^n-1 \quad \text{and} \quad s \in [0,3].
	\end{equation} 
	Analogously, applying \eqref{eq f first holder} and the triangular inequality to \eqref{eq second abs diff vartheta} yields 
	\begin{equation*}
		\big|\vartheta^{(n+1)}_{-1,0} - \vartheta^{(n)}_{-1,0}\big| \le \kappa_1 \cdot 2^{-\alpha n - 2\alpha -1}.
	\end{equation*}
	Next, we prove for the case $i = 2$. Now, further applying the fundamental theorem of calculus to the right-hand side of \eqref{eq second abs diff vartheta} gives
	\begin{equation}\label{eq third abs diff vartheta}
		\scalemath{0.85}{
			\begin{split}
				\vartheta^{(n+1)}_{-1,0} - \vartheta^{(n)}_{-1,0}&= 2^{-n-2}\sum_{j = 0}^{2^n-1}\int_0^1\int_0^1 \left(f'\left(\frac{4j+s+t}{2^{n+2}}\right) -2 f'\left(\frac{4j+s+t+1}{2^{n+2}}\right) + f'\left(\frac{4j+s+t+2}{2^{n+2}}\right)\,dsdt\right)\\&= 2^{-2n-4} \sum_{j = 0}^{2^n-1}\int_0^1\int_0^1\int_0^1\left(f''\left(\frac{4j+s+t+u+1}{2^{n+2}}\right)-f''\left(\frac{4j+s+t+u}{2^{n+2}}\right)\right)\,dsdtdu.
		\end{split}}
	\end{equation}
	Analogously, for $i = 2$, we now have $f'' \in C^{0,\alpha}[0,1]$. Thus, there exists $\kappa_2 > 0$ such that 
	\begin{equation}\label{eq f second holder}
		\left|f''\Big(\frac{4j+s+1}{2^{n+2}}\Big) - f''\Big(\frac{4j+s}{2^{n+2}}\Big)\right| \le \kappa_2 \cdot 2^{-\alpha(n+2)} \qquad \text{for} \quad 0 \le j \le 2^n-1 \quad \text{and} \quad s \in [0,3].
	\end{equation} 
	Applying \eqref{eq f second holder} and the triangular inequality to \eqref{eq third abs diff vartheta} gives
	\begin{equation*}
		\big|\vartheta^{(n+1)}_{-1,0} - \vartheta^{(n)}_{-1,0}\big| \le \kappa_2 \cdot 2^{-(1+\alpha) n-2\alpha -4}.
	\end{equation*}
	Now, taking $\kappa = 2^{-2\alpha+2}\kappa_0 \vee 2^{-2\alpha-1}\kappa_1 \vee 2^{-2\alpha-4}\kappa_2$ yields back \eqref{eq element minus bound holder}, and this completes the proof for \eqref{eq element minus bound holder}. Next, we proceed to prove \eqref{eq element minus bound holder 2}. Further applying the fundamental theorem of calculus to the right-hand side of \eqref{eq third abs diff vartheta} yields 
	\begin{equation}\label{eq fourth abs diff vartheta}
		\scalemath{0.85}{
			\begin{split}
				\vartheta^{(n+1)}_{-1,0} - \vartheta^{(n)}_{-1,0}&= 2^{-2n-4} \sum_{j = 0}^{2^n-1}\int_0^1\int_0^1\int_0^1\left(f''\left(\frac{4j+s+t+u+1}{2^{n+2}}\right)-f''\left(\frac{4j+s+t+u}{2^{n+2}}\right)\right)\,dsdtdu\\&= 2^{-3n -6}\sum_{j = 0}^{2^n-1}\int_0^1\int_0^1\int_0^1\int_0^1 f'''\left(\frac{4j+s+t+u+v}{2^{n+2}}\right)\,dsdtdudv.
		\end{split}}
	\end{equation}
	Applying the triangular inequality to \eqref{eq fourth abs diff vartheta} gives 
	\begin{equation*}
		\big|\vartheta^{(n+1)}_{-1,0} - \vartheta^{(n)}_{-1,0}\big| \le 2^{-3n -6}\sum_{j = 0}^{2^n-1}\int_0^1\int_0^1\int_0^1\int_0^1 \sup_{t \in [0,1]}\left|f'''(t)\right|\,dsdtdudv = 2^{-2n-6}\sup\limits_{t \in [0,1]}|f'''(t)|.
	\end{equation*}
	This completes the proof of \Cref{thm element minus} for $m = -1$.
\end{proof}

\begin{proof}[Proof of \Cref{thm element minus} for $0 \le m \le n-1$]
	Without loss of generality, we assume that $k = 0$. For $1 \le k \le 2^m-1$, the result can be proved analogously. First, we prove \eqref{eq element bound holder} by discussing case by case for $i \in \{0,1,2,3\}$. Applying \eqref{eq integral representation} yields that 
	\begin{equation}\label{eq element bound representation}
		\scalemath{0.85}{
			\begin{split}
				\vartheta^{(n)}_{m,0} - \vartheta^{(n+1)}_{m,0} &= \sum_{j = 0}^{2^{n-m-1}-1}\int_0^{1}2^{(n-m)/2+1} \theta_{n, j}\left(\frac{r}{2^{m+1}}\right) - 2^{(n+1-m)/2+1}\left[\theta_{n+1, 2j}\left(\frac{r}{2^{m+1}}\right)+\theta_{n+1, 2j+1}\left(\frac{r}{2^{m+1}}\right)\right]\,dr\\&= 2^{n-m/2+1}\sum_{j = 0}^{2^{n-m-1}-1}\int_0^{1}\Bigg[f\left(\frac{4j}{2^{n+2}}+\frac{r}{2^{m+1}}\right)-4f\left(\frac{4j+1}{2^{n+2}}+\frac{r}{2^{m+1}}\right)+6f\left(\frac{4j+2}{2^{n+2}}+\frac{r}{2^{m+1}}\right)\\&\quad-4f\left(\frac{4j+3}{2^{n+2}}+\frac{r}{2^{m+1}}\right)+f\left(\frac{4j+4}{2^{n+2}}+\frac{r}{2^{m+1}}\right)\Bigg]\,dr.
		\end{split}}
	\end{equation}
	Applying the triangular inequality to the above equation gives 
	\begin{equation}\label{eq f0 holder}
		\scalemath{0.8}{
			\begin{split}
				\big|\vartheta^{(n)}_{m,0} - \vartheta^{(n+1)}_{m,0}\big|  &\le 2^{n-m/2+1}\sum_{j = 0}^{2^{n-m-1}-1}\int_0^{1}\Bigg[\left|f\left(\frac{4j+1}{2^{n+2}}+\frac{r}{2^{m+1}}\right) - f\left(\frac{4j}{2^{n+2}}+\frac{r}{2^{m+1}}\right)\right|  \\&\quad+3\left|f\left(\frac{4j+2}{2^{n+2}}+\frac{r}{2^{m+1}}\right) - f\left(\frac{4j+1}{2^{n+2}}+\frac{r}{2^{m+1}}\right)\right|+3\left|f\left(\frac{4j+3}{2^{n+2}}+\frac{r}{2^{m+1}}\right) - f\left(\frac{4j+2}{2^{n+2}}+\frac{r}{2^{m+1}}\right)\right| \\&\quad+ \left|f\left(\frac{4j+4}{2^{n+2}}+\frac{r}{2^{m+1}}\right) - f\left(\frac{4j+3}{2^{n+2}}+\frac{r}{2^{m+1}}\right)\right|\Bigg]\,dr.
		\end{split}}
	\end{equation}
	Since $f \in C^{0,\alpha}[0,1]$, then for $s \in [0,3]$ and $r \in [0,1]$, there exists $\kappa_0 > 0$ such that
	\begin{equation*}
		\left|f\left(\frac{4j+s+1}{2^{n+2}}+\frac{r}{2^{m+1}}\right) - f\left(\frac{4j+s}{2^{n+2}}+\frac{r}{2^{m+1}}\right)\right| \le \kappa_0 \cdot 2^{-\alpha(n+2)}.
	\end{equation*}
	Applying the above inequality to \eqref{eq f0 holder} gives 
	\begin{equation*}
		\big|\vartheta^{(n+1)}_{m,0} - \vartheta^{(n)}_{m,0}\big| \le \kappa_0 \cdot 2^{(2-\alpha)n-3m/2 + 4 - 2\alpha}.
	\end{equation*}
	We now proceed to prove for the case $i = 1$. Applying the fundamental theorem of calculus to the right-hand side of \eqref{eq element bound representation} yields 
	\begin{equation}\label{eq f1 representation}
		\scalemath{0.95}{\begin{split}
				\vartheta^{(n)}_{m,0} - \vartheta^{(n+1)}_{m,0} &= 2^{-m/2-1}\sum_{j = 0}^{2^{n-m-1}-1}\int_0^1 \int_0^1 \Bigg[-f'\left(\frac{4j+s}{2^{n+2}}+\frac{r}{2^{m+1}}\right)+3f'\left(\frac{4j+1+s}{2^{n+2}}+\frac{r}{2^{m+1}}\right)\\&\quad-3f'\left(\frac{4j+2+s}{2^{n+2}}+\frac{r}{2^{m+1}}\right)+ f'\left(\frac{4j+3+s}{2^{n+2}}+\frac{r}{2^{m+1}}\right) \Bigg]\,drds\\&= 2^{-m/2 - 1}\sum_{j = 0}^{2^{n-m-1}-1}\int_0^1\int_0^1\Bigg[\left[f'\left(\frac{4j+1+s}{2^{n+2}}+\frac{r}{2^{m+1}}\right) - f'\left(\frac{4j+s}{2^{n+2}}+\frac{r}{2^{m+1}}\right)\right]\\&\quad-2\left[f'\left(\frac{4j+2+s}{2^{n+2}}+\frac{r}{2^{m+1}}\right) - f'\left(\frac{4j+1+s}{2^{n+2}}+\frac{r}{2^{m+1}}\right)\right]\\&\quad+\left[f'\left(\frac{4j+3+s}{2^{n+2}}+\frac{r}{2^{m+1}}\right) - f'\left(\frac{4j+2+s}{2^{n+2}}+\frac{r}{2^{m+1}}\right)\right]\Bigg]\,drds.
		\end{split}}
	\end{equation}
	Since $f \in C^{1,\alpha}[0,1]$, then $f' \in C^{0,\alpha}[0,1]$. Thus, for any $s \in [0,3]$ and $r \in [0,1]$, there exists $\kappa_1 > 0$ such that 
	\begin{equation}\label{eq Holder f1}
		\left|f'\left(\frac{4j+1+s}{2^{n+2}}+\frac{r}{2^{m+1}}\right) - f'\left(\frac{4j+s}{2^{n+2}}+\frac{r}{2^{m+1}}\right)\right| \le \kappa_1 \cdot 2^{-\alpha(n+2)}.
	\end{equation}
	Thus, applying the triangular inequality and \eqref{eq Holder f1} to \eqref{eq f1 representation} gives
	\begin{equation*}
		\big|\vartheta^{(n+1)}_{m,0} - \vartheta^{(n)}_{m,0}\big|\le \kappa_1 \cdot 2^{n-3m/2 - \alpha n - 2\alpha + 1}.
	\end{equation*}
	We now proceed to prove for the case $i = 2$. We further apply the fundamental theorem of calculus to \eqref{eq f1 representation}, and this gives
	\begin{equation}\label{eq f2 representation}
		\scalemath{0.85}{
			\begin{split}
				\vartheta^{(n)}_{m,0} - \vartheta^{(n+1)}_{m,0} &= 2^{-n-m/2-3}\sum_{j = 0}^{2^{n-m-1}-1}\int_0^1 \int_0^1\int_0^1\Bigg[f''\left(\frac{4j+s+t}{2^{n+2}} + \frac{r}{2^{m+1}}\right)\\&\quad-2f''\left(\frac{4j+s+t+1}{2^{n+2}} + \frac{r}{2^{m+1}}\right) + f''\left(\frac{4j+s+t+2}{2^{n+2}} + \frac{r}{2^{m+1}}\right)\Bigg]\,drdsdt\\&= 2^{-n-m/2-3}\sum_{j = 0}^{2^{n-m-1}-1}\int_0^1 \int_0^1\int_0^1\Bigg[-\left[f''\left(\frac{4j+1+s+t}{2^{n+2}} + \frac{r}{2^{m+1}}\right)-f''\left(\frac{4j+s+t}{2^{n+2}} + \frac{r}{2^{m+1}}\right) \right]\\&\quad +\left[f''\left(\frac{4j+2+s+t}{2^{n+2}} + \frac{r}{2^{m+1}}\right)-f''\left(\frac{4j+1+s+t}{2^{n+2}} + \frac{r}{2^{m+1}}\right) \right] \Bigg]\,drdsdt.
			\end{split}
		}
	\end{equation}
	Similarly, as $f \in C^{2,\alpha}[0,1]$, for any $s \in [0,3]$ and $r \in [0,1]$, there exists $\kappa_2 > 0$ such that 
	\begin{equation}\label{eq Holder f2}
		\left|f''\left(\frac{4j+1+s}{2^{n+2}}+\frac{r}{2^{m+1}}\right) - f''\left(\frac{4j+s}{2^{n+2}}+\frac{r}{2^{m+1}}\right)\right| \le \kappa_2 \cdot 2^{-\alpha(n+2)}.
	\end{equation}
	Applying the triangular inequality and \eqref{eq Holder f2} to \eqref{eq f2 representation} yields 
	\begin{equation*}
		\big|\vartheta^{(n+1)}_{m,0} - \vartheta^{(n)}_{m,0}\big| \le \kappa_2 \cdot 2^{-\alpha n - 3m/2 -2 -2\alpha}.
	\end{equation*}
	We now proceed to prove for the case $i = 3$. Further applying the fundamental theorem of calculus to \eqref{eq f2 representation} yields 
	\begin{equation}\label{eq f3 representation}
		\scalemath{0.9}{
			\begin{split}
				\vartheta^{(n)}_{m,0}-\vartheta^{(n+1)}_{m,0} &= 2^{-2n-m/2-5}\sum_{j = 0}^{2^{n-m-1}-1}\int_0^1 \int_0^1\int_0^1\int_0^1\Bigg[f'''\left(\frac{4j+1+s+t+u}{2^{n+2}}+\frac{r}{2^{m+1}}\right) \\&\quad- f'''\left(\frac{4j+s+t+u}{2^{n+2}}+\frac{r}{2^{m+1}}\right) \Bigg]\,drdsdtdu.
			\end{split}
		}
	\end{equation}
	As $f \in C^{3,\alpha}[0,1]$, there exists $\kappa_3 > 0$ such that 
	\begin{equation*}
		\left|f'''\left(\frac{4j+1+s}{2^{n+2}}+\frac{r}{2^{m+1}}\right) - f'''\left(\frac{4j+s}{2^{n+2}}+\frac{r}{2^{m+1}}\right)\right| \le \kappa_3 \cdot 2^{-\alpha(n+2)}.
	\end{equation*}
	Applying the above inequality to \eqref{eq f3 representation} gives 
	\begin{equation*}
		\big|\vartheta^{(n+1)}_{m,0} - \vartheta^{(n)}_{m,0}\big| \le \kappa_3 \cdot 2^{-(1+\alpha)n -3m/2 -5 -2\alpha}.
	\end{equation*}
	Now, taking $\kappa := \kappa_0 \cdot 2^{4-2\alpha} \vee \kappa_1 \cdot 2^{1-2\alpha} \vee \kappa_2 \cdot 2^{-2-2\alpha} \vee \kappa_3 \cdot 2^{-5-2\alpha}$ yields \eqref{eq element bound holder}. Finally, let us prove \eqref{eq element bound holder 2}. Applying the fundamental theorem of calculus to \eqref{eq f3 representation} gives 
	\begin{equation*}
		\scalemath{0.9}{\vartheta^{(n)}_{m,0} - \vartheta^{(n+1)}_{m,0} = 2^{-3n-m/2-7}\sum_{j = 0}^{2^{n-m-1}-1}\int_0^1 \int_0^1\int_0^1\int_0^1\int_0^1\left[f^{(4)}\left(\frac{4j+1+s+t+u}{2^{n+2}}+\frac{r}{2^{m+1}}\right)\right]\,drdsdtdudv.}
	\end{equation*} 
	Thus, we have 
	\begin{equation*}
		\big|\vartheta^{(n+1)}_{m,0} - \vartheta^{(n)}_{m,0}\big|\le 2^{-2n-3m/2-8}\sup_{t \in [0,1]} |f^{(4)}(t)|.
	\end{equation*}
	This completes the proof of \Cref{thm element minus} for $0 \le m \le n-1$. 
\end{proof}
\noindent \textbf{Acknowledgement.} The authors gratefully acknowledge support from the Natural Sciences and Engineering Research Council of Canada through grant RGPIN-2017-04054. We thank Professor Carlos Martins da Fonseca for the critical comments on \Cref{Lemma Dn Inverse}.

	\bibliography{MatrixPaperCitation}{}
	\bibliographystyle{plain}
	
\end{document}